\documentclass[letterpaper, 11pt,  reqno]{amsart}
\usepackage[margin=1.2in,marginparwidth=1.5cm, marginparsep=0.5cm]{geometry}

\usepackage{amsmath,amssymb,amscd,amsthm,amsxtra, esint, xcolor}

\usepackage[implicit=true]{hyperref}

\usepackage[shortlabels]{enumitem}

\allowdisplaybreaks[2]

\usepackage{color}

\definecolor{gr}{rgb}   {0.,   0.69,   0.23 }
\definecolor{bl}{rgb}   {0.,   0.5,   1. }
\definecolor{mg}{rgb}   {0.85,  0.,    0.85}
\definecolor{yl}{rgb}   {0.8,  0.7,   0.}
\definecolor{or}{rgb}  {0.7,0.2,0.2}

\newtheorem{theorem}{Theorem} [section]

\newtheorem{lemma}[theorem]{Lemma}
\newtheorem{proposition}[theorem]{Proposition}
\newtheorem{remark}[theorem]{Remark}

\newtheorem{definition}[theorem]{Definition}

\DeclareMathOperator*{\supp}{supp}

\newcommand{\I}{\hspace{0.5mm}\text{I}\hspace{0.5mm}}

\newcommand{\II}{\text{I \hspace{-2.8mm} I} }
\newcommand{\III}{\text{I \hspace{-2.9mm} I \hspace{-2.9mm} I}}

\newcommand{\IV}{\text{I \hspace{-2.8mm} V} }

\newcommand{\noi}{\noindent}
\newcommand{\Z}{\mathbb{Z}}
\newcommand{\R}{\mathbb{R}}
\newcommand{\C}{\mathbb{C}}
\newcommand{\T}{\mathbb{T}}
\newcommand{\bul}{\bullet}

\newcommand{\hi}{\textup{hi}}
\newcommand{\HI}{\textup{HI}}
\newcommand{\lo}{\textup{lo}}
\newcommand{\LO}{\textup{LO}}

\let\Im=\undefined\DeclareMathOperator*{\Im}{Im}

\let\P= \undefined
\newcommand{\P}{\mathbf{P}}

\newcommand{\Q}{\mathbf{Q}}

\newcommand{\GG}{\mathcal{G}}

\newcommand{\F}{\mathcal{F}}

\newcommand{\al}{\alpha}
\newcommand{\be}{\beta}
\newcommand{\dl}{\delta}

\newcommand{\eps}{\varepsilon}
\newcommand{\kk}{\kappa}
\newcommand{\g}{\gamma}

\newcommand{\ld}{\lambda}

\newcommand{\s}{\sigma}

\newcommand{\ft}{\widehat}
\newcommand{\Ft}{{\mathcal{F}}}
\newcommand{\wt}{\widetilde}
\newcommand{\cj}{\overline}
\newcommand{\dx}{\partial_x}

\newcommand{\dt}{\partial_t}

\newcommand{\embeds}{\hookrightarrow}

\newcommand{\ta}{\theta}

\newcommand{\les}{\lesssim}
\newcommand{\ges}{\gtrsim}

\newcommand{\jb}[1]
{\langle #1 \rangle}

\newcommand{\ind}{\mathbf 1}

\renewcommand{\S}{\mathcal{S}}

\newcommand{\M}{\mathcal{M}}

\newcommand{\N}{\mathbb{N}}
\newcommand{\NN}{\mathcal{N}}

\newtheorem*{ackno}{Acknowledgements}

\renewcommand{\H}{\mathcal{H}}

\def\sgn{\textup{sgn}}

\newcommand{\Pbhi}{\mathbf{P}_{\textup{hi}} }
\newcommand{\Pblo}{\mathbf{P}_{\textup{lo}} }
\newcommand{\Pbhip}{\mathbf{P}_{\textup{+,hi}} }
\newcommand{\PbHIp}{\mathbf{P}_{\textup{+,HI}} }
\newcommand{\PbHI}{\mathbf{P}_{\textup{HI}}}
\newcommand{\PbLO}{\mathbf{P}_{\textup{LO}}}

\newcommand{\Id}{\textup{Id}}

\newcommand{\QQ}{\mathcal{Q}}

\newcommand{\pos}{\text{pos}}
\newcommand{\negg}{\text{neg}}
\newcommand{\Y}{\mathcal{Y}}
\newcommand{\Hf}{\mathfrak{H}}

\newcommand{\Pih}{\Pi_{+,h}}

\newcommand{\lax}{\mathcal{L}}
\newcommand{\peter}{\mathcal{P}}
\newcommand{\op}{\textup{op}}
\newcommand{\TT}{\mathcal{T}}
\newcommand{\tr}{\operatorname{tr}}
\newcommand{\hs}{\mathfrak{I}_2}
\newcommand{\tc}{\mathfrak{I}_1}
\numberwithin{equation}{section}
\numberwithin{theorem}{section}

\makeatletter
\@namedef{subjclassname@2020}{\textup{2020} Mathematics Subject Classification}
\makeatother

\begin{document}
\baselineskip = 14pt

\title[Sub-critical well-posedness for INLS]{Sub-critical well-posedness for the intermediate nonlinear Schr\"{o}dinger equation on the line}

\author[A.~Chapouto, J.~Forlano, T.~Laurens]
{Andreia Chapouto, Justin Forlano, Thierry Laurens}

\address{Andreia Chapouto,
CNRS, and School of Mathematics, Monash University, Clayton, VIC 3800, Australia}

\email{andreia.chapouto@monash.edu}

\address{Justin Forlano,
School of Mathematics, Monash University, Clayton, VIC 3800, Australia}

\email{justin.forlano@monash.edu}

\address{Thierry Laurens,
Department of Mathematics, University of Michigan, MI, 48109, USA}
\email{tlaurens@umich.edu}



\begin{abstract}
We continue our study of the well-posedness theory for the intermediate nonlinear Schr\"{o}dinger equation (INLS).
Firstly, we prove that INLS is locally well-posed in $H^s (\R)$ for any $s>0$. This improves on our previous result of local well-posedness for any $s>\frac 14$, and covers the full scaling-subcritical range for INLS. In particular, we also obtain the local well-posedness for the continuum Calogero-Moser equation without chirality assumption in the full scaling-subcritical range. Our method relies on a gauge transformation, the derivation of a closed system for four auxiliary variables, and nonlinear smoothing estimates, but not on the completely integrable nature of these equations. 

Secondly, for the integrable models, we prove global well-posedness for $0<s<\frac12$ for initial data with small $L^2$-norm.  Moreover, we show that our global well-posedness result applies whenever $L^2$-equicontinuous sets are preserved by the flow, and so the small-data restriction would be removed by an a-priori equicontinuity result.  Our argument relies on a novel family of conserved quantities based upon the Lax pair we discovered in our prior work \cite{CFL1}.
\end{abstract}

\maketitle

\tableofcontents

\section{Introduction}
We consider the Cauchy problem for the intermediate nonlinear Schr\"{o}dinger equation~(INLS):
\begin{equation}
\left\{
\begin{aligned}
  & \dt u+i\dx^2 u = \be u(1+i\mathcal{T}_{h})\dx(|u|^2)+i\g |u|^2 u,\\
   & u|_{t=0}=u_0,
\end{aligned}
 \right. \label{INLS}
\end{equation}
where $0<h<\infty$,  $u:\R\times \R\to \C$, $\be, \g \in \R$, and $\mathcal{T}_{h}$ is the convolution operator 
\begin{align}
\mathcal{T}_{h}f (x) = \frac{1}{2h} \text{p.v.} \int_{-\infty}^{\infty} \coth\bigg( \frac{\pi(x-y)}{2h} \bigg) f(y) dx. \label{tilberty}
\end{align}
The INLS equation arises in the study of quasi-harmonic internal waves in a two fluid system, with bottom fluid of depth $h>0$. The model \eqref{INLS} with $\g\neq 0$ was derived by Pelinovksy-Grimshaw \cite{Pel2} by adding higher order effects to the original derivation of Pelinovsky~\cite{Pel1}, which obtained the $\g=0$ equation. Interestingly, \eqref{INLS} with $\g=0$ is completely integrable, as first observed in \cite{Pel3} through the development of an inverse scattering transform, and later in \cite{CFL1} via a Lax pair.

The equation \eqref{INLS} admits a useful rewriting connecting it to another integrable system, which we now discuss.
 Given a set $A\subseteq \R$, we write $\ind_{A}$ to be the characteristic function of the set $A$. Then, we define $\P_{\pm}$ as the Fourier multiplier operator with symbol $\ind_{\{ \pm \xi>0\}}$, which satisfy
 \begin{align}
\P_{+}+\P_{-} =\text{Id}. 
\notag
% \label{pm1}
\end{align}
Moreover, we have $\H= -i \P_{+}+i\P_{-}$ and thus $1+i\H = 2\P_+$, where $\H$ denotes the Hilbert transform, with Fourier multiplier $\ft{\H}(\xi) = -i \sgn(\xi)$. 
Thus, \eqref{INLS} becomes:
\begin{align}
\dt u+i\dx^2 u = 2\be u \P_{+}\dx(|u|^2)+u \QQ_{h}(|u|^2), \label{INLS2}
\end{align}
where we defined 
\begin{align}
\QQ_{h}&: =  -i\be \GG_{h}+i\g \text{Id}, \label{Qh} \\ 
\mathcal{G}_{h}& : = (\H- \mathcal{T}_{h})\dx.  
%\label{Gh intro}
\end{align}
As $\GG_{h}$ is $L^p\to L^p$ bounded for any $1<p<\infty$ (see Lemma~\ref{LEM:GGh}), it is clear that the same is true for $\QQ_{h}$. 
 See also Remark~\ref{RMK:assumptions} for more general assumptions we can impose on \eqref{Qh} and on the additional nonlinear term in \eqref{INLS} with the factor $\g$.
 
 In the form \eqref{INLS2}, we see the that the fundamental nonlinear term is $\be u \P_{+}\dx(|u|^2)$. Indeed, by formally setting $\QQ_{h}\equiv 0$,\footnote{Alternatively, one may put $\g=0$ and formally take $h\to +\infty$ so that $\QQ_{h}\to 0$. This limit is known as the infinite depth limit and has been studied in \cite{Pel1, CFL1, CFL2}.} \eqref{INLS2} reduces to the continuum Calogero-Moser equation (CCM)
 \begin{align}
\dt u+i\dx^2 u = 2\be u \P_{+}\dx(|u|^2), \label{CCM}
\end{align}
 with $\be>0$ denoting the defocusing CCM and $\be<0$ the focusing CCM. The CCM equation has attracted a lot of attention in recent years as it is completely integrable \cite{GL} and preserves the Hardy space $L^2_{+}(\R)$
 defined as 
\begin{align}
% \label{hardy}
\notag
   L^2_+ (\R) 
    =
    \big\{
    f \in L^2(\R) : \ \supp \ft f \subset [0,\infty)
    \big\} .
\end{align}
 For later use, we will also define the Hardy-Sobolev spaces $H^{s}_+(\R): = H^{s}(\R)\cap L^2_{+}(\R)$ and 
\begin{align*}
H^{s}_{-}(\R) : = \{ f\in H^{s}(\R) \, :\, f=\cj{g} \quad \text{for some} \quad g\in H^{s}_+ (\R)\}.
\end{align*}

Our main goal is to study the well-posedness theory for \eqref{INLS} in $H^{s}(\R)$. This choice of space for the initial data in \eqref{INLS} is motivated by the infinite number of (polynomial) conservation laws for \eqref{INLS} with $\g=0$ and for \eqref{CCM}; see \cite{Rana1, KLV, KKK2, CFL1, KMV}. In particular, \eqref{INLS} preserves the $L^2$-norm and it is natural to study well-posedness here as this turns out to be the scaling critical Sobolev space for \eqref{INLS} in the following sense:
Given $\ld \geq 1$, and $u$ a smooth solution to \eqref{INLS2} on $\R$, we consider the $L^2$-invariant rescaling
\begin{align}
u_{\ld}(t,x) = \ld^{-\frac12} u(\ld^{-2} t,\ld^{-1}x) \label{scale}
\end{align}
which satisfies
 \begin{align}
\dt u_{\ld} +i\dx^2 u_{\ld} = 2\be u_{\ld} \P_{+}\dx( |u_{\ld}|^2) -i\be u_{\ld}\GG_{h\ld}(|u_{\ld}|^2) +i\ld^{-1}\g |u_{\ld}|^2 u_{\ld}, 
\label{scaleINLS}
\end{align}
with rescaled initial data $u_{\ld}(0) = \ld^{-\frac 12} u_0 (\ld^{-1}x)$.
When $\g=0$ and $h=+\infty$, the scaling  \eqref{scale} is an exact symmetry. When $\g=0$ and $0<h<+\infty$, the scaling is no longer an exact symmetry; rather, the family of equations \eqref{INLS2} with depth parameters $0<h<+\infty$ remains invariant under \eqref{scale}. Finally, when $\g\neq 0$, the extra factor $\ld^{-1}$ in the pure cubic term in \eqref{scaleINLS} reflects the sub-critical nature of this term.  

Let us now mention previous well-posedness results. In \cite{demoura1}, de Moura proved local well-posedness of INLS \eqref{INLS} for small initial data in $H^s(\R)$ with $s\geq 1$. Via a gauge transformation, de Moura-Pilod \cite{PMP} extended this to $s>\frac 12$ and removed the small data restriction. See \cite{GL} for the same result for CCM \eqref{CCM} in $H^{s}_{+}(\R)$ for $s>\frac 12$ and \cite{BdMS} for small data in the Besov space $B^{\frac 12}_{2,1}(\R)$. 
Recently, in \cite{CFL1}, we extended these results beyond the $H^{\frac 12}$-barrier and proved local well-posedness for all $s>\frac 14$ and, using the complete integrability, global well-posedness for data with small $L^2$-norm.  
Very recently, \cite{Hadama2} proved global well-posedness for \eqref{INLS} in $L^2(\R)$ for small initial data.
See Remark~\ref{RMK:Hadama} below for a further discussion.

For CCM \eqref{CCM} in the Hardy space, well-posedness in the scaling critical space $L^2_{+}(\R)$ has been established by exploiting the complete integrability through either an explicit formula for solutions~\cite{KLV} or by the method of commuting flows \cite{KMV}. Note that the restriction to the Hardy space for INLS \eqref{INLS} is not  natural as, formally, this assumption is not preserved by the evolution. Moreover, this assumption does not appear in the physical derivation \cite{Pel1,Pel2} of INLS.

We also mention  results on the well-posedness theory for CCM \eqref{CCM} in the Hardy space on the circle $\T$ \cite{Rana1}, and for \eqref{INLS} in $H^{\frac 12}(\T)$ by the authors \cite{CFL2}, as well as well-posedness results for \eqref{INLS} on $\R$ with non-decaying initial data \cite{Chen, ABIK}.
Moreover, we refer to \cite{CFL1} for a further discussion on the long-time behaviour of solutions to CCM \eqref{CCM} on $\R$, which depends heavily on the focusing/defocusing nature. See for example \cite{KKK1, ChenLenz, Chen2}.

Our goal in this paper is to establish the local well-posedness for INLS \eqref{INLS} in the full scaling sub-critical range. This dramatically improves upon our previous result in \cite{CFL1}.

 \begin{theorem}\label{THM:LWP}
Let $\be,\g\in\R$ and $s>0$. Then, \eqref{INLS} is locally well-posed in $H^{s}(\R)$. More precisely, given any $u_0\in H^{s}(\R)$ and $\dl>0$ sufficiently small, there exist $T=T(\|u_0\|_{H^{s}(\R)})>0$ and a unique solution $u\in C([0,T];H^{s}(\R))$ such that:\\
\textup{(i)} $u$ satisfies the Duhamel formulation of \eqref{INLS}:
\begin{align}
u(t) = e^{it \dx^2}u_0 + \int_{0}^{t} e^{i(t-t')\dx^2} \big[ \be u(1+i\mathcal{T}_{h})\dx(|u|^2)+i\g|u|^2 u\big](t') dt' \label{Duhamel}
\end{align}
in $X^{-3,b}_{T}$ with $b:=\frac 12+\dl$, where $X^{s,b}_{T}$ denotes the time-restricted Fourier restriction norm space defined in Section~\ref{SEC:Xsb}. See \eqref{XsbT}. \\
\textup{(ii)} The function $u$ admits the decomposition
\begin{align}
u= e^{-i\be F[v+y+z]}( v+\Y^{\textup{pos}}[v+y+z]+\Y^{\textup{neg}}[v+y+z,w] + z) + w  \label{udecomp-}
\end{align}
where $F[f]:=\dx^{-1}(|f|^2)$ is the primitive defined in \eqref{Fdef},
\begin{align}
\begin{split}
v&:= \P_{+,\textup{hi}}[e^{i\be F[u]}u]\in X^{s,b}_{T}\cap C([0,T]; H^{s}_{+}(\R)), \\
y&:=\P_{-,\textup{hi}}[e^{i\be F[u]}u]\in X^{s-11\dl,b}_{T}\cap C([0,T];H^{s-11\dl}_{-}(\R)), \\
z:=\P_{\textup{lo}}&[e^{i\be F[u]}u]\in X^{\infty,b}_{T}, \quad w:=\P_{-,\textup{hi}}u\in X^{s,b}_{T}\cap C([0,T];H^{s}_{-}(\R)),
\end{split}\label{Xsbproperties}
\end{align}
with $\Y^{\textup{pos}}$ and $\Y^{\textup{neg}}$ two operators defined in \eqref{Ysharp} and \eqref{Yw}, respectively, and $y$ satisfies
\begin{align}
y= \Y^{\textup{pos}}[v+y+z]+\Y^{\textup{neg}}[v+y+z,w]  + e^{i\be F[v+y+z]}w. \label{ymanifold}
\end{align}  
Furthermore, the variables $(v,y,z,w)$ solve the system of equations \eqref{system} with initial data 
\begin{align}
(v,y,z,w)\vert_{t=0} = ( \P_{+,\textup{hi}}[e^{i\be F[u_0]}u_0], \P_{-,\textup{hi}}[e^{i\be F[u_0]}u_0], 
\P_{\textup{lo}}[e^{i\be F[u_0]}u_0],
\P_{-,\textup{hi}}u_0), 
\label{systdata}
\end{align}
with projectors as in \eqref{projs}.
\\
\textup{(iii)} The map $u_0\longmapsto u$ is locally Lipschitz continuous from $H^s(\R)$ into  $C([0,T];H^{s}(\R))$.\\
\textup{(iv)} The variables $v$ and $w$ enjoy nonlinear smoothing: there exists $\al=\al(s)>0$ such that
 \begin{align}
 \begin{split}
\| v- e^{it\dx^2} v_0\|_{C_{T}H^{s+\al}_x}
+ \| w- e^{it\dx^2}w_0\|_{C_{T}H^{s+\al}_x}  \leq C(T, \|u_0\|_{H^{s}} ) .
\end{split} \label{nonlinsmooth0}
\end{align}
 \end{theorem}

Theorem~\ref{THM:LWP} thus completes the local well-posedness for \eqref{INLS} in the full scaling sub-critical range. Moreover, the result allows for large initial data in both the defocusing and focusing cases. In view of part (iii) in Theorem~\ref{THM:LWP}, our solutions are the unique limits of $H^{\infty}(\R)$-solutions and thus, restricting to \eqref{CCM} in the Hardy space $H^{s}_{+}(\R)$, they agree with those constructed using complete integrability \cite{KLV, KMV}. Moreover, it seems difficult using these latter methods to establish that solutions satisfy the Duhamel formula \eqref{Duhamel} and the local Lipschitz dependence of the solution map in (iii). In fact, our method proves that the solution map in (iii) is smooth; see Remark~\ref{RMK:smoothness}. Up to the $L^2$-endpoint, this shows that the results of our method are sharp since whenever $s<0$,  the solution map, if it exists, cannot be $C^3$ at the origin \cite{PdM}.
 
Interestingly, using \eqref{udecomp-} and \eqref{ymanifold}, we actually have the $L^2$-orthogonal decomposition
\begin{align}
e^{i\be F[u]}u = v \, \oplus y \, \oplus z. \label{orthog}
\end{align}
Combining \eqref{orthog} with $L^2$-conservation implies that the $L^2$-norms of $v,y,$ and $z$ are also uniformly bounded in time depending only on $u_0$. It would be of interest to investigate if the decomposition \eqref{orthog} may shed new insights on a priori bounds for \eqref{INLS}.

Our next results address global well-posedness in the completely integrable cases of \eqref{INLS}.  Setting $\g=0$, \eqref{INLS} becomes
\begin{align}
\dt u + i\dx^2 u =  2\be u \Pih\dx(|u|^2), 
\label{INLSG}
\end{align}
where we define 
\begin{align}
\Pi_{\pm,h} := \tfrac{1}{2}(1\pm i\mathcal{T}_{h}) .
\label{Pi+h}
\end{align}
We will also allow $h=\infty$ with the convention that $\TT_{\infty}= \H$, so that the family of equations \eqref{INLSG} includes CCM \eqref{CCM}.
Now, we may also assume that $\be\in \{\pm 1\}$ without loss of generality, where $\be=1$ corresponds to the defocusing equation and $\be=-1$ to the focusing.

To properly describe our results, we first introduce the notion of equicontinuity in $L^2(\R)$.
\begin{definition}[Equicontinuity]\rm
A bounded set $Q\subseteq L^2(\R)$ is said to be \emph{equicontinuous} in $L^2(\R)$ if
\begin{equation}
\sup_{u\in Q}\, \sup_{|y|<\delta} \| u(\cdot+y) - u(\cdot) \|_{L^2} \to 0 \quad\text{as }\delta\to 0.
\label{equicty 1}
\end{equation}
\end{definition}

When the $L^2$-norm in \eqref{equicty 1} is replaced by the supremum norm, we recover the familiar notion of equicontinuity from the Arzel\`a--Ascoli theorem.  Likewise, a subset of $L^2(\R)$ is precompact if and only if it is bounded, tight, and equicontinuous in the sense of \eqref{equicty 1}.  By Plancherel's theorem, this is equivalent to the condition that the family of Fourier transformations is tight:
\begin{equation}
\sup_{u\in Q}\, \int_{|\xi|\geq \kappa} |\ft{u}(\xi)|^2\,d\xi \to 0 \quad\text{as }\kappa\to\infty .
\label{equicty 2}
\end{equation}

Our next theorem is an a-priori result regarding the equicontinuity of solutions to \eqref{INLSG}.
\begin{theorem}[Equicontinuity for small data]\label{t:equicty}
Fix $0<h\leq \infty$.  There exists $\delta>0$ so that if $Q\subseteq H^\infty (\R) \cap\langle x\rangle^{-1}L^2(\R)$ is equicontinuous in $L^2(\R)$ and satisfies
\begin{equation}
\sup_{u\in Q} \|u\|_{L^2} \leq \delta ,
% \label{r}
\notag
\end{equation}
then the set of orbits
\begin{equation}
Q^* := \{ u(t) : u(0)\in Q,\ t\in\R \}
\label{Q*}
% \notag
\end{equation}
reached by the focusing or defocusing equation \eqref{INLSG} is bounded and equicontinuous in $L^2(\R)$.
For the definition of $\jb{x}^{-1}L^2(\R)$, see \eqref{xL2}.
\end{theorem}

As we will see in Theorem~\ref{t:a-priori}, equicontinuity at the scaling-critical regularity is an effective tool in establishing global well-posedness.  In particular, by \eqref{equicty 1}, the equicontinuity of orbits precludes the possibility of solutions concentrating in space, and so rules out the possibility of blow-up.  Concretely, in the case of CCM \eqref{CCM} on $L^2_+(\R)$, the authors of \cite[Lemma~2.4]{KLV2} proved that any smooth solution that is $L^2$-equicontinuous on a time interval $[0,T)$ can be extended, and so finite-time blowup necessarily implies a loss of equicontinuity.  With this in mind, we introduce the following mass threshold for equicontinuity, inspired by~\cite{KLV2,KNV}:

\begin{definition}\rm
Fix $0<h\leq \infty$ and consider the corresponding focusing or defocusing equation \eqref{INLSG}.  Let $M^* \in [0,\infty]$ denote the maximal constant so that for any set $Q\subseteq H^\infty (\R) \cap\langle x\rangle^{-1}L^2 (\R)$ that is bounded and equicontinuous in $L^2(\R)$ and satisfies
\begin{equation}
\sup_{u\in Q} \|u\|_{L^2}^2 < M^* ,
\label{M*}
\end{equation}
the set of partial orbits 
\begin{equation}
Q^*_T := \{ u(t) : u(0)\in Q,\ |t|<T\}
\label{Q*T}
\end{equation}
is equicontinuous in $L^2(\R)$ for each $T>0$.
\end{definition}

We note that the regularity and decay imposed by the condition $Q\subseteq H^\infty (\R)\cap\langle x\rangle^{-1}L^2 (\R)$ is ultimately inconsequential.  The important feature is that this space of initial data is rich enough to witness any counterexamples.  Indeed, the Schwartz class would already suffice, as it is dense in any $H^s(\R)$ space and it includes the initial data for the blowup solutions to the focusing CCM equation constructed in \cite{KKK2}.  We choose this particular space because we are able to prove that it is invariant under the dynamics of \eqref{INLSG}, and that certain key quantities are conserved for such data (see Proposition~\ref{t:A dot} for details).

The results below will establish that the solutions we consider exist globally in time.  We do not require that global existence be known a-priori in the above definition, but as we noted previously, this makes little difference since finite-time blowup will coincide with a loss of equicontinuity.  Nevertheless, it is important that we define $M^*$ in terms of $Q^*_T$ in \eqref{Q*T} rather than $Q^*$ in \eqref{Q*}.  For example, the authors of \cite{GL} demonstrated that two-soliton solutions to the focusing CCM equation blow up in infinite time.  Closer inspection reveals that this is driven by mass migrating to high frequencies in a linear-in-time fashion, and so if we take $Q$ to consist of such a profile, then $Q^*$ fails to be equicontinuous in $L^2(\R)$. By comparison, $Q^*_T$ is equicontinuous for any $T>0$ and so this example would place no restriction on $M^*$.  Strictly speaking, these profiles are excluded in the above definition because they lie outside of $\jb{x}^{-1}L^2(\R)$ due to their slow spatial decay, as do single solitons.  However, the work \cite{HoganKowalski} constructs Schwartz initial data that exhibits similar behavior, and so is captured by our definition.  (This spatial decay is not advertised in the statement of their main results, but is evident within the proof; e.g., see \cite[Lemma~4.2]{HoganKowalski}.)

Theorem~\ref{t:equicty} proves that $M^* > 0$ in both the focusing and defocusing cases.  In the defocusing case we conjecture that $M^* = \infty$, as global well-posedness for large data holds at the higher regularity $H^1(\R)$ \cite{BdMS,CFL2,demoura1}.  On the other hand, for the focusing CCM equation we know that $M^*\leq 2\pi$ by combining the blowup construction in \cite{KKK2} with \cite[Lemma~2.4]{KLV2}.  However, it is unclear yet if $M^* = 2\pi$ for the focusing CCM equation, as the equicontinuity results of \cite{KLV2} are restricted to the Hardy space $L^2_+(\R)$.

With our mass threshold for equicontinuity in hand, we now present our global well-posedness result in full generality.
\begin{theorem}[Global well-posedness]\label{t:a-priori}
Fix $0<s<\frac12$ and $0<h\leq\infty$.  For both the focusing and defocusing equation \eqref{INLSG}, if $Q\subseteq H^\infty (\R)\cap\langle x\rangle^{-1}L^2 (\R)$ is bounded in $H^s(\R)$ and satisfies \eqref{M*}, then the set $Q^*_T$ in \eqref{Q*T} is also bounded in $H^s(\R)$ for any $T>0$.  Consequently, equation \eqref{INLSG} is globally well-posed in the space
\[
B_{M^*}^s := \{ u \in H^s(\R) : \|u\|_{L^2}^2 < M^* \}
\]
endowed with the $H^s$-topology.
\end{theorem}

Combining Theorems~\ref{t:equicty} and \ref{t:a-priori}, given $s>0$, we obtain global well-posedness in $H^s(\R)$ for data that is small in $L^2(\R)$.  However, Theorem~\ref{t:a-priori} is not inherently a small-data result.  Indeed, this formulation ensures that our result applies as soon as the corresponding equicontinuity result is established.

\subsection{Overview of the proofs}

We now discuss the new ideas involved in establishing Theorems~\ref{THM:LWP}, \ref{t:equicty}, and \ref{t:a-priori}, beginning with Theorem~\ref{THM:LWP}.

In our previous work \cite{CFL1}, following \cite{PMP}, which was in turn motivated by \cite{TAO04}, we used the same pair of gauge variables $(v,w)$ defined by
\begin{align*}
v= \P_{+,\text{hi}}[e^{i\be F[u]}u] \quad \text{and}\quad w=\P_{-,\text{hi}}u,
\end{align*}
recalling that $F[f] = \dx^{-1} (|f|^2)$.
From \eqref{INLS}, one can derive a system of equations for $(v,w)$, which is not closed, forcing us to work with the triplet $(u,v,w)$.
Then, the main contribution in the nonlinearity for, say, $v$ involves the term
\begin{align}
\P_{+,\text{hi}}[ v \P_{-}\dx( |u|^2)], \label{vnonlin1}
\end{align}
for which we needed to prove a multilinear estimate in $X^{s,b}$-spaces (see \eqref{Xsb} for a definition). The structure of the frequency projections $\P_{\pm}$ appearing in \eqref{vnonlin1} cancels out the problematic low$\times$low$\times$high$\to$high interactions which are present in $u\P_{+}\dx(|u|^2)$ in \eqref{INLS2}. An inspection of the proof of the main trilinear estimate in \cite[Proposition 4.1]{CFL1} reveals that the worst interactions are the resonant ones where the frequencies of the factors $|u|^2=\cj{u}u$ are the largest and similar.

To estimate \eqref{vnonlin1}, we argued by duality which
allowed us to recover almost $1/2$ of the derivative by using the bilinear Strichartz estimate in \cite{BO98, OT} (Lemma~\ref{LEM:bilin}). Moreover, the sign conditions imposed that the frequency of $v$ also controls the derivative $\dx$. Consequently, controlling the remaining $1/2$ derivative imposed the condition $2s>\frac 12$ forcing $s>\frac 14$.
An improvement here seems clear: we want to also apply the bilinear Strichartz estimate to the factor $|u|^2$ which would (almost) completely overcome the derivative! Unfortunately, this was not possible in \cite{CFL1}, where the gauged variables $(v,w)$ did not satisfy a closed system, thus requiring a bootstrap argument to even control the factors of $u$. In placing these factors into Fourier restriction norm spaces, there was a significant derivative loss, cancelling out the gain from bilinear Strichartz.

 A way to overcome this issue would be to re-inject the gauge variable into the term $|u|^2$, so as to see more factors of $v$ which do belong to $X^{s,\frac 12+}$. Such a step has been effective for studying the low-regularity well-posedness of the Benjamin-Ono equation \cite{BP, IK} and the modified Benjamin-Ono equation \cite{GLM}. Unfortunately, unlike the BO cases, simply injecting the recovery formula \cite[(3.17), p.19]{CFL1} does not seem to work in the full scaling sub-critical range for INLS \eqref{INLS2} as it introduces error terms involving the exponential factors $e^{i\be F[u]}$, which are difficult to control in Fourier restriction norms. 
 
 The key new observation in this work, which is crucial to the proof of Theorem~\ref{THM:LWP}, is the \textit{closed} expression for $u$ in terms of the \textit{quartet} of variables $(v,y,z,w)$ given in \eqref{udecomp-} and with $y$ satisfying \eqref{ymanifold}. In particular, \eqref{udecomp-} allows us to make the trivial observation that 
 \begin{align*}
|u|^2= | v+y+z|^2.
\end{align*}
 Namely, for the main nonlinear terms, such as \eqref{vnonlin1}, we can replace the factor $\dx( |u|^2)$ by $\dx(|v+y+z|^2)$, which not only removes the $u$ dependence, and hence a closed system, but also does not involve the problematic factors $e^{i\be F[v+y+z]}$. The same is true for the equations for $y,z,$ and $w$, which leads us from \eqref{INLS} to a system of four nonlinear Schr\"{o}dinger equations. See Lemma~\ref{LEM:gaugeeqns}. By \textit{positing} that $v,y,z,w\in X^{s,\frac 12+}$, we may fulfill our hope of applying the bilinear Strichartz estimate twice and obtaining the full gain.
 
 Unfortunately, the nonlinearity for the $y$-equation does not have the nice sign structure of \eqref{vnonlin1} and still contains the fatal low$\times$low$\times$high$\to$high interactions, which preclude us from establishing that $y\in X^{s,\frac 12+}$. This is where the further decomposition of $y$ in \eqref{ymanifold} plays a role: as the worst interaction here is due to a single $y$-factor of high frequency with all the other inputs of low frequency; see \eqref{NN2}, we insert the decomposition \eqref{ymanifold} which further develops the equation for $y$. The key point here is that due to favourable frequency projections $\P_{\pm}$, the operators $\mathcal{Y}^{\text{pos}}$ and $\mathcal{Y}^{\text{neg}}$ in \eqref{ymanifold} are \textit{smoother} in space. More precisely, by further using the bilinear Strichartz estimate, these terms gain more spatial regularity when measured in Lebesgue $L^p_{t,x}$-spaces; see Lemma~\ref{LEM:ysharp}. This extra smoothness allows us to close the estimates for the aforementioned bad interactions. 
 
With the Hardy space assumption (for which $w\equiv 0$ in \eqref{Xsbproperties}), this completes the system and we can run a contraction mapping argument to construct $(v,y,z,w)$ locally-in-time. We then reconstruct $u$ via \eqref{udecomp-}, show that it satisfies \eqref{Duhamel}, and prove the local Lipschitz dependence in Theorem~\ref{THM:LWP} (iii); see Section~\ref{SEC:LWP}. 

Outside of the Hardy space, $w$ in \eqref{Xsbproperties} no longer vanishes, and we have a further issue due to the last term $e^{i\be F[u]}w$ in \eqref{ymanifold}: this term enjoys no additional spatial smoothing. 
Our next key observation is that $(v,z,w)$ enjoy a nonlinear smoothing phenomenon as in \eqref{nonlinsmooth0}. For more on the nonlinear smoothing and its applications to dispersive PDEs, see \cite{ET} and the references therein.
Thus, by \textit{reducing} the spatial regularity assumptions on $y$ (see \eqref{Xsbproperties}), we can simultaneously close the estimates for this piece with $e^{i\be F[u]}w$, which has \textit{higher regularity}, and still close the trilinear estimates for the $(v,z,w)$-equations with lower regularity input functions. We then establish that the solution $u$ to \eqref{INLS}, reconstructed via \eqref{udecomp-}, still belongs to $C_{T}H^s(\R)$, and not just $C_{T}H^{s-11\dl}(\R)$ by using \eqref{ymanifold} and the smoothing properties of $\mathcal{Y}^{\text{pos}}$ and $\mathcal{Y}^{\text{neg}}$.
 
 Thus, the setup of this modified closed system of four equations in \eqref{system} is the main new novelty in this work regarding the local well-posedness and reveals the special algebraic structure inherent in \eqref{INLS} and \eqref{CCM}, allowing us to establish well-posedness results which almost encompass those exploiting the complete integrability of \eqref{CCM} in $L^2_{+}(\R)$.

Next, we turn our attention to the proofs of Theorems~\ref{t:equicty} and \ref{t:a-priori}.
Our argument is based on the Lax pair from \cite[Proposition~6.1]{CFL1}: for any $0<h\leq \infty$, $u(t)$ solves~\eqref{INLSG} on the line if and only if the operators
\begin{equation}
\lax_{u;h} = -i\partial_x +\be u\Pih\cj{u} 
\quad\text{and}\quad
\peter_{u;h} = -i\partial_x^2  +2\be u\partial_x\Pih \cj{u} 
% \label{Lax}
\notag
\end{equation}
on $L^2(\R)$ satisfy
\begin{equation}
\tfrac{d}{dt}\lax_{u;h} = [\peter_{u;h},\lax_{u;h}] .
\label{lax}
\end{equation}
\noi 
However, 
after this point, our argument diverges sharply from that in our prior works \cite{CFL1,CFL2}.  Both of these works rely on conserved quantities of the form
\[
\big\langle u , (\lax_{u;h}^2 + \kappa^2)^s u \big \rangle ,
\]
which are well-defined for each $u\in H^{\frac14}$, for all $\kappa>0$ sufficiently large.  The operator $(\lax_{u;h}^2 + \kappa^2)^{s}$ here was introduced as a replacement for $(\lax_{u;\infty} + \kappa)^{2s}$ which originally appeared in \cite[Section~6]{KLV2} in the context of CCM \eqref{CCM} on the Hardy space:  When $u\in L^2_+$, both $\lax_0$ and $\lax_{u;\infty}$ are operators on $L^2_+$ that are semi-bounded from below, and so $\lax_{u;\infty} + \kappa$ is positive definite for all $\kappa$ sufficiently large.  Working with $(\lax_{u;h}^2 + \kappa^2)^{s}$ has the advantage that it continues to be positive definite for $u\notin L^2_+$ and $h<\infty$, and yields a relatively short proof of a-priori estimates in $H^s(\R)$ for $\frac14 \leq s \leq 1$.  However, $\lax_{u;h}^2$ contains a term that is quartic in $u$, and the works \cite{CFL1,CFL2} leverage that $u\in H^{\frac14}\subseteq L^4$ in order to make sense of $\lax_{u;h}^2$ as a quadratic form.  Consequently, it is unclear how to make sense of these quantities outside~$H^{\frac14}$.  

Instead, Theorems~\ref{t:equicty} and \ref{t:a-priori} rely on a new family of conserved quantities $A(\kappa;u,h)$.
For sufficiently regular $u$ and $\kappa>0$ sufficiently large, this quantity corresponds to
\begin{align}
\tr\big\{ &(\lax_{u;h}+i\kappa)^{-1} - (\lax_{0}+i\kappa)^{-1} \big\} 
\nonumber \\
&= -i \sum_{n\geq 1} (i\beta)^n \tr\big\{ \big[ (\kappa-\dx)^{-1} u\Pih\cj{u} \big]^n (\kappa-\dx)^{-1} \big\} .
\label{A series}
\end{align}
Formally, the Lax pair~\eqref{lax} implies that the spectrum of $\lax_{u;h}$ is preserved, and so this trace should be a conserved quantity whenever it is finite.  When $h=\infty$, $u$ is in $L^2_+$, and with $i\kappa$ replaced by $\kappa>0$, this corresponds to the conserved quantity introduced in \cite[Lemma~2.7]{KLV2} for CCM on the Hardy space.  However, as noted above, the Lax operator is no longer semi-bounded from below for our purposes, and so we are forced to take the spectral parameter off of the real axis.  Ultimately, we do not find the analysis in \cite{KLV2} for the special case $h=\infty$ helpful in proving the convergence of the series~\eqref{A series}, as it heavily relies upon the assumption that $u\in L^2_+(\R)$ and the resulting frequency cancellations.  

In Section~\ref{s:consv}, we develop new estimates for the operators appearing in \eqref{A series}, which we then employ in Proposition~\ref{t:A series} to prove the convergence of the series for $n\geq 2$ and arbitrarily large $u\in L^2(\R)$.  This leaves the $n=1$ term, which we separate from the series because it is unclear if the operator in question is trace class for merely $u\in L^2(\R)$. Indeed, the source of our difficulty here is precisely due to the fact that $h<\infty$ and the singular nature of the operator $\mathcal{T}_{h}$ in \eqref{tilberty}.
Nevertheless, we present an alternative formulation of this term in \eqref{A F} that continues to make sense for $u\in L^2(\R)$, and we demonstrate that it agrees with the $n=1$ term in \eqref{A series} on a dense subset of $L^2(\R)$.  

Combining these two pieces, we then show that the resulting quantity is conserved under the flow of \eqref{INLSG}.  Again, we are unable to follow the elegant argument from \cite[Lemma~2.7]{KLV2} in the special case $h=\infty$ and $u\in L^2_+(\R)$; instead, we differentiate the series \eqref{A series} term-by-term and exhibit all of the necessary cancellations, much in the spirit of \cite{KVZ}.

Finally, in Section~\ref{s:a-priori}, we present the proofs of Theorems~\ref{t:equicty} and \ref{t:a-priori}.  For each argument, we craft a linear combination of the mass $\|u\|_{L^2}^2$ and the quantities $A(\kappa;u,h)$ for $\kappa > 0$, which is thus still conserved.
The particular weights are selected so that the quadratic-in-$u$ part resembles \eqref{equicty 2} for our equicontinuity result and $\|u\|_{H^s}^2$ for our a-priori estimates in $H^s(\R)$.  From our analysis in the preceding section, one may readily verify that the full quantity can be bounded above by its quadratic part.  However, to complete the proof, we must show the opposite direction: that the quadratic part is controlled by the entire series.  For the special case of $u\in L^2_+(\R)$ in \cite[Section~3]{KLV2}, this quickly followed from an upper/lower bound on the quadratic form associated to $\lax_{u;\infty}$ on $L^2_+(\R)$ in the focusing/defocusing cases, respectively.  As we were forced to move the spectral parameter in LHS\eqref{A series} off of the real axis however, these inequalities lose their efficacy.  In lieu of this, we carefully bound the contributions made by each frequency of $u$ to the quartic-and-higher order terms of \eqref{A series}.

We close this introduction with a few remarks.

\begin{remark}\rm \label{RMK:assumptions}
Our method for Theorem~\ref{THM:LWP} does not rely on the precise structure of $\QQ_{h}$ in~\eqref{Qh} nor on the pure cubic nonlinear term in \eqref{INLS} with $\g$. Indeed, our result carries over with minor modifications 
to other $L^2$-critical perturbations, such as replacing $\g |u|^2 u$ by $\g |u|^4 u$.
Moreover, if $\QQ_{h}=\QQ_{h}\P_{+}$, then we only require that $\QQ_{h}$ satisfies 
\begin{align*}
\sup_{\xi\in \R} \frac{| \ft \QQ_{h}(\xi)|}{ \jb{\xi}^{1-\ta}} <\infty
\end{align*}
for some $\ta>0$. For example, this includes the filtered nonlinearities $u \, \P_{+}(|u|^{2k})$, $k=1,2$ and the variant with derivative $u\, \P_{+}|\dx|^{1-\ta}(|u|^{2})$. Indeed, for the latter term, it is amenable to a trilinear estimate at the $L^2$-endpoint:
\begin{align}
\| u_{1} \P_{+}|\dx|^{1-\ta} (\cj{u_2}u_3)\|_{X^{s,-\frac 12+2\eps}} \les \prod_{j=1}^{3} \| u_j\|_{X^{s,\frac 12+\eps}} \label{trilinta}
\end{align}
for all $\eps>0$ sufficiently small and any $s\geq 0$. The main point is to apply the bilinear Strichartz estimate (Lemma~\ref{LEM:bilin}) twice when the derivative is large, and use that $\ta>0$ compensates for 
the $O(\eps)$-loss from applying this with the dual function. In particular, this strategy for establishing \eqref{trilinta} provides an alternative argument for the well-posedness results in \cite{Hadama1}. 
\end{remark}
 
 \begin{remark}\label{RMK:Hadama} \rm
During the preparation of this manuscript, the very recent preprint \cite{Hadama2} appeared which 
establishes local well-posedness for \eqref{INLS} in the critical space $L^2(\R)$ for small initial data. We point out that our method and the method in \cite{Hadama2} are completely different: whilst we employ gauge transformations in conjunction with the Fourier restriction norm method and our solutions satisfy the Duhamel formula \eqref{Duhamel}, the method in \cite{Hadama2} is based on constructing a propagator $S_{u}(t,\tau)$ for the Schr\"{o}dinger operator with the rough time-dependent potential $-\dx^2 + \be \dx(|u|^2)$ and constructing solutions to \eqref{INLS} (say for $\g=0$, for simplicity) formally satisfying $u(t)  = S_{u}(t,0)u_0$. Interestingly, both approaches heavily exploit the bilinear Strichartz estimates (Lemma~\ref{LEM:bilin}) and do not rely on complete integrability.

We also remark that our solutions agree with those in \cite{Hadama2}, which is not a-priori clear given that it is not clear if the solutions in \cite{Hadama2} satisfy the Duhamel formulation.
By using the decomposition \eqref{udecomp-} and the $X^{s,b}$-properties in \eqref{Xsbproperties}, we see that the solutions $u$ in Theorem~\ref{THM:LWP} belong to the Strichartz space $L^{4}_{T}L^{\infty}_{x}$. Moreover, by the bilinear Strichartz estimate \eqref{bilin1} we  also have that $|u|^2 \in L^{2}_{T}\dot{H}^{\frac 12}_x$. By applying the uniqueness statement in \cite[Theorem 1.1 (iii)]{Hadama2}, we see that our solutions in Theorem~\ref{THM:LWP} agree with those constructed in~\cite{Hadama2}.
\end{remark}

\section{Preliminaries}\label{SEC:Prelim}

\subsection{Notation}

In this subsection, we introduce relevant notation, projections, and function spaces, which will be used throughout.

We use $A\les B$ to denote $A\leq C B$ for some constant $C>0$,
$A\ll B$ if there is a small $c>0$ such that $A\le cB$, and $A\sim B$ if both $A\les B$ and $B\les A$ hold. 
The notation $a-$ refers to $a-\eps$ for any $\eps>0$. 
Also, $a\land b$ and $a\lor b$ denote the minimum and the maximum between $a$ and $b$, respectively.
Given dyadic numbers $N_0,N_1,N_2,N_3$, we let $N_{(1)}\geq N_{(2)}\geq N_{(3)}\geq N_{(4)}$ denote the decreasing rearrangement of $\{N_0,N_1,N_2,N_3\}$. 

Given a function $f$ on $\R$, we use $\F f$ and $\ft f$ to denote its Fourier transform
\begin{align}
\ft f(\xi) = \frac{1}{\sqrt{2\pi}} \int_\R f(x) e^{-i\xi x} dx. 
    \notag
\end{align}
For space-time functions $u:\R\times\R \to \C$, we may use the notation $\F_t u$ and $\F_x u$ to indicate the Fourier transform with respect to the time and space variables, respectively. We omit this indexing, when clear from context.

Let $s\in \R$ and $1 \le p \le \infty$. We define the $L^p$-based Sobolev spaces $W^{s, p}(\R)$ by the norm:
\begin{align*}
    \| f\|_{W^{s, p}} 
    = \| J^s f \|_{L^p}
    = \big\| \Ft^{-1} \big( \jb{\xi}^s \ft f(\xi) \big) \big\|_{L^p}, 
\end{align*}
where $J^s$ denotes the Bessel potential with Fourier multiplier $\jb{\xi}^s$, where $\jb{x} = (1+|x|^2)^\frac12$ and $\Ft^{-1}$ stands for the inverse Fourier transform. 
We also use $\dot{W}^{s,p}(\R)$ for the homogeneous Sobolev spaces with norm
\begin{align*}
\| f\|_{\dot{W}^{s, p}} 
    = \| D^s f \|_{L^p}
    = \big\| \Ft^{-1} \big( |\xi|^s \ft f(\xi) \big) \big\|_{L^p}
    ,
\end{align*}
where $D^s$ is the Riesz potential, with Fourier multiplier $|\xi|^s$.
When $p=2$, we write $W^{s,2}(\R) = H^s(\R)$ for the $L^2$-based Sobolev spaces, with norm
\begin{align*}
    \| f \|_{H^s} = \| \jb{\xi}^s \ft f (\xi) \|_{L^2_\xi}. 
\end{align*}
On the spatial side, we also define
\begin{equation}
\jb{x}^{-1}L^2 = \{ f\in L^2 : \jb{x}f\in L^2 \} .
\label{xL2}
\end{equation}

Given $1\leq p\leq \infty$ and an operator $R:L^p(\R) \to L^{p}(\R)$, we use $\|R\|_{L^p \to L^p}$ to denote its operator norm on $L^p$.
When $p=2$, we will use the shorthand notation $\|R\|_{\op}$. 
When working with space-time functions, given $T>0$, we often use the shorthand notation $L^p_T W^{s, q}_x$ for $L^p([0,T]; W^{s,q}(\R))$ and $L^p_T L^q_x$ for $L^p([0,T]; L^q(\R))$.

For operators on the Hibert space $L^2(\R)$, we use $\mathfrak{I}_p$ for $1\leq p \leq \infty$ to denote the Schatten class of operators whose singular values are $\ell^p$-summable, and $\| R \|_{\mathfrak{I}_p}$ to denote the corresponding norm.  When $p=\infty$ this corresponds to the operator norm, and $\mathfrak{I}_\infty$ to the closed subspace of compact operators.  When $p=2$ this reproduces the class of Hilbert--Schmidt operators, and for an integral operator $(Kf)(x) = \int k(x,y)f(y)\,dy$ we have $\| K \|_{\hs} = \| k \|_{L^2_{x,y}}$.

The space $\tc$ corresponds to the trace class.  The trace is a bounded linear functional on this space, with 
\[
|\tr(R)| \leq \|R\|_{\tc} .
\]
For an integral operator $K\in\tc$ with a sufficiently regular kernel $k(x,y)$, one recovers the formula $\tr(K) = \int k(x,x)\,dx$.  Moreover, whenever $\frac{1}{p} + \frac{1}{p'} = 1$ we have
\[
\| RS \|_{\tc} \leq \|R\|_{\mathfrak{I}_p} \|S\|_{\mathfrak{I}_{p'}} .
\]
This is simply an application of H\"older's inequality to the singular values of these operators.  In particular, if $R$ is Hilbert--Schmidt then $R^*R$ is trace class.  In fact, we have the following equivalent definition of the $\hs$-norm: 
\[
\| R\|_{\hs}^2 = \tr(R^*R) .
\]

For any $1\leq p\leq\infty$, $\mathfrak{I}_p$ forms a two-sided ideal in the space of bounded operators, due to the inequality
\[
\| RST \|_{\mathfrak{I}_{p}} \leq \|R\|_{\op} \|S\|_{\mathfrak{I}_p} \|T\|_{\op} .
\]
For a thorough introduction to such trace ideals, we recommend the book \cite{Simon}.

We now introduce notation to perform Litlewood-Paley decompositions. 
Let $\eta:\R\to [0,1]$ be a smooth function supported on $[-2,2]$ and equal to $1$ on $[-1,1]$.
Given $N\in 2^{\Z}$, let $\eta_{N}(\xi)=\eta(\frac{\xi}{N})$ and $\psi_N(\xi)=\eta(\frac{\xi}{N})-\eta(\frac{2\xi}{N})$. 
Note that
\begin{align*}
\sum_{N\geq 1} \psi_{N}(\xi) = 1- \eta_{\frac 12}(\xi)
\quad \text{when } \xi \in \R\setminus\{0\}
.
\end{align*}
Moreover, we use $\P_{\le N}$ and $\P_N$ to denote
 the Littlewood-Paley projectors defined by
 \begin{align*}
  \F \,(\P_{\leq N} f)  &= \eta_{N} \ft f 
, 
\\
\F(\P_1 f ) 
&= 
\eta_1  \ft f 
\quad
\text{and}
\quad
\F\, \P_N f 
= \F \,\P_{\leq N} f -\F\, 
\P_{\leq \frac{N}{2}} f 
=\psi_N \ft f, \quad \text{when } N\ge 2 ,
 \end{align*}
 and $\P_{>N}:= 1-\P_{\leq N}$. Note that $\sum_{N\ge 1} \P_N f = f ,$
which we will often use in our estimates, where by abuse of notation, we assume to sum over dyadic numbers in $2^{\Z_{\ge0}}$. 
 We use $\wt\P_{N}$ for the wider projector with multiplier $\wt{\psi}_{N}(\xi) =\psi_{N}(\frac{\xi}{2}) + \psi_{N}(\xi) +\psi_{N}(2\xi)$.
 We then set
 \begin{equation}
\begin{alignedat}{3}
 \F\,( \P_{+} f)(\xi) &= \ind_{\xi> 0} \ft f(\xi), 
 &
 \qquad   \F\, ( \P_{-} f)(\xi)  &= \ind_{\xi< 0} \ft f(\xi),
 \\
\Pbhi & = \sum_{N\geq 2}\P_{N},  
&
\quad 
\mathbf{P}_{\text{HI}} & = \sum_{N\geq 2^{3}}\P_N, 
\\
\Pblo &= \text{Id}-\Pbhi, 
& \quad 
\mathbf{P}_{\text{LO}} &=\text{Id} - \mathbf{P}_{\text{HI}}.
\end{alignedat}
\label{projs}
% \notag
\end{equation}
We also define the shorthand $\P_{\pm, \text{hi}}=\P_{\pm}\Pbhi$ and similarly for $\P_\HI, \P_\lo, \P_\LO$. 

For space-time functions $u: \R\times \R \to \C$, we define frequency projectors on the space-time Fourier variables $(\tau, \xi)$: given $K\in 2^{\N}$, we set 
\begin{align}
\label{Qpro}
\begin{split}
\mathcal{F}_{t,x}\{ \Q_{\ll K}u\}(\tau,\xi) &= \eta_{10^{-10}K}(\tau -\xi^2)\ft u(\tau,\xi),\\
 \mathcal{F}_{t,x}\{ \Q_{\ges K}u\}(\tau,\xi) &=(1- \eta_{10^{-10}K})(\tau -\xi^2)\ft u(\tau,\xi).
\end{split}
\end{align}

\subsection{Product estimates}

In this section, we discuss some useful product-type estimates.

\begin{lemma}[Fractional Leibniz rule]
\label{LEM:leib}
Let $s>0$ and $1<p_j,q_j,r\leq \infty$ 
such that $\frac1r=\frac{1}{p_j}+\frac{1}{q_j}$, $j=1,2$,. 
Then, we have
\begin{align*}
  \| J^s(fg)\|_{L^r (\R)} \les \|J^s f\|_{L^{p_1}(\R)} \|g\|_{L^{q_1}(\R)}  + \|f\|_{L^{p_2}(\R)} \| J^s  g\|_{L^{q_2}(\R)}.
\end{align*}
In particular, the endpoint $r=p_j=q_j=\infty$, $j=1,2$ is allowed.
\end{lemma}
\noi
For a proof of Lemma~\ref{LEM:leib}, see \cite{KP, CW,GO,BL}.
\begin{remark}\rm

It is clear that $e^{i\be F}\in L^{\infty}(\R)$ but that $e^{i\be F} \notin L^2(\R)$. Consequently, $e^{i\be F}$ is merely a tempered distribution. Nonetheless, as $u\in L^2(\R)$, $\dx e^{i\be F}\in L^1(\R)$, and we have that for almost every $\xi\in\R$, 
\begin{align*}
\mathcal{F}\{e^{i\be F}\}(\xi) = \frac{1}{i\xi} \int_{\R} e^{-ix\xi} \dx(e^{i\be F}) dx.
\end{align*}
Whilst $\Pbhi e^{i\be F}$, $\PbHI e^{i\be F}$, and $\Pbhip e^{i\be F}$ are well-defined and belong to $L^2(\R)$, due to the non-integrable singularity at the origin, the expressions $\P_{\pm}(e^{i\be F})$ are ill-defined. Moreover, we need to carefully define $\Pblo e^{i\be F}$ and $\PbLO e^{i\be F}$. Here, we define these as:
\begin{align*}
\Pblo e^{i\be F} : = e^{i\be F} -\Pbhi e^{i\be F} \quad \text{and} \quad \PbLO e^{i\be F} : = e^{i\be F} -\PbHI e^{i\be F}.
\end{align*}
It follows from these definitions that $\PbHI \Pblo(e^{i\be F})=0$ and $\dx \Pblo(e^{i\be F}) = \Pblo \dx e^{i\be F}$.
For $0<\al\leq 1$, it holds that
\begin{align}
\| D^{\al}_{x} \Pblo e^{i\be F}\|_{L^{\infty}_{x}} & \les 1, \label{DeLinfty}\\
\| D^{\al}_{x} \Pblo [e^{i F_1}-e^{iF_2}]\|_{L^{\infty}_{x}}& \les \|e^{iF_1}-e^{iF_2}\|_{L^{\infty}_x}. \label{DeLinfty2}
\end{align}
See \cite[Section 2.2]{CFL1}.
\end{remark}

\begin{lemma}\label{LEM:JsL2}
Let $0\leq s < 1$. Assume that $F_j$, $j=1,2$ are two real-valued functions such that $\dx F_j= |f_j|^2$.
Then, 
\begin{align}
\| J^{s} ( e^{iF_1}g)\|_{L^{2}_{x}} &\les   (1 + \|f_1\|_{H^{\max(s-\frac12+,0)}} \big)^4 \|g\|_{H^s},
  \label{L2F1} \\
\| J^{s} ( (e^{iF_1}-e^{iF_2})g)\|_{L^{2}_{x}} & \les  (1 + \|f_1\|_{H^{\max(s-\frac12+,0)}} + \|f_2\|_{H^{\max(s-\frac12+,0)}} \big)^5 \|f_1-f_2\|_{L^2}\|g\|_{H^{s}}.
 \label{L2F2}
\end{align}
\end{lemma}
\begin{proof}
The estimate \eqref{L2F1} is obvious when $s=0$ or $e^{iF}g$ is replaced by $\Pblo(e^{iF_1}g)$. Thus, we only estimate the contribution from $\Pbhi(e^{iF_1}g)$ when $s>0$.
Writing 
\begin{align*}
\| J^{s}\Pbhi ( e^{iF_1}g)\|_{L^{2}_{x}} & \leq \| D^{s} \Pbhi ( e^{iF_1}g)\|_{L^{2}_{x}}  \leq \| D^{s} ( \Pbhi(e^{iF_1})g)\|_{L^{2}_{x}} + \| D^{s} ( \Pblo(e^{iF_1})g)\|_{L^{2}_{x}},
\end{align*}
we have by the fractional Leibniz rule and \eqref{DeLinfty} that 
\begin{align*}
\| D^{s} ( \Pblo(e^{iF})g)\|_{L^{2}_{x}} \les \|D^{s}g\|_{L^{2}_x}, 
\end{align*}
when $0< s < 1$. 
For the $\Pbhi e^{iF_1}$ contribution we consider two cases: (i) $0 < s < \frac12$ and (ii) $\frac12 \le s \le 1$.
If (i) holds, 
 by the fractional Leibniz rule and Sobolev inequality, we have
\begin{align*}
\| D^{s} ( \Pbhi(e^{iF_1})g)\|_{L^{2}} 
& \les  \| D^{s-1}\Pbhi( |f_1|^2 e^{i\be F_1})\|_{L^{\frac{1}{s}}} \|g\|_{L^{\frac{2}{1-2s}}} + \|\Pbhi e^{i\be F_1}\|_{L^{\infty}} \|D^{s}g\|_{L^2} \\
& \les \| \Pbhi( |f_1|^2 e^{i\be F_1})\|_{L^{1}} \|g\|_{H^{s}} + \|g\|_{H^s} \\
&\les (1+\| f_1\|_{L^{2} }^{2}) \|g\|_{H^s}.
\end{align*}
If (ii) holds, we instead have
\begin{align}
\| D^{s} ( \Pbhi(e^{iF_1})g)\|_{L^{2}}^2
& \les  \| D^{s-1 +}\Pbhi( |f_1|^2 e^{i\be F_1})\|_{L^{2}} \|g\|_{L^{\infty-}} + \|\Pbhi e^{i\be F_1}\|_{L^{\infty}} \|D^{s}g\|_{L^2} 
\notag
\\
& \les \| D^{s-\frac12 + }\Pbhi( |f_1|^2 e^{i\be F_1})\|_{L^{1}} \|g\|_{H^{s}} + \|g\|_{H^s} , 
\label{aux000}
\end{align}
where for the first contribution on the RHS, by fractional Leibniz, \eqref{DeLinfty}, Sobolev inequality, and H\"older's inequality,  we get
\begin{align*}
    &
    \| D^{s-\frac12 + }\Pbhi( |f_1|^2 e^{i\be F_1})\|_{L^{1}}
    \\
    &
    \les
    \|  f_1 \|_{L^2}^2 \| D^{s-\frac12 + }\P_\lo e^{i\be F_1}\|_{L^{\infty}} 
    +
    \|  f_1 \|_{L^{2}}^2 \| D^{s-\frac12 + }\P_\hi e^{i\be F_1}\|_{L^{\infty}} 
    + 
    \| f_1\|_{L^2} \|f_1\|_{H^{s-\frac12+}} 
    \\
    &
    \les \|f_1\|_{L^2} ( 1+  \|f_1\|_{H^{s-\frac12+}} )
    + 
    \|  f_1 \|_{L^{2}}^2 
    \| \dx \P_\hi ( e^{i\be F_1}  ) \|_{L^{\frac{2}{3-2s}+}}
    \\
    &
    \les \|f_1\|_{L^2} ( 1+  \|f_1\|_{H^{s-\frac12+}} ) + \| f_1\|^3_{L^2} \| f_1\|_{L^{\frac{1}{1-s}+}}
    \\
    &
    \les ( 1 + \|f_1\|_{L^2} )^3 \|f_1\|_{H^{s-\frac12+}} 
\end{align*}
which together with \eqref{aux000} completes the proof in Case (ii). 

The difference estimate \eqref{L2F2} follows similarly by using \eqref{DeLinfty2} and \eqref{Fdef} which imply
\begin{align}
\|F_1-F_2\|_{L^{\infty}} \les (\|f_1\|_{L^2}+\|f_2\|_{L^2})\|f_1-f_2\|_{L^2}.
\label{Fdiff}
\end{align}
We omit details. 
\end{proof}

We recall the following from \cite[Lemma 2.4]{CFL1}.

\begin{lemma}\label{LEM:GGh}
Let $0<h<\infty$ and $1<p<\infty$.  Then, $\mathcal{G}_{h}$ is $L^p(\R)\to L^p(\R)$ bounded and 
\begin{align}
    \| \GG_{h}\|_{L^{p}(\R)\to L^p(\R)} \les \tfrac{1}{h}
    ,
    % \label{GGconv}
    \notag
\end{align}
where the implicit constant is uniform in $1\leq h<\infty$.
\end{lemma}

\section{The gauge transform and properties of solutions}
\subsection{Fourier restriction norm spaces}\label{SEC:Xsb}

For $s,b\in \R$, we consider the Fourier restriction norm spaces $X^{s,b}(\R\times \R)$ as the completion of $\S(\R\times \R)$ under the norm:
\noi
\begin{align}
\| u\|_{X^{s, b}(\R\times \R)} &= \big\| \jb{\tau-\xi^2}^{b}\jb{\xi}^s \ft u(\tau, \xi)\big\|_{L^2_{\tau,\xi}}, 
\label{Xsb}
% \notag
\end{align}
Given a time interval $I \subset \R$, we define localised in time versions of these spaces as follows: if $u: I \times \R \to \C$, then 
\begin{align}
\|u\|_{X^{s,b}_{I}}: =\inf\{ \|\wt{u}\|_{X^{s,b}} \, : \, \wt{u}:\R\times \M \to \C, \,\, \wt{u}\vert_{I\times \M} = u\}. 
\label{XsbT}
\end{align}
When $I=[0,T]$ for some $T>0$, we denote the spaces $X^{s,b}_{I}=X^{s,b}_{T}$.
For any $b>\frac 12$, the following embedding holds:
\begin{align}
X^{s,b}_{T} \embeds C([0,T];H^{s}(\R)). \label{YsCTHs}
\end{align}

We recall the following linear estimates related to the Fourier restriction norm spaces. See~\cite{MP}.

\begin{lemma}\label{LEM:linXsb}
Let $s\in \R$ and $0<\dl<\frac 12$. Then, the following estimates hold\\
\textup{(i)} 
\begin{align}
\|  S(t)f\|_{X_{T}^{s,b}} \les \|f\|_{H^{s}} \quad \textup{and}\quad \bigg\|  \int_{0}^{t} S(t-t') g(t')dt' \bigg\|_{X_{T}^{s,\frac 12+\dl}}  \les \|g\|_{X^{s, -\frac 12 +\dl}_{T}}
\notag
\end{align}
\noi
\textup{(ii)} For any $T>0$ and $-\frac 12 < b'\leq b<\frac 12$, it holds that 
\begin{align}
\| u\|_{X^{s,b'}_{T}} \les T^{b-b'}\|u\|_{X^{s,b}_{T}}.  \label{timeloc}
\end{align}
\textup{(iii)} For any $b>\frac 12$ and $2< p\leq 6$, it holds that 
\begin{align}
\|u\|_{L^{p}_{T,x}}& \les \|u\|_{\wt{L^{p}_{T,x}}} \les   \|u\|_{X_{T}^{0,3b (\frac{1}{2}-\frac{1}{p})}}.
\label{L4}
\end{align}
Moreover, for any $N\in \N$ and for any $6<p\leq \infty$, we have
 \begin{align}
\| \P_{N} u\|_{L^{p}_{T,x}} \les N^{\frac{1}{2}-\frac{3}{p}} \|\P_{N}u\|_{X^{0,\frac 12+}_{T}} .
\label{L6hi}
\end{align}
\end{lemma}

When $p=6$, \eqref{L4}  follows from the $L^6_{t,x}$ Strichartz estimate on $\R$ and transference. Then, we get the range $2\leq p\leq 6$ by interpolation with the trivial estimate for $p=2$.

We will also use the local smoothing and maximal function properties of the group $\{e^{it\dx^2}\}$ adapted to the Fourier restriction norm spaces. See \cite[Proposition 2.5]{MR}.

\begin{lemma}\label{LEM:locsmoothing}
For any $b>\frac 12$ and $N\in 2^{\N}$, it holds that
\begin{align}
\| \P_{N}u\|_{L^{\infty}_{x}L^{2}_{t}} \les N^{-\frac 12}\| \P_{N}u\|_{X^{0,b}}.
 \label{localsmooth0}
\end{align}
and for any $2\leq p\leq 4$, 
\begin{align}
\| \P_{N}u\|_{L^{p}_{x}L^{\infty}_{t}} \les N^{\frac 1p+}\| \P_{N}u\|_{X^{0,b}}.
 \label{maximal}
\end{align}
\end{lemma}

\begin{lemma}[Bilinear Strichartz estimate \cite{BO98, OT}] \label{LEM:bilin}
For $N,N_1,N_2\in 2^{\N}$ such that $N_1 \gg N_2$ and $0\leq \dl \ll 1$ sufficiently small, it holds that
\begin{align}
\| \P_{N}[ u \cdot \cj{v}]\|_{L^{2}_{t,x}}& \les N^{-\frac 12+10\dl} \|u\|_{X^{0,\frac 12-2\dl}}\|v\|_{X^{0,\frac 12-2\dl}}, \label{bilin1} \\
 \| \P_{N_1} u \cdot \P_{N_2}v\|_{L^{2}_{t,x}} &\les N_1^{-\frac 12+10\dl} \|\P_{N_1}u\|_{X^{0,\frac 12-2\dl}}\|\P_{N_2}v\|_{X^{0,\frac 12-2\dl}}. \label{bilin3}
\end{align}
\end{lemma}

Local-in-time versions of \eqref{bilin1} and \eqref{bilin3} also hold.
We also make use of a trilinear Strichartz estimate.

\begin{lemma}\label{LEM:trilin}
Let $f_1,f_2,f_3\in L^2(\R)$ and $N_1,N_2,N_3, N\in 2^{\N}$ be such that $N_1 \gg N\vee N_3$. 
Let $m:\R^{3}\to \mathbb{C}$ be measurable and bounded and define the trilinear operator $\mathcal{A}_{m}$ via
\begin{align*}
\mathcal{F}_{x}\{\mathcal{A}_{m}(u_1,u_2,u_3) \}(\xi) : =\int_{ \xi=-\xi_1-\xi_2+\xi_3 } m(\xi_1,\xi_2,\xi_3)  \cj{\mathcal{F}_{x}\{u_1\}(\xi_1)} \cj{\mathcal{F}_{x}\{u_2\}(\xi_2)} \mathcal{F}_{x}\{u_3\}(\xi_3)  d\xi_1d\xi_3
\end{align*}
Then,
\begin{align}
\Big\| \P_{N}\mathcal{A}_{m}(\P_{N_1}u_1,\P_{N_2}u_2, \P_{N_3}u_3) \Big\|_{L^{2}_{t,x}} 
\les&   N_{3}^{15\dl }\bigg( \frac{N_3}{N_1}\bigg)^{\frac 12-10\dl} \|m\|_{L^{\infty}(\R^3)} \label{trilin} \\
& \times \| \P_{N_1}u_1\|_{X^{0,\frac 12-2\dl}}  \| \P_{N_2}u_2\|_{X^{0,\frac 12-2\dl}}   \| \P_{N_3}u_3\|_{X^{0,\frac 12+\dl}} \notag
\end{align} 
for any $\dl>0$ sufficiently small, with implicit constant uniform in $N_1,N_2,N_3,N$.
\end{lemma}

We remark that the additional loss in $N_{3}^{15\dl}$ is due to the stronger assumption on the right-hand side of $X^{s,b'}$ with $b'<\frac 12$. Whilst this loss could be removed, say by working in the $U^2, V^{2}$ functional framework, the loss in the exponent for the ratio $\frac{N_3}{N_1}$ seems necessary in order to reduce the modulations below $1/2$.

\begin{proof}
We first prove a version of \eqref{trilin} for linear Schr\"{o}dinger waves:
\begin{align}
\Big\| \P_{N}\mathcal{A}_{m}( S(t)\P_{N_1}f_1, S(t)\P_{N_2}f_2, S(t)\P_{N_3}f_3) \Big\|_{L^{2}_{t,x}} \les \|m\|_{L^{\infty}(\R^3)} \bigg( \frac{N_3}{N_1}\bigg)^{\frac 12} \prod_{j=1}^{3}\|\P_{N_j}f_j\|_{L^2_x}. \label{trilin0}
\end{align}
Let $\eta$ be a real-valued smooth cut-off function such that $\eta(t)=1$ on $[-1,1]$ and $\supp \eta \subset [-2,2]$, and for $T\geq 1$ set $\eta_{T}(t):= \eta(t/T)$. We show that 
\begin{align}
\Big\| \P_{N}\mathcal{A}_{m}( \eta_T S(t)\P_{N_1}f_1, \eta_T S(t)\P_{N_2}f_2,\eta_T S(t)\P_{N_3}f_3) \Big\|_{L^{2}_{t,x}}  \les \|m\|_{L^{\infty}(\R^3)}  
\bigg( \frac{N_3}{N_1}\bigg)^{\frac 12} \prod_{j=1}^{3}\|\P_{N_j}f_j\|_{L^2_x} 
,
\label{trilin00}
\end{align}
with implicit constant uniform in $T\geq 1$. Then, \eqref{trilin0} follows by the monotone convergence theorem taking $T\to\infty$. Taking the space-time Fourier transform, we have
\begin{align*}
 &\mathcal{F}_{t,x} \bigg\{\P_{N}\mathcal{A}_{m}( \eta_T  S(t)\P_{N_1}f_1, \eta_T  S(t)\P_{N_2}f_2, \eta_T S(t)\P_{N_3}f_3) \bigg\}(\tau,\xi) \\
   & = \int_{\R^2} \psi_{N}(\xi) {\psi}_{N_1}(\xi_1)  {\psi}_{N_2}(-\xi - \xi_1 + \xi_3) {\psi}_{N_3}(\xi_3) m(\xi_1,-\xi - \xi_1 + \xi_3,\xi_3) \\
 & \hphantom{XXXXXXXX} \times \cj{\ft f_1(\xi_1)} \cj{\ft f_2(-\xi - \xi_1 + \xi_3)} \ft f_3(\xi_3)  \mathcal{F}_{t}\{ \eta_{T}^{3}\}(\tau - \xi_1^2 - (\xi+\xi_1-\xi_3)^2 + \xi_3^2)d\xi_1 d\xi_3.
\end{align*} 
By a Cauchy-Schwarz argument, we find
\begin{align*}
\text{LHS} \eqref{trilin00} & \les \|m\|_{L^{\infty}(\R^3)} M_{N_1,N_2,N_3,N}^{\frac 12} \prod_{j=1}^{3} \|\P_{N_j}f_j\|_{L^2_x},
\end{align*}
where 
\begin{align*}
M_{N_1,N_2,N_3,N}  : = \sup_{\tau\in \R, |\xi|\sim N} \int_{\R} |{\psi}_{N_3}(\xi_3)|^2 \int_{\R} & |{\psi}_{N_1}(\xi_1)  {\psi}_{N_2}(-\xi - \xi_1 + \xi_3)|^2 \\
&\times  |\mathcal{F}_{t}\{ \eta_{T}^{3}\}(\tau -\xi_1^2 -(\xi+\xi_1-\xi_3)^2 +\xi_3^2)| d\xi_1 d\xi_3.
\end{align*}
For fixed $(\tau, \xi,\xi_3)$, we define $\Phi(\xi_1) : = \tau -\xi_1^2 -(\xi+\xi_1-\xi_3)^2 +\xi_3^2$ and note that 
\begin{align*}
|\Phi'(\xi_1)| = 2 | 2\xi_1 + (\xi-\xi_3)| \ges N_1
\end{align*}
where we used that $N_1 \gg N\vee N_3$. Therefore, with the change of variables $\xi_1\mapsto y=\Phi(\xi_1)$, we have
\begin{align*}
\int_{\R} |\wt{\phi}_{N_1}(\xi_1)  &  \wt{\phi}_{N_2}( -\xi-\xi_1+\xi_3)|^2
  |\mathcal{F}_{t}\{ \eta_{T}^{3}\}(\tau -\xi_1^2 -(\xi+\xi_1-\xi_3)^2 +\xi_3^2)|^2 d\xi_1 \\
  & \les \frac{1}{N_1} \int_{\R} |\mathcal{F}_{t}\{ \eta_{T}^{3}\}(y)|^2 dy \les \frac{1}{N_1}
\end{align*}
uniformly in $T\geq 1$, and $(\tau,\xi,\xi_3)$. Inserting this back into $M_{N_1,N_2,N_3,N}$ and integrating over $|\xi_3|\sim N_3$ shows that $M_{N_1,N_2,N_3,N} \les N_3 N_{1}^{-1}$ which proves \eqref{trilin00}. 

To obtain \eqref{trilin}, we first note that by transference, \eqref{trilin0} implies 
\begin{align}
\bigg\| \P_{N}\mathcal{A}_{m}(\P_{N_1}u_1,\P_{N_2}u_2, \P_{N_3}u_3) \bigg\|_{L^{2}_{t,x}} \les  \|m\|_{L^{\infty}(\R^3)} \bigg( \frac{N_3}{N_1}\bigg)^{\frac 12}  \prod_{j=1}^{3} \| \P_{N_j}u_j\|_{X^{s,\frac 12+\dl}}.
\label{trilin02}
\end{align} 

We want to reduce the modulation exponent to below $\frac 12$ and we will do this by interpolating \eqref{trilin02} with another inequality. 
By Plancherel's theorem, H\"{o}lder and Bernstein inequality, \eqref{L6hi}, and \eqref{YsCTHs}, we have 
\begin{align}
\begin{split}
&\bigg\| \P_{N}\mathcal{A}_{m}(\P_{N_1}u_1,\P_{N_2}u_2, \P_{N_3}u_3) \bigg\|_{L^{2}_{t,x}}  \\
& \les \|m\|_{L^{\infty}(\R^3)} \| \mathcal{F}^{-1}_{x}\{ | \ft{\P_{N_1}u_1}|\} \mathcal{F}^{-1}_{x}\{ | \ft{\P_{N_2}u_2}|\}  \mathcal{F}^{-1}_{x}\{ | \ft{\P_{N_3}u_3}|\}  \|_{L^{2}_{t,x}} \\
& \les \|m\|_{L^{\infty}(\R^3)} \| \mathcal{F}^{-1}_{x}\{ | \ft{\P_{N_1}u_1}|\}\|_{L^{4}_{t,x}} \| \mathcal{F}^{-1}_{x}\{ | \ft{\P_{N_2}u_2}|\}\|_{L^{4}_{t,x}}  \| \mathcal{F}^{-1}_{x}\{ | \ft{\P_{N_3}u_3}|\}\|_{L^{\infty}_{t,x}}  \\
&\les  \|m\|_{L^{\infty}(\R^3)} N_{3}^{\frac 12} \|\P_{N_1}u_1\|_{X^{0,\frac 38+\frac{\dl}{2}}}\|\P_{N_2}u_2\|_{X^{0,\frac 38+\frac{\dl}{2}}} \|\P_{N_3}u_3\|_{X^{0,\frac 12+\dl}}. 
\label{trilin03}
\end{split}
\end{align}
Now \eqref{trilin} follows from interpolating\footnote{We want $\frac 12-2\dl = \ta (\frac 38+\frac{\dl}{2}) + (1-\ta)(\frac 12+\dl)$ which gives $\ta=\frac{24\dl}{1+4\dl}$. So we lose $N_{3}^{\frac{12\dl}{1+4\dl}}\leq N_{3}^{15\dl}$ and $(N_3/N_1)^{\frac 12- \frac{12\dl}{1+4\dl}}\leq (N_3/N_1)^{\frac 12- 10\dl}  $  by making $\dl>0$ sufficiently small. }
 \eqref{trilin02} and \eqref{trilin03}.
\end{proof}

\subsection{Gauge transformation and the new decomposition} \label{SEC:Gauge}

We construct a gauge transformation as in \cite{PMP}. 
We define
\begin{align}
F=F[u] = \dx^{-1}(|u|^2) = \int_{-\infty}^{\bul} |u(y,t)|^2 dy \label{Fdef}
\end{align}
and note that 
\begin{align*}
\dx F = |u|^2.
\end{align*}
We define the gauged variable
\begin{align}
v := \P_{+,\text{hi}}[ e^{i\be F[u]} u] .
\label{gauge}
\end{align}

We now detail a decomposition which is one of the key ideas in improving upon our previous results.
This decomposition is motivated by a similar decomposition in \cite{GLM} for the modified Benjamin-Ono equation. The decomposition is distinct from the usual recovery formula. %in \eqref{recovery}.
We write 
\begin{align}
u= e^{-i\be F} e^{i\be F} u = e^{-i\be F}[ v+ \P_{-,\text{hi}}(e^{i\be F} u) + \Pblo(e^{i\be F}u)] \label{ue}
\end{align}
 We define 
\begin{align}
y:= \P_{-,\text{hi}}[e^{i\be F}u] \quad \text{and} \quad z: = \Pblo[ e^{i\be F}  u] \label{YZ}
\end{align}
so that \eqref{ue} becomes
\begin{align}
u=  e^{-i\be F[u]}(v+y+z)= e^{-i\be F[v+y+z]}(v+y+z) \label{ue2}
\end{align}
If $u$ does not belong to the Hardy space $L^2_{+}(\R)$, then we need another auxiliary variable 
\begin{align}
w: = \P_{-,\text{hi}}u. \label{w}
\end{align}
If $u$ does belong to the Hardy space, then $w\equiv 0$.
The decomposition \eqref{ue2} combined with the observation that 
\begin{align}
|u|^2 = |v+y+z|^2 \label{u^2}
\end{align}
motivates us to study the quartet $(v,y,z, w)$, rather than the original variable $u$. In particular, \eqref{ue2} is a closed expression relating $(v,y,z,w)$ to $u$.

\begin{lemma}\label{LEM:gaugeeqns}
Let $u \in C([0,T];H^{\infty}(\R))$ be a smooth solution of \eqref{INLS2}.
Then, the variable $v$ defined in \eqref{gauge}, the variables $y,z$ defined in \eqref{YZ} and $w$ defined in \eqref{w} satisfy the following equations:
\begin{align}
&\dt v + i\dx^2 v = \mathfrak{N}_{v}[v,v+y+z]:=-2\be \Pbhip\big[ v \P_{-}\dx(  |v+y+z|^2)\big]  \notag \\
&\hphantom{XXXXXXXXXXXXXXXX} + \P_{+, \hi}[ (v+y+z) \QQ_{h}(|v+y+z|^2)],\label{veq} \\
&\dt y + i\dx^2 y  = \mathfrak{N}^{(0)}_{y}[v+y+z]=-2\be \P_{-,\textup{hi}}\big[ (v+y+z) \P_{-}\dx(  |v+y+z|^2)\big]  \notag\\
& \hphantom{XXXXXXXXX}+ \P_{-,\textup{hi}}[ (v+y+z) \QQ_{h}(|v+y+z|^2)] ,\label{Yeq} \\
&\dt z + i\dx^2 z = \mathfrak{N}_{z}[v+y+z] :=-2\be \P_{\textup{lo}}\big[ (v+y+z) \P_{-}\dx(  |v+y+z|^2)\big] \notag \\
& \hphantom{XXXXXXXXXXXXXXXX}+ \P_{\textup{lo}}[ (v+y+z) \QQ_{h}(|v+y+z|^2)] \label{Zeq},\\
&\dt w + i\dx^2 w =\mathfrak{N}_{w}[w, v+y+z]:=-2\be \P_{-,\textup{hi}}\big[ w \P_{+}\dx(  |v+y+z|^2)\big]  \notag \\
&\hphantom{XXXXXXXXXXXXXXXX}+ \P_{-, \hi}[ e^{-i\be F[v+y+z]} (v+y+z) \QQ_{h}(|v+y+z|^2)]\label{weq}.
\end{align}
\end{lemma}

\begin{proof}
Following \cite{MP, CFL1}, we compute
\begin{align*}
\dt F+ i\dx^2 F - \be (\dx F)^2 = 2i u \dx \cj{u}
\end{align*} 
and 
\begin{align*}
(\dt +i\dx^2)(e^{i\be F}u)& = e^{i\be F}  \big\{ i \be  (\dt F +i\dx^2 F - \be (\dx F)^2)u +(\dt u +i\dx^2 u) -2\be (\dx F)\dx u \big\}.
\end{align*}
Inserting \eqref{INLS2} and using that $u \dx(|u|^2) = u^2 \dx \cj{u}+|u|^2 \dx u$, we find
\begin{align*}
(\dt +i\dx^2)(e^{i\be F}u)& = -2\be e^{i\be F} u \P_{-}\dx (|u|^2) +e^{i\be F} u \QQ_{h}(|u|^2).
\end{align*}
Replacing the factors $e^{i\be F} u$ by \eqref{ue2} and the factors $|u|^2$ by \eqref{u^2} and then taking either $\P_{-,\text{hi}}$ or $\Pblo$ of both sides yields \eqref{Yeq} and \eqref{Zeq}. 
Taking $\Pbhip$ of both sides gives
\begin{align*}
(\dt +i\dx^2)v & = -2\be \Pbhip[ (v+y+z) \P_{-}\dx (|v+y+z|^2)] + \Pbhip[ (v+y+z)  \QQ_{h}(|v+y+z|^2)].
\end{align*}
Using that $\P_+ \P_{-}=0$ and $\P_{+}[\P_{-}f \P_{-}g]=0$,  we see that  
\begin{align*}
-2\be \Pbhip[ (v+y+z) \P_{-}\dx (|v+y+z|^2)] = -2\be \Pbhip[ v \P_{-}\dx (|v+y+z|^2)]
\end{align*}
which establishes \eqref{veq}. Lastly applying $\P_{-,\text{hi}}$ to both sides of \eqref{INLS2} and using $\P_{-,\text{hi}}[ \P_{-,\text{lo}}f \cdot \P_{+}g]=0$ gives \eqref{weq}.
\end{proof}

 We further study the $y$ term in \eqref{YZ}, by writing
\begin{align}
 y &= \P_{-,\text{hi}}[ e^{i\be F} (\Id -\P_{-,\text{hi}})u]  + \P_{-,\text{hi}}[ e^{i\be F} w] \notag   \\
 & = \P_{-,\text{hi}}[ e^{i\be F} (\Id -\P_{-,\text{hi}}) ( e^{-i\be F}(v+y+z))   ]  -\Pblo [e^{i\be F}w] - \Pbhip[e^{i\be F}w] + e^{i\be F} w    \notag \\
& = \mathcal{Y}^{\text{pos}}[v+y+z] +\mathcal{Y}^{\text{neg}}[v+y+z,w]+ e^{i\be F}w, \label{Ydecomp}
\end{align} 
where 
\begin{align}
 \Y^{\text{pos}} [f]& : = \P_{-,\text{hi}}[ e^{i\be F[f]} (\Id -\P_{-,\text{hi}}) ( e^{-i\be F[f]}f)   ], \label{Ysharp}\\
 \Y_{\textup{lo}}^{\text{neg}}[f,g] & : = -\Pblo [e^{i\be F[f]}\P_{-}g], 
 \label{Yneglo} 
 \\
\Y_{\textup{hi}}^{\text{neg}}[f,g]& :=- \Pbhip[e^{i\be F[f]}\P_{-}g], \label{Yneghi}\\
\Y^{\text{neg}}[f,g]&:= \Y_{\textup{lo}}^{\text{neg}}[f,g]+\Y_{\textup{hi}}^{\text{neg}}[f,g]. \label{Yw}
\end{align}
Notice that if $u$ belongs to the Hardy space, then \eqref{Ydecomp} simplifies to 
\begin{align}
y= \Y^{\text{pos}}[v+y+z] \quad \text{if} \quad u\in L^2_{+}(\R). 
\notag 
% \label{Hardyy}
\end{align}

Due to the frequency interactions, the terms $\Y^{\pos}[\cdot]$ and $\Y^{\negg}[\cdot]$ enjoy extra spatial smoothing when measured in $L^{\infty}_{T}H^{s}_x$ and $L^{\frac{3}{2}}_{T,x}$.  We defer the proof to the next section.

\begin{lemma}\label{LEM:ysharp}
Let $0 \le s < 1$, $N\in 2^{\N}$ such that $N\gg 1$, $0<T\leq 1$, and $b>\frac 12$. 
Then, for any $\s>0$ and $0 \le \al < \frac12 + \s$, it holds that
\begin{align}
\| \Y^{\textup{pos}}[f]\|_{L^{\infty}_{T}H^{\s+\al}_x} & \les  
  (1+\| f\|_{L^\infty_T L^2_x})^2 \| f\|^3_{L^\infty_T H^\s_x},
\label{YposLTHs1} \\
\|\Y_{\textup{lo}}^{\textup{neg}}[f,g]\|_{L^{\infty}_{T}H^{s}_x}+\|\Y_{\textup{hi}}^{\textup{neg}}[f,g]\|_{L^{\infty}_{T}H^{s}_x} &\les  (1+\|f\|_{L^\infty_T H_x^{\max(s-\frac12+,0)}})^4 \|g\|_{L^\infty_T H^{s}_x} \label{YnegLTHs1}.
\end{align}
Moreover, there exists a non-negative polynomial $C:\R^{2}\to \R_{+}$ such that for $A_{N}=N^{-\frac{3}{2}-\min(s,\frac 12)+}$, we have
\begin{align}
\| \P_{N}\Y^{\textup{pos}}[f]\|_{L^{\frac 32}_{T,x}} &\les |\be|  \jb{\be}^2 A_{N} (1+ \|f\|_{X^{0,b}_{T}})^{6} \|f\|_{X^{s,b}_{T}} , 
\label{YPOS1}
\\ 
\| \P_{N}\big(\Y^{\textup{pos}}[f_1]-\Y^{\textup{pos}}[f_2]\big) \|_{L^{\frac 32}_{T,x}} & \les \jb{\be}^{2} A_{N}  C(\|f_1\|_{X^{s,b}_{T}}, \|f_2\|_{X^{s,b}_{T}}) \|f_1-f_2\|_{X^{s,b}_{T}},   \label{YPOS2} \\
\| \P_{N}\Y^{\textup{neg}}_{\textup{hi}}[f,g]\|_{L^{\frac{3}{2}}_{T,x}} &\les  \jb{\be} A_{N}(1+ \|f\|_{X^{s,b}_{T}})^{6} \|g\|_{X^{0,b}_{T}} , \label{YNEG21}\\
\| \P_{N}\big(\Y^{\textup{neg}}_{\textup{hi}}[f_1,g]-\Y^{\textup{neg}}_{\textup{hi}}[f_2,g]\big)\|_{L^{\frac{3}{2}}_{T,x}}
&\les \jb{\be}^2 A_{N}  C(\|f_1\|_{X^{s,b}_{T}}, \|f_2\|_{X^{s,b}_{T}}) \notag\\
& \hphantom{XXXXXXXXXX}\times \|f_1-f_2\|_{X^{s,b}_{T}} \|g\|_{X^{0,b}_{T}}. \label{YNEG22}
\end{align}
\end{lemma}

The main difficulty is to then obtain the trilinear estimate for the following term in the $y$-equation:
\begin{align*}
\P_{-,\textup{hi}}[ (v+y+z) \P_{-}\dx(|v+y+z|^2)] .
\end{align*}
We define 
\begin{align}
\begin{split}
\NN_{0}(f,g,h) : = -2\be \P_{-,\textup{hi}}\big[ f &\P_{-}\dx(  \cj{g} h)\big] 
+ 2\be \P_{-,\textup{HI}}\big[ f \P_{-,\textup{HI}}\dx( \cj{g}h)\big]
\end{split}\label{N0} 
\end{align} 
which is defined so that either the output frequency is of size one or the frequency for the derivative is bounded. For either of these reasons, we will be able to estimate $\NN_0$ using either two uses of the bilinear Strichartz estimate or the $L^4$-Strichartz estimate. 

By definining the trilinear operator
\begin{align}
\NN_{1}(f,g,h):=-2\be \P_{-,\textup{HI}}[f \P_{-,\text{HI}}\dx( \cj{g}\cdot h)] \label{N1}
\end{align}
we have reduced to the essential term
\begin{align}
-2 \be \P_{-,\textup{HI}}[ (v+y+z) \P_{-,\text{HI}}\dx(|v+y+z|^2)]=\NN_{1}(v+y+z), 
\notag
% \label{Yterm1}
\end{align}
namely, 
\begin{align*}
-2\be \P_{-,\textup{hi}}\big[ (v+y+z) \P_{-}\dx(  |v+y+z|^2)\big]  = \mathcal{N}_{0}(v+y+z) +\mathcal{N}_{1}(v+y+z).
\end{align*}
where we abused notation and defined $\NN_j( f):=\NN_0 (f,f,f)$, $j=0,1$.
The key difference with this trilinear form as compared to that in \eqref{vtrilin} is that the frequency projections do not in general imply that there are two input frequencies controlling the derivative. This means that we genuinely see high$\times$low $\to$ high interactions and thus we need a more careful analysis.
First, by impossible frequency interactions, we have 
\begin{align*}
\P_{-,\text{HI}}[ |z|^2 + \cj{y}v+\cj{y}z+\cj{z}v]=0,
\end{align*}
so that 
\begin{align*}
\NN_{1}(v+y+z)&= -2\be \P_{-,\textup{HI}}[ (v+y+z) \P_{-,\text{HI}}\dx( |v|^2 +|y|^2 +\cj{v}(y+z)+\cj{z}y  )]  \\
& = -2\be \P_{-,\textup{HI}}[ (v+y+z) \P_{-,\text{HI}}\dx( |v|^2 +\cj{v}z)]  \\
& \quad -2\be \P_{-,\textup{HI}}[ (v+y+z) \P_{-,\text{HI}}\dx( \cj{(v+y+z)}y)]  \\
& =  -2\be \P_{-,\textup{HI}}[ (v+y+z) \P_{-,\text{HI}}\dx( |v|^2 +\cj{v}z)]  \\
& \quad -2\be \sum_{ \substack{N_0, N_1,N_2, N_{23} \\ N_0\le 2^{10} (N_1\vee N_2) }}\P_{N_0} \P_{-,\textup{HI}}[\P_{N_1} (v+y+z) \P_{N_{23}}\P_{-,\text{HI}}\dx( \cj{\P_{N_2}(v+y+z)} y)] \\
& \quad -2\be \sum_{ \substack{N_0,  N_1,N_2, N_{23} \\ N_0  > 2^{10} ( N_1 \vee N_2)  }}\P_{N_0} \P_{-,\textup{HI}}[ \P_{N_1}(v+y+z) \P_{N_{23}}\P_{-,\text{HI}}\dx( \cj{\P_{N_2}(v+y+z)}y)]  \\
&=: \NN_{1}(v+y+z,v,v+z) + \NN_{2}^{\perp}(v+y+z,v+y+z,y) + \NN_{2}(v+y+z, v+y+z, y)
\end{align*}
where 
\begin{align}
 \NN_{2}^{\perp}(f,g,h)&:=-2\be \sum_{ \substack{N_0, N_1,N_2, N_{23}, \\  N_0 \le 2^{10} ( N_1\vee N_2) }}\P_{N_0} \P_{-,\textup{HI}}[\P_{N_1} f \P_{N_{23}}\P_{-,\text{HI}}\dx( \cj{\P_{N_2}g}\cdot h)] , 
 \label{NN2perp}
 \\
\NN_{2}(f,g,h)& := -2\be \sum_{ \substack{N_0,  N_1,N_2, N_{23} \\ N_0  > 2^{10} ( N_1 \vee N_2)  }}\P_{N_0} \P_{-,\textup{HI}}[ \P_{N_1}f \cdot  \P_{N_{23}} \P_{-,\text{HI}}\dx( \cj{\P_{N_2}g} \cdot  h)]. \label{NN2}
\end{align}

As it is, it seems to be impossible to control the nonlinear term $\NN_{2}$ due to the high$\times$low$\to$high interaction. To overcome this, we insert the decomposition \eqref{Ydecomp} into the last slot of $\NN_2$ to obtain
\begin{align*}
\NN_{2}(v+y+z, v+y+z, y) & = \NN_{2}(v+y+z, v+y+z, \Y^{\pos}[v+y+z]+\Y^{\negg}[v+y+z,w])  \\
&\quad + \NN_{2}(v+y+z, v+y+z, e^{i\be F}w).
\end{align*}
In fact, by definition of $\NN_{2}$ and $\Y^{\negg}$ in \eqref{Yw}, we have
\begin{align}
\NN_{2}(v+y+z, v+y+z, \Y^{\negg}[v+y+z,w]) 
&= 0.
 \notag 
 % \label{Ynegsimple}
\end{align}
To simplify the notation, we define
\begin{align}
\label{NN2pos}
 \NN_{2}^{\pos}(f,g,h)&: = \NN_{2}(f, g, \Y^{\pos}[h]) \quad \text{and} \quad  \NN_{2}^{\pos}(f): = \NN_{2}^{\pos}(f, f,f).
\end{align}

To summarise, we have the decomposition
\begin{align}
\begin{split}
-2\be& \P_{-,\textup{hi}}[ (v+y+z) \P_{-}\dx(|v+y+z|^2)]  \\
& = \NN_{0}(v+y+z)  +\NN_{1}(v+y+z,v,v+z) +\NN_{2}^{\perp}(v+y+z,v+y+z,y) \\
&  \quad +  \NN_{2}^{\pos}(v+y+z) 
+ \NN_{2}(v+y+z, v+y+z, e^{i\be F}w).
\end{split}
\notag 
% \label{ydecomp}
\end{align}
Thus, we arrive at the following modified equation for $y$:
\begin{align}
\dt y + i\dx^2 y & =\NN_{0}(v+y+z)  +\NN_{1}(v+y+z,v,v+z) +\NN_{2}^{\perp}(v+y+z,v+y+z,y)  \notag\\ 
& \hphantom{XXXXXXXX}  +\NN_{2}^{\pos}(v+y+z) + \P_{-,\textup{hi}}[ (v+y+z) \QQ_{h}(|v+y+z|^2)] \notag \\
& \hphantom{XXXXXXXX}
+ \NN_{2}(v+y+z, v+y+z, e^{i\be F}w)
\notag \\
&  = : \mathfrak{N}_{y}[v,y,z,w].
\label{Yeq2}
\end{align}
Note that when $u$ is in the Hardy space, the last term in \eqref{Yeq2} vanish.

\subsection{Estimating $\Y^{\pos}$ and $\Y^{\negg}$}

Before we proceed to the proof of Lemma~\ref{LEM:ysharp}, we need some preliminary results.
The first of these are variants of the bilinear Strichartz estimates in Lemma~\ref{LEM:bilin}.

\begin{lemma}\label{LEM:bilinvar}
Let $s\geq 0$ and $b>\frac 12$. Then, it holds that
\begin{align}
\| \P_{\ges N}[u \cj{v}]\|_{L^{2}_{T,x}} &\les  N^{-\frac 12}\|u\|_{X^{0,b}_{T}}\|v\|_{X^{0,b}_{T}},\label{highbilin} \\
\| \P_{N}[ u \cj{v}]\|_{L^{6}_{T,x}}&\les N^{\frac 12-s} \big( \| u\|_{X^{s,b}_{T}}\|v\|_{X^{0,b}_{T}}+ \| u\|_{X^{0,b}_{T}}\|v\|_{X^{s,b}_{T}}\big),\label{L6bilin} \\
\| \P_{\les N}[ u \cj{v}]\|_{L^{6}_{T,x}}&\les N^{\max(\frac 12-s,0+)} \big( \| u\|_{X^{s,b}_{T}}\|v\|_{X^{0,b}_{T}}+ \| u\|_{X^{0,b}_{T}}\|v\|_{X^{s,b}_{T}}\big).\label{L6bilin2}
\end{align}
\end{lemma}
\begin{proof}
The estimate \eqref{highbilin} follows from applying \eqref{bilin1} (with $\dl=0$) since 
\begin{align*}
\| \P_{\ges N}(u \cj{v})\|_{L^{2}_{T,x}}^{2} = \sum_{M\ges N} \|\P_{M}[ u \cj{v}]\|_{L^{2}_{T,x}}^{2} 
 \les \sum_{M\ges N} M^{-\frac 12} \|u\|_{X^{0,\frac 12+}_{T}}\|v\|_{X^{0,\frac 12+}_T } \les \text{RHS} \eqref{highbilin}.
\end{align*}
Next, \eqref{L6bilin} follows from interpolating the two estimates
\begin{align}
\| \P_{N}[ u \cj{v}]\|_{L^{\infty}_{T,x}} &\les N \big( \|\P_{\ges N} u\|_{L^{\infty}_{T}L^2_x} \|v\|_{L^{\infty}_{T}L^2_x}+\| u\|_{L^{\infty}_{T}L^2_x} \|\P_{\ges N}v\|_{L^{\infty}_{T}L^2_x} \big)\notag \\
& \les  N^{1-s}\big( \| u\|_{X^{s,b}_{T}}\|v\|_{X^{0,b}_{T}}+ \| u\|_{X^{0,b}_{T}}\|v\|_{X^{s,b}_{T}}\big), 
\notag
%\label{L6bilin2}
\\
\| \P_{N}[ u \cj{v}]\|_{L^{2}_{T,x}}
&
\les N^{-\frac12-s } \big( \| u\|_{X^{s,b}_{T}}\|v\|_{X^{0,b}_{T}}+ \| u\|_{X^{0,b}_{T}}\|v\|_{X^{s,b}_{T}}\big), 
\notag
\end{align}
where the latter follows from \eqref{bilin1} for $\dl=0$ with the decomposition $\P_{N}[u\cj{v}]= \P_{N}[ \P_{\ges N}u \cdot \cj{v}] +\P_{N}[ \P_{\ll N} u \cdot  \cj{ \P_{\ges N}v}]$. 
Lastly, \eqref{L6bilin2} follows from \eqref{L6bilin} in a similar way as we deduced \eqref{highbilin} from \eqref{bilin1}.
\end{proof}

Next, we give a smoothing estimate on the terms $e^{i\be F[f]}$ in the `weaker' space $L^{3}_{T,x}$.

\begin{lemma}
Let $s\geq 0$, $b>\frac 12$, $N\gg 1$, and assume that $F_{j}$, $j=1,2,$ are two real-valued functions such that $\dx F_{j}=|f_j|^2$.
Then, it holds that
\begin{align}
\| \P_{N}e^{i\be F_1}\|_{L^{3}_{T,x}} &\les
|\be| \jb{\be}  
N^{-2+\max(1-s,0)+} \|f_1\|_{X^{0,b}_{T}}(1+\|f_1\|_{X^{0,b}_{T}})^2 \|f_1\|_{X^{s,b}_{T}}, 
\label{L2pos2} \\
\| \P_{N}(e^{i\be F_1} - e^{i\be F_2})\|_{L^{3}_{T,x}} & \les  |\be|^2\jb{\be} N^{-2+\max(1-s,0)+} C_2( \|f_1\|_{X^{s,b}_{T}}, \|f_2\|_{X^{s,b}_{T}}) \|f_1-f_2\|_{X^{s,b}_{T}} \label{L2posdiff}
\end{align}
for some non-negative polynomial $C_2:\R^2\to \R_{+}$.
\end{lemma}
\begin{proof}
We first show \eqref{L2pos2}, which follows from \eqref{L4} and \eqref{YsCTHs} once we prove the following bound:
\begin{align}
\| \P_{N} e^{i\be F_1} \|_{L^3_{T,x}}  & \les |\be|\jb{\be} N^{-2+\max(1-s,0)+}\|f_1\|_{L^{6}_{T,x}} \big\{  \|J^{s}_x f_1\|_{L^{6}_{T,x}}+ \|f_1\|_{L^{\infty}_{T}L^2_x}  \|f_1\|_{L^{6}_{T,x}}      \|f_1\|_{L^{\infty}_{T}H^{s}_x} \big\}.
 \label{EL31}
\end{align}
For $0\le s < \frac12$, using the following bound 
\begin{align}
\| \wt{\P}_{N} e^{i\be F_1}\|_{L^{3}_{T,x}}& \les |\be| N^{-1} \| \P_{N}( |f_1|^2 e^{i\be F_1})\|_{L^{3}_{T,x}} \les |\be| N^{-1} \|f_1\|_{L^{6}_{T,x}}^{2},  
\notag
% \label{Ebd1} 
\end{align}
together with  Bernstein's inequality, and Sobolev embedding, gives
\begin{align*}
N\| {\P}_{N} e^{i\be F_1}\|_{L^{3}_{T,x}} &\les 
|\be| \big\{  \| \P_{\ges N}(|f_1|^2)\|_{L^{3}_{T,x}} + \|\P_{\ll N}(|f_1|^2)\|_{L^{\infty}_{T,x}} \| \wt{\P}_{N}e^{i\be F_1}\|_{L^{3}_{T,x}} \big\} \\
&\les |\be| \jb{\be} N^{-s} \big\{ 
\|f_1\|_{L^{6}_{T,x}} \|J_x^{s}f\|_{L^{6}_{T,x}}  + \| f_1\|_{L^{\infty}_{T}L^{2}_x} \| f_1\|_{L^{\infty}_{T}L^{\frac{2}{1-2s}}_x} \|f_1\|_{L^{6}_{T,x}}^{2} 
\big\} \\
& \les |\be| \jb{\be} N^{-s} \|f_1\|_{L^{6}_{T,x}} \big\{  \|J^{s}_x f_1\|_{L^{6}_{T,x}}+ \|f_1\|_{L^{\infty}_{T}L^2_x}  \|f_1\|_{L^{6}_{T,x}}      \|f_1\|_{L^{\infty}_{T}H^{s}_x} \big\}, 
\end{align*} 
proving \eqref{EL31}.

If instead $s\geq \frac 12$, we estimate the factor $\|\P_{\ll N}(|f_1|^2)\|_{L^\infty_{T,x}}$ above differently. By dyadic decomposition, we have 
\begin{align*}
 \|\P_{\ll N}(|f_1|^2)\|_{L^{\infty}_{x}} & \les \sum_{M \ll N} \sum_{N_1,  N_2} \| \P_{M}( \cj{\P_{N_1}f_1} \cdot \P_{N_2}f_1)\|_{L^{\infty}_{x}}.
\end{align*}
Without loss of generality, we may assume that $N_1 \geq N_2$ and note that we always have $M\les N_1$. 
If $N_1 \sim N_2$, then this contribution is bounded by 
\begin{align*}
\sum_{M \ll N} \sum_{N_1 \sim N_2} M \| \cj{\P_{N_1}f_1}\cdot \P_{N_2}f_1]\|_{L^{1}_x}
& \les \sum_{M \ll N} M \sum_{N_1 \sim N_2}  \|\P_{N_1}f_1\|_{L^2_x} \|\P_{N_2}f_1\|_{L^2_x} \\
& \les \sum_{M \ll N} M \sum_{N_1 \sim N_2} N_1^{-s}  \|\P_{N_1}f_1\|_{H^s_x} \|\P_{N_2}f_1\|_{L^2_x}  \\
& \les \bigg(\sum_{M \ll N} M^{1-s} \bigg) \|f_1\|_{H^{s}_x} \|f_1\|_{L^2_x} \\
&\les N^{\max(1-s,0)+}  \|f_1\|_{H^{s}_x} \|f_1\|_{L^2_x}
.
\end{align*}
If instead $N_1\gg N_2$, then $N_1\sim M$ and we instead get
\begin{align*}
\sum_{M \ll N} \sum_{N_1 \gg N_2} M \| \P_{N_1}f_1\cdot \P_{N_2}f_1\|_{L^{1}_x}
&\les N^{\max(1-s,0)+}  \|f_1\|_{H^{s}_x} \|f_1\|_{L^2_x}.
\end{align*}
Combining both cases, we then obtain \eqref{EL31}.

As for the difference estimate in \eqref{L2posdiff}, it follows by similar arguments as we used for \eqref{EL31} and its improvement, with the addition
of the mean value theorem and \eqref{Fdef}.
In particular, we have
\begin{align}
    \| \wt{\P}_{N} (e^{i\be F_1}-e^{i\be F_2})\|_{L^{3}_{T,x}} &\les |\be| \jb{\be} N^{-1} \big\{ \big(  \|f_1\|_{L^{6}_{T,x}}+ \|f_2\|_{L^{6}_{T,x}} \big) \|f_1-f_2\|_{L^{6}_{T,x}} \notag \\
 & \quad  +  \|f_2\|_{L^{6}_{T,x}}^{2} (\|f_1\|_{L^{\infty}_{T}L^2_x} + \|f_2\|_{L^{\infty}_{T}L^2_x})\|f_1-f_2\|_{L^{\infty}_{T}L^2_x}
 \big\} , 
  \label{Ebd2} 
 \\
 \| e^{i\be F_1} -e^{i\be F_2}\|_{L^{\infty}_{T,x}} 
&\les|\be| (\|f_1\|_{L^{\infty}_{T}L^{2}_x} +\|f_2\|_{L^{\infty}_{T}L^2_x}) \|f_1- f_2\|_{L^{\infty}_{T}L^2_x} \notag \\
& \les |\be| ( \|f_1\|_{X^{0,b}_{T}}+\|f_2\|_{X^{0,b}_{T}})\|f_1-f_2\|_{X^{0,b}_{T}}. \label{Expdiffs}
\end{align}
Then, to upgrade \eqref{Ebd2}, we use that 
\begin{align*}
\P_{N}(e^{i\be F[f_1]}-e^{i\be F[f_2]}) & = \be \dx^{-1} \P_{N}( (|f_1|^2-|f_2|^2)e^{i\be F[f_1]}) + \be\dx^{-1} \P_{N}( |f_2|^2 (e^{i\be F[f_1]}-e^{i\be F[f_2]} )  )
\end{align*}
and repeat the arguments for the improvement of \eqref{EL31}, placing the difference $e^{i\be F_1}-e^{i\be F_2}$ into $L^{\infty}_{T,x}$ and using \eqref{Expdiffs} after obtaining any high frequency gain from this term. We omit the details.
\end{proof}

\begin{proof}[Proof of Lemma~\ref{LEM:ysharp}]

The estimate in \eqref{YnegLTHs1} follows from \eqref{Yneglo}, \eqref{Yneghi}, and \eqref{L2F1}, thus we first show \eqref{YposLTHs1}. From \eqref{Ysharp} with $F=F[f]$, we have
\begin{align}
    \| \Y^\pos [f] \|_{L^\infty_T H^{\s+\al}_x}
    &
    \les \sum_{N_1 \ges N \lor N_2} N^{\s+\al} \big\| \P_N \big[ \P_{N_1} e^{i\be F} \cdot \P_{N_2}(e^{-i\be F} f ) \big] \big\|_{L^\infty_T L^2_x} 
    \notag
    \\
    &
    \les |\be| \sum_{N_1 \ges N \lor N_2} N^{\s+\al} N_1^{-1} \| \P_{N_1} (e^{i\be F} |f|^2) \|_{L^\infty_T L^2_x} \| \P_{N_2} (e^{-i\be F } f ) \|_{L^\infty_{T,x}}.
    \label{NS1a}
\end{align}
For the second factor in \eqref{NS1a}, by H\"older's inequality and Sobolev embedding, we have
\begin{align*}
    \| \P_{N_2} (e^{-i\be F } f ) \|_{L^\infty_{T,x}}
    \les 
    N_2^{\frac1p+} \| \P_{N_2} (e^{-i\be F } f ) \|_{L^\infty_{T} L^p_x}
    &
    \les 
    N_2^{\frac1p +} \| f \|_{L^\infty_T H^{\frac12-\frac1p}_x}
    \les 
    N_2^{(\frac12-\s)\lor 0 +} \|f\|_{L^\infty_T H^\s_x },
\end{align*}
for $p=\frac{2}{(1-2\s)\lor 0}$. 
To estimate the first factor in \eqref{NS1a}, we write 
\begin{align*}
    \P_{N_1}(e^{i\be F} |f|^2) 
    &
    = \P_{N_1} \big(  e^{i\be F} \cdot \P_{\ges N_1} (|f|^2) \big)  + \P_{N_1} \big( \wt\P_{N_1} e^{i\be F} \cdot \P_{\ll N_1} (|f|^2) \big)
\end{align*}
and estimate each term: for $\s>0$, $p=\frac{2}{(1-2\s)\lor 0}$, and $\frac12 = \frac1p + \frac1q$, 
\begin{align*}
 \big\|  
  \P_{N_1} \big(  e^{i\be F} \cdot \P_{\ges N_1} (|f|^2) \big) 
 \big\|_{L^\infty_T L^2_x } 
 & 
 \les \| f \|_{L^\infty_T L^{p}_x} \|\P_{\ges N_1} f \|_{L^\infty_T L^{q}_x}
 \\
 &
 \les \| f \|_{L^\infty_T H^\s_x} \| \P_{\ges N_1} f \|_{L^\infty_T H^{\frac12- \frac1q}_x}
 \\
 &
 \les N_1^{(\frac12-\s)\lor 0 - \s} \| f \|^2_{L^\infty_T H^\s_x}, 
 \\
 \big\| 
 \P_{N_1} \big( \wt\P_{N_1} e^{i\be F} \cdot \P_{\ll N_1} (|f|^2) \big)
 \big\|_{L^\infty_T L^2_x } 
 & \les N_1^\frac12 \| \wt\P_{N_1} e^{i\be F} \|_{L^\infty_T L^1_x} \|\P_{\ll N_1} (|f|^2) \|_{L^\infty_{T,x}}
 \\
 &
 \les |\be| N_1^{\frac12 -1 + \frac2p +} \| \wt\P_{N_1} (e^{i\be F}|f|^2 )\|_{L^\infty_T L^1_x} \| |f|^2 \|_{L^\infty_T L^{\frac{p}{2}-}_x}
 \\
 &
 \les |\be| N_1^{-\frac12 + (1-2\s) \lor 0 + } \| f\|_{L^\infty_T L^2_x}^2 \| f\|^2_{L^\infty_T H^\s_x}, 
\end{align*}
which combined give
\begin{align*}
    \|  
  \P_{N_1}(e^{i\be F} |f|^2) 
 \|_{L^\infty_T L^2_x }  \les \jb{\be} N_1^{ - \s + (\frac12-\s)\land 0 +} (1+\| f\|_{L^\infty_T L^2_x})^2 \| f\|^2_{L^\infty_T H^\s_x}. 
\end{align*}
Then, we conclude that
\begin{align*}
    \| \Y^\pos [f] \|_{L^\infty_T H^s_x}
    &
    \les |\be| \jb{\be} (1+\| f\|_{L^\infty_T L^2_x})^2 \| f\|^3_{L^\infty_T H^\s_x}
    \sum_{N_1 \ges N \lor N_2} 
    N_1^{\al - \frac12 - \s } 
\end{align*}
with a negative power of $N_1$
given that $\al < \frac12 + \s$, which suffices to sum in the dyadics, completing the estimate.

Next, we obtain the time-averaged estimates for $\Y^{\pos}[f]$, as to illustrate the method in a simpler context. From \eqref{Ysharp}, we write 
\begin{align*}
\Y^{\pos}[f] = \P_{-,\text{hi}}[ e^{i\be F[f]} \Pbhip ( e^{-i\be F[f]}f)   ] + \P_{-,\text{hi}}[ e^{i\be F[f]} \Pblo ( e^{-i\be F[f]}f)   ] =: \Y_{1}^{\pos} [f]+\Y^{\pos}_{2}[f].
\end{align*}
By H\"{o}lder's inequality, \eqref{highbilin}, \eqref{L6bilin2}, and \eqref{L2pos2}, and writing $F=F[f]$, we have
\begin{align}
& \|\P_{N}\Y_{1}^{\pos} [f] \|_{L^{\frac 32}_{T,x}}\notag \\
&\les \sum_{N_1 \ges N\vee N_2}\| \P_{N_1}e^{i\be F} \cdot \P_{N_2}( e^{-i\be F}f)\|_{L^{\frac 32}_{T,x}}\notag \\
& \les |\be|  \sum_{N_1 \ges N\vee N_2} \Big\{  \| \dx^{-1}\P_{N_1}( \P_{\ll N_1}(|f|^2) \wt{\P}_{N_1}e^{i\be F})\|_{L^{2}_{T,x}} \notag  \\
 & \hphantom{XXXXXXXX} +  \|  \dx^{-1}\P_{N_1}( \P_{\ges N_1}(|f|^2) e^{i\be F})\|_{L^{2}_{T,x}} \Big\}  \|\P_{N_2}( e^{-i\be F}f)\|_{L^{6}_{T,x}}  \notag \\
  & \les |\be|  \|f\|_{L^{6}_{T,x}}\sum_{N_1 \ges N} N_{1}^{-1+}\Big\{ \|\P_{\ll N_1}(|f|^2)\|_{L^{6}_{T,x}}  \|\wt{\P}_{N_1}e^{i\be F}\|_{L^{3}_{T,x}} +\|\P_{\ges N_1}(|f|^2)\|_{L^2_{T,x}}  \Big\} \notag  \\
& \les |\be| \|f\|_{X^{0,b}_{T}}^2 \sum_{N_1 \ges N} N_{1}^{-1+}\Big\{ N_{1}^{\frac 12+} \|f\|_{X^{0,b}_{T}}  \|\wt{\P}_{N_1}e^{i\be F}\|_{L^{3}_{T,x}} +N_{1}^{-\frac 12}\|\P_{\ges N_1}f\|_{X^{0,b}_{T}} \Big\}\notag \\
& \les |\be| \jb{\be}^2\|f\|_{X^{0,b}_{T}}^{2}  (1+\|f\|_{X^{0,b}_{T}})^{4} \|f\|_{X^{s,b}_{T}}
\sum_{N_1 \ges N} N_{1}^{-\frac 32+}
\{ N_1^{-\min(s,1)+}+ N_1^{-s}\}  \notag \\
& \les |\be| \jb{\be}^2 \|f\|_{X^{0,b}_{T}}^{2} (1+\|f\|_{X^{0,b}_{T}})^{4}  \|f\|_{X^{s,b}_{T}} \sum_{N_1 \ges N} N_{1}^{-\frac 32-\min(s,1)+} \notag \\
& \les |\be| \jb{\be}^2 N^{-\frac 32-\min(s,1)+}\|f\|_{X^{0,b}_{T}}^{2} (1+\|f\|_{X^{0,b}_{T}})^{4}   \|f\|_{X^{s,b}_{T}}. \label{Yposbd1}
\end{align}
Applying $\P_{N}$ to $\Y_2^{\pos}[f]$ we see that the frequency projections imply 
\begin{align*}
\P_{N}\Y_{2}^{\pos}[f] =  \P_{N}\P_{-,\text{hi}}[ \wt{\P}_{N}e^{i\be F[f]} \cdot  \Pblo ( e^{-i\be F[f]}f)   ] 
\end{align*}
so that we can apply the same argument as for $\Y_{1}^{\pos}[f]$ in \eqref{Yposbd1}, since $N\sim N_1$ and $N_2 \les 1$. Thus, \eqref{YPOS1} follows.

In the following, we use the shorthand notation $F_j : = F[f_{j}]$, $j=1,2$, and $G_{12}:= e^{i\be F_1}-e^{i\be F_2}$. 
Now we consider the difference estimate for the operator $\Y^{\pos}_{1}$, which we further write~as
\begin{align*}
\Y_{1}^{\pos}[f_1]-\Y_{1}^{\pos}[f_2]& =  \P_{-,\text{hi}}\big[ G_{12}\Pbhip[ e^{-i\be F_1} f_1]\big] +\P_{-,\text{hi}}\big[e^{i\be F_2} \Pbhip[ \cj{G_{12}} f_1]\big]  \\
& \quad  + \P_{-,\text{hi}}\big[ e^{i\be F_2}\Pbhip[ e^{-i\be F_2} (f_1-f_2)]\big].
\end{align*}
It is straightforward from \eqref{Yposbd1} and \eqref{Expdiffs} to see that we have
\begin{align*}
\|\P_{N}\P_{-,\text{hi}} &\big[e^{i\be F_2} \Pbhip[ \cj{G_{12}} f_1]\big] \|_{L^{\frac 32}_{T,x}} \\
& \les  N^{-\frac 32-\min(s,1)+} \|f_1\|_{X^{0,b}_{T}} \|f_2\|_{X^{0,b}_{T}} (1+\|f_2\|_{X^{0,b}_{T}})^4 \|G_{12}\|_{L^{\infty}_{T,x}}  \|f_2\|_{X^{s,b}_{T}} \\
&\les |\be |N^{-\frac 32-\min(s,1)+}\|f_1\|_{X^{0,b}_{T}} \|f_2\|_{X^{0,b}_{T}}  (1+\|f_1\|_{X^{0,b}_{T}}+\|f_2\|_{X^{0,b}_{T}})^5 \|f_2\|_{X^{s,b}_{T}} \|f_1-f_2\|_{X^{0,b}_{T}}
\end{align*}
and similar for $ \P_{-,\text{hi}}\big[ e^{i\be F_2}\Pbhip[ e^{-i\be F_2} (f_1-f_2)]\big].$ An estimate for $  \P_{-,\text{hi}}\big[ G_{12}\Pbhip[ e^{-i\be F_1} f_1]\big]$ then follows from repeating the argument in \eqref{Yposbd1} but this time using \eqref{L2posdiff}. The same strategy applies to estimate $\Y_{2}^{\pos}[f_1]-\Y_{2}^{\pos}[f_2]$, and \eqref{YPOS2} follows.

From \eqref{Yw}, we have $\P_{N}\Y^{\negg}[f,g]= -\P_{N}\Pbhip[ e^{i\be F[f]} \P_{-} g]$, and we see that the exponential term carries the highest frequency and the arguments for $\Y_{1}^{\pos}$ apply. As for \eqref{YNEG21} and \eqref{YNEG22}, we apply the same arguments using that the frequency signs in the expression \eqref{Yneghi} have again imply that the exponential factor carries the largest frequency.
\end{proof}

\section{Nonlinear estimates}\label{SEC:ests}
Given $\dl>0$ sufficiently small, we fix $b':=\frac 12-2\dl$ and $b:=\frac 12+\dl$.  Note that the following estimates hold for all $\dl>0$ sufficiently small.

\subsection{Bounds for $v$, $w$, and $z$}

The main point here is that the main trilinear forms for the $v$ and $w$ equations, in \eqref{veq} and \eqref{weq}, have a good sign structure which means that there is always at least two input frequencies which control the derivative. In particular, this allows us to additionally prove a nonlinear smoothing estimate for these terms.

\begin{lemma}[$v$ and $w$ bounds]\label{LEM:vwXsb}
For $0<\dl\ll 1$, let $s\geq s_0> 10\dl$, $0\leq \g<\min(\frac 12,s_0)-10\dl$, and $0<T\leq 1$. Then,
\begin{align}
\| \P_{\pm, \textup{hi}}[f_1 \P_{\mp}\dx( \cj{f_2} f_3)]\|_{X^{s+\g,-\frac 12+\dl}_{T}} \les T^{\dl} \max_{\s\in S_3}  \|f_{\s(1)}\|_{X^{s,b}_{T}}  \|f_{\s(2)}\|_{X^{s_0,b}_{T}}\|f_{\s(3)}\|_{X^{s_0,b}_{T}}. \label{vtrilin}
\end{align}
\end{lemma}

\begin{remark}\rm
When the frequency of the derivative is large, the argument in \eqref{doublebilin} suffices to close the estimate \eqref{vtrilin} when $\g=0$. However, in order to obtain the nonlinear smoothing (\eqref{vtrilin} with some $\g>0$), in the case high$\times$low$\times$low $\to$ high,  we need to see additional smoothing which requires the use of the phase function in \eqref{mods}.
\end{remark}

\begin{proof}
By Lemma~\ref{LEM:linXsb}, we reduce to estimating
\begin{align*}
\I : =\int_{0}^{T}\int_{\R} \jb{\dx}^{s+\g} \cj{g} \cdot \P_{\pm, \text{hi}}[(\P_{\pm}f_1) \P_{\mp}\dx( \cj{f_2} f_3)] dx dt
\end{align*}
for any $g\in X^{0,b'}_{T}$ with $\|g\|_{X^{0,b'}_{T}}\leq 1$. We let $F_{j}$ be any extensions of $f_j$ on $[0,T]$, $j=1,2,3$. and define $g_{T} := \ind_{[0,T)}g$. 
Then, by Cauchy-Schwarz, we have 
\begin{align*}
\I &=\sum_{N_{23}} \int_{\R}\int_{\R} \jb{\dx}^{s+\g}  \P_{\pm, \text{hi}} [\cj{g_{T}}]\cdot F_1  \P_{\mp}\P_{N_{23}} \dx ( \cj{F_2} F_{3})dx dt  \\
& = \sum_{ \substack{ N_{23},N_0,N_1,N_2,N_3 \\ N_{23} \les N_2 \vee N_3}}  \int_{\R}\int_{\R} \jb{\dx}^{s+\g}  \P_{\pm, \text{hi}} [\cj{\P_{N_0}g_{T}}]\cdot  \P_{N_1} F_{1} \cdot  \P_{\mp}\P_{N_{23}} \dx ( \cj{\P_{N_2}F_2} \cdot \P_{N_3}F_3)dx dt \\
& = :\sum_{ \substack{ N_{23},N_0,N_1,N_2,N_3 \\ N_{23} \les N_2 \vee N_3}} \I_{\cj{N}}.
\end{align*}
Note that we always have $N_{(1)}\sim N_{(2)}$.

\medskip
\noi
\underline{$\bullet$ \textbf{Case 1:} $N_{23}\les 1$}

\medskip
\noi
\underline{$\bullet$ \textbf{Case 1.1:} $N_0 \leq N_{(3)}$}

\smallskip
\noi
 By H\"{o}lder and Bernstein's inequality and \eqref{L4} we get
\begin{align}
 \sum_{ \substack{ N_{23},N_0,N_1,N_2,N_3 \\ N_{23} \les N_2 \vee N_3}} |\I_{\cj{N}}| & \les  \sum_{ \substack{ N_{23},N_0,N_1,N_2,N_3 \\ N_{23} \les N_2 \vee N_3}}  N_0^{s+\g} \| \P_{N_0}g_{T}\|_{L^4_{t,x}} \| \P_{N_1}F_1\|_{L^{4}_{t,x}} 
\|\P_{\mp} \P_{N_{23}} \dx ( \cj{\P_{N_2}F_2} \cdot \P_{N_3}F_3)\|_{L^{2}_{t,x}} \notag \\
& \les  \sum_{ \substack{ N_{23},N_0,N_1,N_2,N_3 \\ N_{23} \les N_2 \vee N_3}}  N_0^{s+\g} \| \P_{N_0}g_{T}\|_{L^4_{t,x}} \prod_{j=1}^{3} \| \P_{N_j}F_j\|_{L^{4}_{t,x}} \notag  \\
& \les \|g_{T}\|_{X^{0,b'}}  \sum_{ \substack{ N_{23},N_0,N_1,N_2,N_3 \\ N_{23} \les N_2 \vee N_3}}  N_0^{s+\g}  N_{(1)}^{-s}N_{(2)}^{-s_0} \max_{\s \in S_3} \| F_{\s(1)}\|_{X^{s,b}} \prod_{j=2}^3 \|F_{\s(j)}\|_{X^{s_0,b}}  \notag  \\
& \les \|g_{T}\|_{X^{0,b'}} \max_{\s \in S_3} \| F_{\s(1)}\|_{X^{s,b}} \prod_{j=2}^3 \|F_{\s(j)}\|_{X^{s_0,b}}, \label{L4tri}
\end{align}
where in the last inequality we used that $\g<s_0$.

\medskip
\noi
\underline{$\bullet$ \textbf{Case 1.2:} $N_0 \geq  N_{(2)}$}

\smallskip
\noi
First,  we note that by interpolating between \eqref{localsmooth0} and the trivial estimate $\|\P_{N}f\|_{L^{2}_{t,x}}\les \|f\|_{X^{0,0}}$, we obtain
\begin{align}
\| \P_{N}f\|_{L^{4}_{x}L^{2}_{t}}\les N^{-\frac 14} \|\P_{N}f\|_{X^{0,\frac 14+}}. \label{L4locsmooth}
\end{align}
Then we consider two subcases: (i) $N_2 \vee N_3 \sim N_{(1)}$ or (ii) $N_1 \sim N_{(1)}$.
Consider first case (i). Assuming that $N_2\geq N_3$, we use H\"{o}lder's inequality, \eqref{L4locsmooth}, \eqref{localsmooth0}, and \eqref{maximal} to find
\begin{align*}
|\I_{\cj{N}}|& \les \| \jb{\dx}^{s+\g} \P_{\pm, \text{hi}}\P_{N_0}g_{T}\|_{L^{4}_{x}L^{2}_{t}} \|\P_{N_1}F_1\|_{L^{4}_{x}L^{\infty}_{t}} \| \P_{\mp}\P_{N_{23}}\dx( \cj{\P_{N_2}F_2} \cdot \P_{N_3}F_3)\|_{L^{2}_{t,x}} \\
& \les N_0^{s+\g-\frac 14} N_{1}^{\frac 14-s_0+}  \|g_{T}\|_{X^{0,b'}} \|F_1\|_{X^{s_0,b}} \|\P_{N_2}F_2\|_{L^{\infty}_{x}L^{2}_{t}} \|\P_{N_3}F_{3}\|_{L^{2}_{x}L^{\infty}_{t}} \\
&\les N_0^{s+\g-\frac 14} N_{1}^{\frac 14-s_0+}  N_{2}^{-\frac 12-s}N_{3}^{\frac 12-s_0+}  \|g_{T}\|_{X^{0,b'}}\|F_2\|_{X^{s,b}}\|F_1\|_{X^{s_0,b}}\|F_{3}\|_{X^{s_0,b}}  \\
&\les N_{(1)}^{\g-2s_0+}  \|g_{T}\|_{X^{0,b'}}\|F_2\|_{X^{s,b}}\|F_1\|_{X^{s_0,b}}\|F_{3}\|_{X^{s_0,b}} 
\end{align*}
which is a negative power of $N_{(1)}$ provided that $\g<\min(\frac34, 2s_0)$. The case when $N_3>N_2$ follows by swapping the roles of $N_2$ and $N_3$.

Now we consider case (ii) where $N_1\sim N_0\gg N_2\vee N_3$. We again use H\"{o}lder's inequality and \eqref{L4locsmooth} to find
\begin{align}
|\I_{\cj{N}}|& \les \| \jb{\dx}^{s+\g}\P_{\pm, \text{hi}}\P_{N_0}g_{T}\|_{L^{4}_{x}L^{2}_{t}} \|\P_{N_1}F_1\|_{L^{4}_{x}L^{2}_{t}} \| \P_{\mp}\P_{N_{23}}\dx( \cj{\P_{N_2}F_2} \cdot \P_{N_3}F_3)\|_{L^{2}_{x}L^{\infty}_{t}}  \notag \\
&\les N_{(1)}^{\g-\frac 12}\|g_{T}\|_{X^{0,b'}}\|F_1\|_{X^{s,b}} \| \P_{\mp}\P_{N_{23}}\dx( \cj{\P_{N_2}F_2} \cdot \P_{N_3}F_3)\|_{L^{2}_{x}L^{\infty}_{t}}. \label{vcaseii}
\end{align}
In order to control the last quantity, we note that 
\begin{align}
\| \P_{\mp}\dx \PbLO f\|_{L^{2}_{x}L^{\infty}_{t}} \les \|f\|_{L^{2}_{x}L^{\infty}_{t}}. 
\label{maximalproj}
\end{align}
Indeed, the operator $\P_{\mp}\dx \PbLO$ has integral kernel $k(x)  = \pm \int_{\mp \infty}^{0} i\xi \eta(\xi)e^{ix\xi} d\xi$
which by two rounds of integration by parts (using that $\xi\eta(\xi)\vert_{\xi=0}=0$)  gives the bound $|k(x)|\les \jb{x}^{-2}$, so $k\in L^1_x (\R)$ and \eqref{maximalproj} then follows from Young's inequality.
Then, \eqref{maximalproj}, Cauchy-Schwarz, and \eqref{maximal} imply
\begin{align*}
\| \P_{\mp}\P_{N_{23}}\dx( \cj{\P_{N_2}F_2} \cdot \P_{N_3}F_3)\|_{L^{2}_{x}L^{\infty}_{t}} & \les \| \P_{N_2}F_2\|_{L^{4}_{x}L^{\infty}_{t}} \|\P_{N_3}F_3\|_{L^{4}_{x}L^{\infty}_{t}} \\
&\les (N_2 N_3)^{\frac 14-s_0+} \|F_2\|_{X^{s_0,b}}\|F_3\|_{X^{s_0,b}}.
\end{align*}
Inserting this bound into \eqref{vcaseii} we find the final dyadic factor $N_{(1)}^{\g-2s_0+}$ which is a negative power for $\g<2s_0$.

\medskip
\noi
\underline{$\bullet$ \textbf{Case 2:} $N_{23}\gg 1$}

\smallskip
\noi
 In view of the signs of the frequencies, we have $N_1 \ges N_{23}\vee N_0$. 
 Moreover, $N_2 \vee N_3 \ges N_{23}$. 
 We first assume that $N_1 \les  N_{2}\vee N_3$.
 Without loss of generality, we suppose that $N_3 = N_2 \lor N_3$.
By Cauchy-Schwarz and the bilinear Strichartz estimate \eqref{bilin1} with $0<\dl<s_0$, we have 
\begin{align}
\begin{split}
|\I_{\cj{N}}| & \les N_{23} \| \P_{N_{23}}[  \jb{\dx}^{s+\g}\cj{ \P_{\mp,\text{hi}}\P_{N_0}g_{T}} \P_{N_1}F_1]\|_{L^2_{t,x}} \|\P_{N_{23}}[  \cj{\P_{N_2}F_2} \P_{N_3}F_{3}]\|_{L^{2}_{t,x}} \\
& \les N_{23}N_{23}^{-1+\dl} N_0^{s+\g} \|\P_{N_0}g_{T}\|_{X^{0,b'}} \|\P_{N_1}F_1\|_{X^{0,b}} \|\P_{N_2}F_2\|_{X^{0,b}} \| \P_{N_3}F_3\|_{X^{0,b}} \\
& \les N_{23}^{\dl} N_{0}^{s+\g} N_{1}^{-s}N_{3}^{-s_0} \|\P_{N_0}g_{T}\|_{X^{0,b'}} \|\P_{N_1}F_1\|_{X^{s,b}} \|\P_{N_2}F_2\|_{X^{0,b}} \| \P_{N_3}F_3\|_{X^{s_0,b}} \\
& \les N_{3}^{\g-s_0+\dl}  \|g_{T}\|_{X^{0,b'}} \|F_1\|_{X^{s,b}} \|F_2\|_{X^{s_0,b}} \| F_3\|_{X^{s_0,b}} .
\end{split} \label{doublebilin}
\end{align}
We then have a negative power of the maximum frequency provided that $\g<s_0-\dl$. We can then perform the dyadic summations. 

Now we assume that $N_1 \gg N_{2}\vee N_3$ and hence $N_1 \sim N_0$. 
We need to rely on the phase function to obtain additional smoothing.
Given $\tau, \tau_j, \xi,\xi_j \in \R$, with $\s := \tau - \xi^2$ and $\s_j := \tau_j - \xi_j^2$, $j=1,2,3$, satisfying $\tau = \tau_1 - \tau_2 + \tau_3$ and $\xi = \xi_1 - \xi_2+\xi_3$, we have the following resonance identity
\begin{align}
\s_1-\s_2 + \s_3 -\s= \xi_{1}^{2}-\xi_{2}^{2}+\xi_{3}^{2} -\xi^{2} = -2(\xi-\xi_1)(\xi-\xi_3) =: \Phi(\cj{\xi}). \label{mods}
\end{align}

\noi 
Note that under $\xi_1-\xi_2+\xi_3=\xi$, we have $\xi-\xi_1=\xi_3-\xi_2$.
Thus, for $|\xi_3-\xi_2| \sim N_{23} \gg 1$, we have
\begin{align}
\s_{\max}:=\max( |\s_1|,|\s_2|,|\s_3|,|\s|) \ges N_{23}|\xi-\xi_3|.  \label{nonres}
\end{align}
As $N_1\sim N_0\gg N_3$, \eqref{nonres} becomes $\s_{\max}\ges N_{23}N_0=:K$. We then write 
\begin{align*}
\I_{\cj{N}} =&  \int_{\R}\int_{\R} \jb{\dx}^{s+\g}  \P_{\pm, \text{hi}} [ \cj{ \Q_{\ges K}\P_{N_0}g_{T}}]\cdot  \P_{N_1} F_{1} \cdot  \P_{\mp}\P_{N_{23}} \dx ( \cj{\P_{N_2}F_2} \cdot \P_{N_3}F_3)dx dt  \\
& +\int_{\R}\int_{\R} \jb{\dx}^{s+\g}  \P_{\pm, \text{hi}} [ \cj{ \Q_{\ll K}\P_{N_0}g_{T}}]\cdot  \Q_{\ges K}\P_{N_1} F_{1} \cdot  \P_{\mp}\P_{N_{23}} \dx ( \cj{\P_{N_2}F_2} \cdot \P_{N_3}F_3)dx dt \\
&+ \int_{\R}\int_{\R} \jb{\dx}^{s+\g}  \P_{\pm, \text{hi}} [ \cj{ \Q_{\ll K}\P_{N_0}g_{T}}]\cdot  \Q_{\ll K}\P_{N_1} F_{1} \cdot  \P_{\mp}\P_{N_{23}} \dx ( \cj{ \Q_{\ges K}\P_{N_2}F_2} \cdot \P_{N_3}F_3)dx dt\\
& + \int_{\R}\int_{\R} \jb{\dx}^{s+\g}  \P_{\pm, \text{hi}} [ \cj{ \Q_{\ll K}\P_{N_0}g_{T}}]\cdot  \Q_{\ll K}\P_{N_1} F_{1} \cdot  \P_{\mp}\P_{N_{23}} \dx ( \cj{\Q_{\ll K}\P_{N_2}F_2} \cdot \Q_{\ges K}\P_{N_3}F_3)dx dt \\
& =: \I_{\cj{N}}^{(0)}+\I_{\cj{N}} ^{(1)}+\I_{\cj{N}} ^{(2)}+\I_{\cj{N}} ^{(3)}, 
\end{align*}
recalling the definition of $\Q_{\ges K}, \Q_{\ll K}$ in \eqref{Qpro}. 
To estimate $\I_{\cj{N}}^{(0)}$ we will use an auxiliary estimate, which we now derive. 
By H\"{o}lder's inequality, \eqref{L4}, and \eqref{L6hi} we have
\begin{align*}
\| \P_{N_{23}}[ \cj{\P_{N_2}F_2}\cdot \P_{N_3}F_3]\|_{L^{6}_{t,x}} \les  (N_2 \wedge N_3)^{\frac 12} \|\P_{N_2}F_2\|_{X^{0,b}}\|\P_{N_3}F_3\|_{X^{0,b}}.
\end{align*}
Interpolating this with \eqref{bilin1} (with $\dl=0$), we find
\begin{align}
\| \P_{N_{23}}[ \cj{\P_{N_2}} \P_{N_3}F_3]\|_{L^{3}_{t,x}}   \les N_{23}^{-\frac 14} (N_2 \wedge N_3)^{\frac 14} \|\P_{N_2}F_2\|_{X^{0,b}}\|\P_{N_3}F_3\|_{X^{0,b}} \label{L3est}.
\end{align}
Then, by H\"{o}lder's inequality, \eqref{L4}, and \eqref{L3est}, we have
\begin{align*}
|\I_{\cj{N}}^{(0)}| & \les \|\jb{\dx}^{s+\g}  \Q_{\ges K}\P_{N_0}g_{T}\|_{L^{2}_{t,x}} \|\P_{N_1}F_1\|_{L^{6}_{t,x}} \| \P_{N_{23}} \dx ( \cj{\P_{N_2}F_2} \cdot \P_{N_3}F_3)\|_{L^{3}_{t,x}} \\
& \les N_{0}^{s+\g}K^{-\frac 12+2\dl} N_1^{-s}N_{23}^{\frac 34}(N_2 \wedge N_3)^{\frac 14-s_0} (N_2\vee N_3)^{-s_0} \|g_{T}\|_{X^{0,b'}} \|F_1\|_{X^{s,b}}\|F_2\|_{X^{s_0,b}}\|F_{3}\|_{X^{s_0,b}} \\
& \les N_{0}^{\g-\frac 12+4\dl +2\max(\frac 14-s_0,0)}  \|g_{T}\|_{X^{0,b'}} \|F_1\|_{X^{s,b}}\|F_2\|_{X^{s_0,b}}\|F_{3}\|_{X^{s_0,b}} 
\end{align*}
which is a negative power of $N_{(1)}$ provided that $\g< \frac 12-2\max(\frac 14-s_0,0) -4\dl$.
In a similar manner, we estimate $\I_{\cj{N}}^{(1)}$ by placing $\Q_{\ges K}\P_{N_1}F_1$ into $L^{2}_{t,x}$ and gaining the factor $K^{-\frac 12-\dl}$, the product $\cj{\P_{N_2}F_2}\cdot \P_{N_3}F_3$ goes into $L^{3}_{t,x}$ and we use \eqref{L3est} again, while for the dual function, we place it into $L^{6}_{t,x}$ and use \eqref{L4}:
\begin{align*}
\| \Q_{\ll K}\P_{N_0}g_{T}\|_{L^{6}_{t,x}} \les \| \Q_{\ll K}\P_{N_0}g_{T}\|_{X^{0,\frac 12+}} \les K^{2\dl+}\|\P_{N_0}g_{T}\|_{X^{0,b'}}.
\end{align*}
Thus, overall we at least gain the factor $K^{-\frac 12+\dl+}$ which is sufficient to control all of the dyadic factors as we did for  $\I_{\cj{N}}^{(0)}$. 
For $\I_{\cj{N}}^{(2)}$, we use \eqref{bilin1} and Bernstein's inequality:
\begin{align*}
|I_{\cj{N}}^{(2)}| &\les N_{23} \| \P_{N_{23}}\big[  \P_{\pm, \text{hi}} [ \cj{ \jb{\dx}^{s+\g}\Q_{\ll K}\P_{N_0}g_{T}}]\cdot  \Q_{\ll K}\P_{N_1} F_{1} ]\|_{L^{2}_{t,x}} \|\Q_{\ges K}\P_{N_2}F_2\|_{L^{2}_{t,x}} \|\P_{N_3}F_3\|_{L^{\infty}_{t,x}} \\
& \les  N_{23}^{\frac 12+10\dl} N_{0}^{\g} K^{-\frac 12-\dl} N_3^{\frac 12-s_0} \|g_{T}\|_{X^{0,b'}}\|F_1\|_{X^{s,b}}\|F_2\|_{X^{0,b}}\|F_{3}\|_{X^{s_0,b}} \\
& \les N_0^{\g-\frac 12+10\dl +\max(\frac 12-s_0,0)}  \|g_{T}\|_{X^{0,b'}}\|F_1\|_{X^{s,b}}\|F_2\|_{X^{0,b}}\|F_{3}\|_{X^{s_0,b}},
\end{align*}
which is a negative power of $N_{(1)}$ provided that $\g<\frac 12-\max(\frac 12-s_0,0)-10\dl$. Finally we estimate $\I_{\cj{N}}^{(3)}$ the same way as we did for $\I_{\cj{N}}^{(2)}$ which amounts to interchanging the roles of $N_2$ and $N_3$. We omit the similar details. 
Collating the cases we find that 
\begin{align*}
\g<\min( s_0, \tfrac 12-2\max(\tfrac 14-s_0,0)-4\dl, \tfrac 12-\max(\tfrac 12-s_0,0)-10\dl)=\min(\tfrac12 ,s_0)-10\dl
\end{align*}
and $\g>0$ provided that $s_0>10\dl$ and $\dl<\frac{1}{20}$.
This completes the proof of \eqref{vtrilin}.
\end{proof}

\begin{lemma}[$z$ bound] \label{LEM:Zbd}
Let $\s\geq s_0>0$ and $0<T\leq 1$. Then,
\begin{align}
\| \Pblo[f_1 \P_{-}\dx( \cj{f_2} f_3)]\|_{X^{\s,-\frac 12+2\dl}_{T}} \les  \|f_{1}\|_{X^{s_0,b}_{T}}  \|f_{2}\|_{X^{s_0,b}_{T}}  \|f_{3}\|_{X^{s_0,b}_{T}}\label{ztrilin}
\end{align}
\end{lemma}
\begin{proof}
In view of the $\Pblo$ projection, we have 
\begin{align*}
\| \Pblo[f_1 \P_{-}\dx( \cj{f_2} f_3)]\|_{X^{\s,-b'}} \les \| \Pblo[f_1 \P_{-}\dx( \cj{f_2} f_3)]\|_{X^{0,-b'}_{T}}.
\end{align*}
We then apply Lemma~\ref{LEM:linXsb}, and reduce to controlling 
\begin{align*}
\I_{\cj{N}}^{\lo}:= \int_{\R} \int_{\R}  \cj{\Pblo g_{T}} \cdot  \P_{N_1}F_1 \cdot  \P_{N_{23}}\P_{-}\dx ( \cj{\P_{N_2}F_2} \cdot \P_{N_3}F_3) dx dt
\end{align*}
for $g\in X^{0,b'}_{T}$ such that $\|g\|_{X^{0,b'}_{T}}\leq 1$, $g_T = \ind_{[0,T]} g$, and $F_j$ an extension of $f_j$ on $[0,T]$, $j=1,2,3$. If $N_{23}\les 1$, then we argue as in \eqref{L4tri} using the $L^4$-Strichartz estimate and the embedding \eqref{L4}. If instead $N_{23}\gg 1$, then we in fact have $N_1 \sim N_{23}\les N_2 \vee N_3$ and we can apply the argument in \eqref{doublebilin} using the bilinear Strichartz estimate twice. Note that the loss $\dl>0$ is compensated for by using the fact that there are two frequencies $N_1$ and $N_2 \vee N_3$ which control $N_{23}$. We omit the similar details.
\end{proof}

\subsection{Bounds for $y$}

\begin{lemma}[$\NN_1$ estimate] \label{LEM:N1est}
Let $s\geq s_0>13\dl$ and $0<T\leq 1$. 
Then,
\begin{align}
\begin{split}
\big\|\NN_{1}(f_1,f_2, g  )\big\|_{X^{s,-b'}_{T}}  \les \max\big( &\|f_1\|_{X^{s,b}_T}\|f_2\|_{X^{s_0,b}}\|g\|_{X^{s_0,b}_T},  
\|f_1\|_{X^{s_0,b}_T}\|f_2\|_{X^{s,b}_T}\|g\|_{X^{s_0,b}_T}, \\
&   \|f_1\|_{X^{s_0,b}_T}\|f_2\|_{X^{s_0,b}}\|g\|_{X^{s+11\dl,b}_T} \big),  \label{N1est} 
\end{split}
\end{align}
where $\NN_1$ is as in \eqref{N1}.
\end{lemma}
\begin{proof}
We use Lemma~\ref{LEM:linXsb} to reduce to estimating 
\begin{align*}
\II:=\int_{\R}\int_{0}^{T} \jb{\dx}^{s}\cj{h} \cdot\NN_{1}(f_1,f_2,g)dx dt,
\end{align*}
for $h\in X^{0,b'}_{T}$ satisfying $\|h\|_{X^{0,b'}_{T}}\leq 1$.
We take extensions $F_1,F_2,G$ of $f_1,f_2, g$ on $[0,T]$, respectively. We then further dyadically decompose 
\begin{align}
\II =&\sum_{\substack{N_1,N_2,N_3 \\ N_{23},N_0}} \iint_{\R^2}\jb{\dx}^{s}\cj{\P_{N_0}h_{T}} \cdot\NN_1( \P_{N_1}F_1, \P_{N_2}F_2, \P_{N_3}G )dx dt =:\sum_{\substack{N_1,N_2,N_3 \\ N_{23},N_0}} \II_{\cj{N}}, 
\notag
% \label{decompI}
\end{align}
with $h_T : = \ind_{[0,T]} h$. 
In view of the $\P_{\text{HI}}$ on the derivative in \eqref{N1}, we have $N_{23}\gg 1$. Moreover, $N_{23}\les N_2 \vee N_3$. 
We first consider the cases when
\begin{align}
N_{(2)}\sim N_{(3)}\qquad  \text{or}  \qquad N_0 \leq  N_{(3)}. \label{goodcase}
\end{align}
In these cases, we can apply the argument in \eqref{doublebilin} where we always have at least two large frequencies which are not the output frequency $N_{0}$ to absorb $N_0^s$ and the $N_{23}^{10\dl}$ loss from the bilinear Strichartz used on $h_{T}$ and $\P_{N_1}F_1$.  We omit the similar details.

Thus, we may now assume that:
\begin{align}
N_{(2)} \gg N_{(3)} \quad \text{and} \quad N_0 \geq N_{(2)},  \label{badcase}
\end{align}
and consider different cases depending on which $N_j$ satisfies $N_j \sim N_0$, $j=1,2,3$.

 \smallskip
\noi
\underline{$\bullet$ \textbf{Case 1:} $N_2 \gg N_3 \vee N_1$}

\smallskip
\noi 
In this case, we need the phase function in \eqref{mods}, and recall the notation from the proof of Lemma~\ref{LEM:vwXsb}.
The key point is that, whilst we are in the high$\times$low$\times$$\to$high regime, the phase function is very large. 
Indeed, from the frequency assumption and \eqref{nonres}, we see that $\s_{\max}\ges N_{23}N_0=:K$. 
We will prove that 
\begin{align}
| \II_{\cj{N}}| \les  N_{(1)}^{-\ta} \| h_{T}\|_{X^{0,b'}}    \|F_{2}\|_{X^{s,b}} \|F_{1}\|_{X^{s_0,b}}\|G\|_{X^{s_0,b}}, \label{IIbd}
\end{align}
for some $\ta>0$. The negative power of $N_{(1)}$ allows us to perform all of the dyadic summations. 
We write 
\begin{align*}
\II_{\cj{N}}  =&  \iint_{\R^2} \jb{\dx}^{s}\P_{-,\text{HI}}\cj{\Q_{\ges K}\P_{N_0}h_T} \cdot \P_{N_1}F_1  \cdot\P_{N_{23}}\P_{-,\text{HI}}\dx( \cj{\P_{N_2}F_2} \P_{N_3}G) dt dx \\
&  +  \iint_{\R^2} \jb{\dx}^{s}\P_{-,\text{HI}}\cj{\Q_{\ll K}\P_{N_0}h_T} \cdot \P_{N_1}\Q_{\ges K} F_1  \cdot\P_{N_{23}}\P_{-,\text{HI}}\dx( \cj{\P_{N_2}F_2} \P_{N_3}G) dt dx \\
& +  \iint_{\R^2} \jb{\dx}^{s}\P_{-,\text{HI}}\cj{\Q_{\ll K}\P_{N_0}h_T} \cdot \P_{N_1}\Q_{\ll K} F_1  \cdot\P_{N_{23}}\P_{-,\text{HI}}\dx( \cj{\Q_{\ges K} \P_{N_2}F_2} \P_{N_3}G) dt dx \\
&+   \iint_{\R^2} \jb{\dx}^{s}\P_{-,\text{HI}}\cj{\Q_{\ll K}\P_{N_0}h_T} \cdot \P_{N_1}\Q_{\ll K} F_1  \cdot\P_{N_{23}}\P_{-,\text{HI}}\dx( \cj{\Q_{\ll K} \P_{N_2}F_2} \Q_{\ges K}\P_{N_3}G) dt dx \\
& =: \II_{\cj{N}} ^{0}+\II_{\cj{N}}^{1} + \II_{\cj{N}}^{2}+ \II_{\cj{N}}^{3}.
\end{align*}
Using H\"{o}lder, Bernstein, and the bilinear Strichartz inequality \eqref{bilin1} (with $\dl=0$), we have 
\begin{align*}
|\II_{\cj{N}} ^{0}| & \les N_{0}^{s}N_{23} \|\Q_{\ges K}\P_{N_0}h_{T}\|_{L^{2}_{t,x}} \|\P_{N_1}F_1\|_{L^{\infty}_{t,x}} \|\P_{N_{23}}[  \cj{\P_{N_2}F_2} \cdot \P_{N_3}G]\|_{L^{2}_{t,x}} \\
& \les  N_0^{s} K^{-\frac 12+2\dl} N_{23}^{\frac 12} N_{1}^{\frac 12-s_0} N_2^{-s}N_3^{-s_0}  \|\P_{N_0}h_{T}\|_{X^{0,b'}}\|\P_{N_1}F_1\|_{X^{s_0,b}} \|\P_{N_2}F_2\|_{X^{s,b}} \|\P_{N_3}G\|_{X^{s_0,b}} \\
& \les  N_{(1)}^{-\frac 12+4\dl+\max(\frac 12-s_0,0)} \|\P_{N_0}h_{T}\|_{X^{0,b'}} \|\P_{N_1}F_1\|_{X^{s_0,b}} \|\P_{N_2}F_2\|_{X^{s,b}} \|\P_{N_3}G\|_{X^{s_0,b}},
\end{align*}
which shows \eqref{IIbd} as long as $s_0>4\dl$.

For $\II_{\cj{N}}^{1}$, we argue similarly but first using H\"{o}lder and Bernstein (in $x$), followed by \eqref{YsCTHs},
\begin{align*}
|\II_{\cj{N}} ^{1}| & \les N_{0}^{s}N_{23} \|\Q_{\ll K}\P_{N_0}h_{T}\|_{L^{\infty}_{t}L^{\frac{1}{4\dl}}_{x}} \|\Q_{\ges K}\P_{N_1}F_1\|_{L^{2}_{t}L^{ \frac{2}{1-8\dl}}_{x}  } \|\P_{N_{23}}[  \cj{\P_{N_2}f_2} \P_{N_3}G]\|_{L^{2}_{t,x}} \\
&  \les N_{0}^{\frac 12-4\dl} N_{23}^{\frac 12} N_{1}^{4\dl} N_3^{-s_0} \|\Q_{\ll K}\P_{N_0}h_{T}\|_{L^{\infty}_{t}L^{2}_{x}} \|\Q_{\ges K}\P_{N_1}F_1\|_{L^{2}_{t,x}  }\|\P_{N_2}F_2\|_{X^{s,b}} \|\P_{N_3}G\|_{X^{s_0,b}} \\
& \les N_0^{\frac 12}N_{23}^{\frac 12}  K^{2\dl-b+}   (N_1 N_3)^{-(s_0-4\dl)}   \|\P_{N_0}h_{T}\|_{X^{0,b'}}\|\P_{N_1}F_1\|_{X^{s_0,b}} \|\P_{N_2}F_2\|_{X^{s,b}} \|\P_{N_3}G\|_{X^{s_0,b}}\\
& \les  N_0^{-2\dl+} \|\P_{N_0}h_{T}\|_{X^{0,b'}}\|\P_{N_1}F_1\|_{X^{s_0,b}} \|\P_{N_2}F_2\|_{X^{s,b}} \|\P_{N_3}G\|_{X^{s_0,b}},
\end{align*}
since $s_0>4\dl$,
which establishes \eqref{IIbd}.

For $\II_{\cj{N}}^{2}$, we now use the bilinear Strichartz for the first two factors $h_{T}$ and $F_1$, and argue similarly to $\II_{\cj{N}}^{1}$ for the remaining two factors, placing $\P_{N_2}F_{2}$ into $L^{2}_{t,x}$ and $\P_{N_3}G$ into $L^{\infty}_{t,x}$. We end up with the dyadic factor:
\begin{align*}
N_0^{s} N_{23} N_{23}^{-\frac 12+10\dl} N_1^{-s_0}  K^{-\frac 12-\dl} N_{2}^{-s}N_{3}^{\frac 12-s_0} 
\les  N_{0}^{-\frac 12+8\dl}N_{3}^{\frac 12} (N_1 N_3)^{-s_0} 
\les N_{(3)}^{-\frac 12+\max(\frac 12-s_0,0)+8\dl},
\end{align*}
which is a negative power of $N_{(3)}$ provided that $s_0>8\dl$.

Lastly, for $\II_{\cj{N}}^{3}$, we need to argue more carefully since mimicking the argument for $\II_{\cj{N}}^{1}$ would involve placing $\P_{N_2}F_2$ into $L^{\infty}_{t,x}$ which loses in the high frequency and is too much to recoup. Instead, we distribute the derivative and write 
\begin{align*}
&\II_{\cj{N}}^{3} \\
& = \iint_{\R^2}  \Big\{\P_{N_{23}}\P_{-,\text{HI}}\Big[ \jb{\dx}^{s}\P_{-,\text{HI}}\cj{\Q_{\ll K}\P_{N_0}h_{T}} \cdot \P_{N_1}\Q_{\ll K} F_1 \Big] \cj{\Q_{\ll K}\dx  \P_{N_2}F_2} \Big\} \Q_{\ges K}\P_{N_3}G dt dx \\
& + \iint_{\R^2} \jb{\dx}^{s}\P_{-,\text{HI}}\cj{\Q_{\ll K}\P_{N_0}h_{T}} \cdot \P_{N_1}\Q_{\ll K} F_1  \cdot\P_{N_{23}}\P_{-,\text{HI}}( \cj{\Q_{\ll K} \P_{N_2}F_2} \dx \Q_{\ges K}\P_{N_3}G) dt dx \\
& =:\II_{\cj{N}}^{3,1}+ \II_{\cj{N}}^{3,2}.
\end{align*}
For $\II_{\cj{N}}^{3,2}$, we can use the argument of $\II_{\cj{N}}^{1}$ using the bilinear Strichartz estimate with distant supports \eqref{bilin3} (since $N_0\gg N_1$):
\begin{align*}
|\II_{\cj{N}}^{3,2}|& \les N_{0}^{s} \|\Q_{\ll K}\P_{N_0}h_{T} \cdot \Q_{\ll K}\P_{N_1}F_1\|_{L^{2}_{t,x}} \|\Q_{\ll K}\P_{N_2}F_2\|_{L^{\infty}_{t,x}} \| \dx \Q_{\ges K} \P_{N_3}G\|_{L^{2}_{t,x}} \\
& \les N_{0}^{s} N_{0}^{-\frac 12+10\dl} N_{2}^{\frac 12-s} (N_1 N_3)^{-s_0} N_{3} K^{-\frac 12-\dl} \|\P_{N_0}h_{T}\|_{X^{0,b'}}  \|\P_{N_1}F_1\|_{X^{s_0,b}} \|\P_{N_2}F_2\|_{X^{s,b}}\|\P_{N_3}G\|_{X^{s_0,b}}\\
& \les N_{(1)}^{-1+8\dl+\max(1-s_0,0)}  \|\P_{N_0}h_{T}\|_{X^{0,b'}}  \|\P_{N_1}F_1\|_{X^{s_0,b}} \|\P_{N_2}F_2\|_{X^{s,b}}\|\P_{N_3}G\|_{X^{s_0,b}}\,
\end{align*}
which is acceptable provided that $s_0>8\dl$.

For $\II_{\cj{N}}^{3,2}$, we use Cauchy-Schwarz followed by \eqref{trilin} (with $N_0\gg N_3\vee N_1$) to obtain
\begin{align*}
\II_{\cj{N}}^{3,2} & \les \Big\| \P_{N_3}\Big\{\P_{N_{23}}\P_{-,\text{HI}}\Big[ \jb{\dx}^{s}\P_{-,\text{HI}}\cj{\Q_{\ll K}\P_{N_0}h_T} \cdot \P_{N_1}\Q_{\ll K} F_1 \Big] \cj{\Q_{\ll K}\dx  \P_{N_2}F_2} \Big\} \Big\|_{L^{2}_{t,x}}  \\
&\quad  \times \| \Q_{\ges K} \P_{N_3}G\|_{L^{2}_{t,x}} \\
&\les N_{1}^{15\dl}\bigg( \frac{N_1}{N_0} \bigg)^{\frac 12-10\dl}  N_{2} N_1^{-s_0} K^{-\frac 12-\dl} \| \P_{N_0}h_T\|_{X^{0,b'}} \|\P_{N_1}F_1\|_{X^{s_0,b}}\|\P_{N_2}F_2\|_{X^{s,b}} \|\P_{N_3}G\|_{X^{0,b}} \\
& \les N_{(1)}^{-\frac 12+13\dl+\max(\frac 12-s_0,0)} \| \P_{N_0}h_T\|_{X^{0,b'}} \|\P_{N_1}F_1\|_{X^{s_0,b}}\|\P_{N_2}F_2\|_{X^{s,b}} \|\P_{N_3}G\|_{X^{0,b}}
\end{align*}
which is again acceptable provided that $s_0>13\dl$.

 \medskip
\noi
\underline{$\bullet$ \textbf{Case 2:} $N_2 \les N_3 \vee N_1$}

\smallskip
\noi
Due to \eqref{badcase}, we must have either (i) $N_0\sim N_1\gg N_2 \vee N_3$ or (ii) $N_0\sim N_3 \gg N_2 \vee N_1$. 
If (i) $N_0\sim N_1$ holds, then we can simply apply Cauchy-Schwarz and the bilinear Strichartz estimate \eqref{bilin1} twice with $(N_0,N_1)$ and $(N_2,N_3)$ using $(N_2 \vee N_3)^{-s_0}$ to counteract the slight loss $N_{23}^{10\dl}$ coming from the pair $(N_0,N_1)$. Namely, we do not need to use the phase function in this case. 
We end up with the estimate
\begin{align}
| \II_{\cj{N}}| \les  N_{(3)}^{-\ta} \|\P_{N_0} h_{T}\|_{X^{0,b'}}    \|F_{2}\|_{X^{s,b}} \|\P_{N_1} F_{1}\|_{X^{s_0,b}}\|F_{3}\|_{X^{s_0,b}}, %\label{IIbd}
\end{align}
for some $\ta>0$.  
The negative power of $N_{(3)}$ allows us to perform the dyadic summations over $(N_2,N_3)$ and using Cauchy-Schwarz and that $N_0\sim N_1$, we can sum over $(N_0,N_1)$ since
\begin{align*}
\sum_{N_0\sim N_1} \| \P_{N_0}h_{T}\|_{X^{0,b'}}    \|\P_{N_1}F_{1}\|_{X^{s,b}} \les \| h_{T}\|_{X^{0,b'}}    \|F_{1}\|_{X^{s,b}}.
\end{align*}

If (ii) $N_0 \sim N_3 \gg N_1 \lor N_2$, we proceed as above, applying bilinear twice, to obtain 
\begin{align*}
    | \II_{\cj{N}}| & \les N_0^s N_{23}^{10\dl} \|\P_{N_0} h_{T}\|_{X^{0,b'}}    \|F_{2}\|_{X^{0,b}} \|F_{1}\|_{X^{0,b}}\| \P_{N_3} F_{3}\|_{X^{0,b}}
    \\
    &
    \les N_{(1)}^{-\dl} \| h_{T}\|_{X^{0,b'}}    \|F_{2}\|_{X^{0,b}} \|F_{1}\|_{X^{0,b}}\| F_{3}\|_{X^{s+11\dl ,b}}, 
\end{align*}
where we can use the negative power of the largest frequency to handle the sums. 
 This completes the proof of \eqref{N1est}.
\end{proof}

\begin{lemma}[$\NN_2^{\perp}$ estimate]
\label{LEM:N2perp}
Let $s\geq s_0>13\dl$ and $0<T\leq 1$. 
Then,
\begin{align}
\begin{split}
\big\|\NN_{2}^{\perp}(f_1,f_2, g )\big\|_{X^{s,-b'}_{T}}  \les \max\big( &\|f_1\|_{X^{s,b}_T}\|f_2\|_{X^{s_0,b}_T}\|g\|_{X^{s_0,b}_T},   \\
&  \|f_1\|_{X^{s_0,b}_T}\|f_2\|_{X^{s,b}_T}\|g\|_{X^{s_0,b}_T},  \|f_1\|_{X^{s_0,b}_T}\|f_2\|_{X^{s_0,b}_T}\|g\|_{X^{s,b}_T} \big),  \label{N2est} 
\end{split}
\end{align}
where $\NN_2^\perp$ is as in \eqref{NN2perp}. 
\end{lemma}
\begin{proof}
We adapt much of the argument for $\NN_{1}$ in the proof of Lemma~\ref{LEM:N1est}. By duality and a dyadic decomposition in the last function, we reduce to bounding
\begin{align}
\sum_{\substack{ N_0, N_1,N_2,N_{23},N_3 \\ N_0 \les N_1\vee N_2  }} \iint_{\R^2} \jb{\dx}^{s}\P_{-,\text{HI}}\cj{\P_{N_0}h_{T}} \cdot \P_{N_1}F_1 \cdot  \P_{N_{23}}\P_{-,\text{HI}} \dx( \cj{\P_{N_2}F_2} \cdot \P_{N_3}G) dx dt \label{N1bd1}
\end{align}
where $F_1,F_2,G$ are any extensions of $f_1,f_2,g$, respectively. 

If \eqref{goodcase} is satisfied, then we estimate \eqref{N1bd1} exactly as we did in the proof of \eqref{N1est}, using the bilinear Strichartz estimate \eqref{bilin1} twice on $(\P_{N_0}h_T, \P_{N_1}F_1)$ and on $(\P_{N_2}F_2, \P_{N_3}G)$ and using that we have two high input frequencies which can control the $N_{23}^{10\dl}$ loss from \eqref{bilin1}.

Thus, we may now assume that additionally \eqref{badcase} holds within the support of the dyadic summation in \eqref{N1bd1}. In view of the second condition in \eqref{badcase}, we have two remaining cases. 
 \smallskip
\noi
\underline{$\bullet$ \textbf{Case 1:} $N_0\sim N_1 \gg N_2\vee N_3$}

\smallskip
\noi
In this case, we follow the argument from Case 2 in the proof of Lemma~\ref{LEM:N1est} where we used \eqref{bilin1} twice and the fact that $s_0>10\dl$ to overcome the slight loss $N_{23}^{10\dl}$ from applying \eqref{bilin1} with the pair $(\P_{N_0}h_{T}, \P_{N_1}F_1)$. We omit the similar details.

 \smallskip
\noi
\underline{$\bullet$ \textbf{Case 2:} $N_0\sim N_2 \gg N_1\vee N_3$}

\smallskip
\noi
In this case, we have the large phase gain from $\s_{\max}\ges N_{23}N_0\sim N_{(1)}^{2}$ and we can follow the proof of \eqref{IIbd} from Case 1 in the proof of Lemma~\ref{LEM:N1est} based on gaining extra smoothing from the phase function.

This completes the proof of \eqref{N2est}.
\end{proof}

We now estimate the operator $\NN_2^{\textup{pos}}$ in \eqref{NN2pos}, where we rely on the extra spatial smoothing of $\Y^{\text{pos}}$ to overcome the bad frequency case in the supports of $\NN_2^{\text{pos}}$. 

\begin{lemma}[$\NN^{\textup{pos}}_{2}$ bound]
\label{LEM:N2pos}
Let $6\dl< s_0\leq s<1$, with the additional condition $s + 6\dl < 2s_0 + \frac12$ when $s_0 \le \frac12 \le s < 1$, and $0< T \le1 $. Then, 
\begin{align}
\big\| \NN_{2}^{\textup{pos}}( &f_1,f_2, f_3)\big\|_{X^{s,-b'}_{T}}
\leq  \jb{\be} C_1(\|f_1\|_{X^{s_0,b}_{T}}, \|f_2\|_{X^{s_0,b}_{T}}, \|f_3\|_{X^{0,b}_{T}}) \| f_3 \|_{X^{s,b}_{T}}  
,
\label{Yposest0} \\
\big\| \NN_{2}^{\textup{pos}}( &f_1,f_2, f_3)- \NN_{2}^{\textup{pos}}( f_1,f_2,f_4)\big\|_{X^{s,-b'}_{T}}\notag\\ 
&\leq \jb{\be}^2 C_2(\|f_1\|_{X^{s,b}_{T}}, \|f_{2}\|_{X^{s,b}_{T}}, \|f_3\|_{X^{s,b}_{T}},\|f_4\|_{X^{s,b}_{T}}) \| f_3 -f_4 \|_{X^{s,b}_{T}}  
, 
\label{Yposest} 
\end{align}
where $\NN_2^\textup{pos}$ is as in \eqref{NN2pos} and 
for some non-negative polynomials $C_1: \R^3 \to \R_+$ and $C_2: \R^4 \to \R_+$. 
\end{lemma}
Noticing from \eqref{NN2pos} that $\Y^{\pos}[0]=0$, putting $f_4 \equiv 0$ in \eqref{Yposest} implies a multilinear estimate for simply $\NN_{2}^{\textup{pos}}( f_1,f_2, f_3)$. Given $f_3$, the map $(f_1,f_2)\mapsto \NN_{2}^{\textup{pos}}( f_1,f_2, f_3)$ is bilinear so \eqref{Yposest} already implies difference estimates in these inputs. We still give \eqref{Yposest0} since they will be used later on in a persistence of regularity type argument. 
  
\begin{proof}
As the estimate \eqref{Yposest0} follows from \eqref{Yposest} with $f_4 \equiv0$, we only show the latter. By duality, we have 
\begin{align}
\begin{split}
\text{LHS} \eqref{Yposest} & =  \sup_{\| g\|_{X^{0,b'}_{T} }\leq 1} \bigg|  \sum_{\substack{ N_0, N_1,N_2,N_{23} \\  N_0 \gg N_1\vee N_2  }}\int_{\R}\int_{0}^{T} \jb{\dx}^{s} \P_{-,\text{HI}} \cj{\P_{N_0}g} \cdot \P_{N_1}f_1  \\
&\hphantom{XXXXXXXX} \times \P_{-,\text{HI}}\P_{N_{23}} \dx \big[ \cj{\P_{N_2}f_2} \cdot  ( \Y^{\pos}[f_3]-\Y^{\pos}[f_4])  \big] dx dt  \bigg| \\
& =: \sup_{\| g\|_{X^{0,b'}_{T} }\leq 1} \bigg| \sum_{\substack{ N_0, N_1,N_2,N_{23} \\  N_0\gg N_1\vee N_2  }} \III_{\cj{N}}[g]\bigg|
.
\end{split} 
\label{Yduality}
\end{align}
Note that in view of the summation condition, we may place a wider projection $\wt{\P}_{N_0}$ onto the term $ \Y^{\pos}[f_3]-\Y^{\pos}[f_4]$ for free.
Fix $g\in X^{0,b'}_{T}$ with $\|g\|_{X^{0,b'}_{T}}\leq 1$. Given $0<\dl\ll 1$, let $p_{\dl}=\frac{6+2\dl}{1+9\dl}=6-\frac{52\dl}{1+9\dl}$ such that \eqref{L4} with $b=\frac12 + \frac{\dl}{6}$ gives 
\begin{align}
\|u\|_{L^{p_{\dl}}_{T,x}} \les \|u\|_{X^{0,\frac 12-2\dl}_{T}} 
.
\label{pdl}
\end{align}
We also define $r_{\dl}=6+\frac{156\dl}{3-25\dl}$ which satisfies $\frac{1}{r_{\dl}}+\frac{1}{p_{\dl}}=\frac 13$. We assume that $N_1 \geq N_2$ as the remaining case $N_1<N_2$ can be dealt with in a similar by interchanging the roles of $N_1$ and $N_2$.
Then, by H\"{o}lder's inequality with $\frac{1}{p_{\dl}}+\frac{1}{r_{\dl}}+\frac{1}{\infty}+\frac{2}{3}=1$, \eqref{pdl}, \eqref{L6hi}, Bernstein's inequality, \eqref{YsCTHs}, and \eqref{YPOS2},
\begin{align*}
&|\III_{\cj{N}}[g]| \les N_{23} \| \jb{\dx}^{s}\P_{N_0}g\|_{L^{p_{\dl}}_{T,x}} \|\P_{N_1}f_{1}\|_{L^{r_{\dl}}_{T,x}} \|\P_{N_2}f_2\|_{L^{\infty}_{T,x}} \| \wt{\P}_{N_0}(\Y^{\pos}[f_3]-\Y^{\pos}[f_4])\|_{L^{\frac 32}_{T,x}} \\
&\les N_{0}^{1+s} N_{1}^{-s_0+\frac{13\dl}{3+\dl}} N_{2}^{\frac 12-s_0}N_{0}^{-\frac 32-\min(s,\frac 12)+} \|g\|_{X^{0,b'}_{T}} \|f_1\|_{X^{s_0,b}_{T}} \|f_2\|_{X^{s_0,b}_{T}}C(\|f_3\|_{X^{s,b}_{T}}, \|f_4\|_{X^{s,b}_{T}}) \|f_3-f_4\|_{X^{s,b}_{T}}.
\end{align*}
Considering the dyadic factor here, we have
\begin{align*}
N_{0}^{1+s} N_{1}^{-s_0+\frac{13\dl}{3+\dl}} N_{2}^{\frac 12-s_0}N_{0}^{-\frac 32-\min(s,\frac 12)+} 
&
\les N_0^{-\dl} 
N_1^{-\frac12 +s - s_0 - \min(s, \frac12) + \max(\frac12-s_0, 0) + \frac{13\dl}{3+\dl} + \dl + } \les N_0^{-\dl}
\end{align*}
given that $0<s < 1$ and $s_0 > 6\dl$, with the additional condition that $s + 6\dl < \frac12 + 2s_0$ when $s_0 < \frac12 \le s$. The negative power of the maximum dyadic allows us to sum over all of the dyadics. This proves \eqref{Yposest}, completing the proof.
\end{proof}

\begin{lemma}[Remaining $\NN_2$ bounds]
\label{LEM:N2rem}
Let $s \geq s_0>10\dl$, $s_1\geq s+11\dl$, and $0< T \le 1$. Then, 
\begin{align}
\big\| \NN_{2}(f_1,f_2,e^{i\be F[f_3]}g)\big\|_{X^{s,-b'}_{T}}  & \leq \jb{\be} (1+\|f_3\|_{X^{0,b}_{T}})^{4}  \|f_1\|_{X^{s_0,b}_{T}}\|f_2\|_{X^{s_0,b}_{T}} \|g\|_{X^{s_1,b}_{T}} \label{Ywbd1} \\
\big\| \NN_{2}(f_1,f_2,(e^{i\be F[f_3]}-e^{i\be F[f_4]})&g)\big\|_{X^{s,-b'}_{T}}  
 \leq \jb{\be}  \wt{C}_2(\|f_3\|_{X^{s,b}_{T}}, \|f_4\|_{X^{s,b}_{T}})    \|f_1\|_{X^{s,b}_{T}}\|f_2\|_{X^{s,b}_{T}}  \notag \\
&\qquad \times  \|g\|_{X^{s_1,b}_{T}} \|f_3-f_4\|_{X^{s,b}_{T}}, 
\label{Ywbd2}
\end{align}
where $\NN_2$ is as in \eqref{NN2}. 
\end{lemma}
\begin{proof}
We first consider \eqref{Ywbd1}.
As in \eqref{Yduality}, we reduce to estimating
\begin{align*}
\IV_{\cj{N}}:=\int_{\R}\int_{0}^{T} \jb{\dx}^{s} \P_{-,\text{HI}} \cj{\P_{N_0}k} \cdot \P_{N_1}f_1 \cdot \P_{-,\text{HI}}\P_{N_{23}} \dx \big[ \cj{\P_{N_2}f_2} \cdot (  e^{i\be F[f_3]}g  )  \big] dx dt  
\end{align*}
with $N_0 \gg N_1\vee N_2$, and $k\in X^{0,b'}_{T}$ with $\|k\|_{X^{0,b'}_{T}}\leq 1$. 
We further decompose $\IV_{\cj{N}}=\IV_{\cj{N}}^{(1)} +\IV_{\cj{N}}^{(2)}+\IV_{\cj{N}}^{(3)}$, where $e^{i\be F[f_3]}$ in $\IV_{\cj{N}}^{(1)}$ is replaced by $\Pblo e^{i\be F[f_3]}$, in $\IV_{\cj{N}}^{(2)}$, it is replaced by $\Pbhi \P_{\ll N_0}e^{i\be F[f_3]} $, and in $\IV_{\cj{N}}^{(3)}$ it is replaced by $\P_{\ges N_{0}}e^{i\be F[f_3]} $. 
We consider first the contribution $\IV_{\cj{N}}^{(1)}$. By impossible frequency interactions, we may place a widened projector $\wt{\P}_{N_0}$ onto $g$ for free. Then, by the bilinear Strichartz estimate \eqref{bilin1} and \eqref{DeLinfty}, we have
\begin{align*}
|\IV_{\cj{N}}^{(1)}| & \les N_{23}^{\frac 12+ 10\dl} N_{0}^{s} \|\P_{N_0}k\|_{X^{0,b'}_{T}} \|\P_{N_1}f_1\|_{X^{0,b}_{T}} \|\Pblo e^{i\be F[f_3]}\|_{L^{\infty}_{T,x}} \| \cj{\P_{N_2}f_2}  \cdot \wt{\P}_{N_0}g\|_{L^2_{T,x}}\\
& \les N_{0}^{s+10\dl-s_1}\|k\|_{X^{0,b'}_{T}} \|f_1\|_{X^{0,b}_{T}}  \|f_2\|_{X^{0,b}_{T}} \|g\|_{X^{s_1,b}_{T}}.
\end{align*}
Since $s_1\geq s+10\dl+\dl$, the dyadic factor is bounded by $N_{0}^{-\dl}$, which is a negative power of the maximum frequency allowing to sum over all of the dyadics. The same estimates apply to  $\IV_{\cj{N}}^{(2)}$ since $\| \Pbhi \P_{\ll N_0}e^{i\be F[f_3]}\|_{L^{\infty}_{T,x}}\les 1$. We now consider the remaining term $\IV_{\cj{N}}^{(3)}$. By H\"{o}lder's inequality, \eqref{bilin1}, \eqref{L2pos2}, Bernstein's inequality, \eqref{YsCTHs}, and \eqref{L4}, we find
\begin{align*}
|\IV_{\cj{N}}^{(3)}|& \les N_0^{s} N_{23}^{\frac 12 +10\dl} \|\P_{N_0}k\|_{X^{0,b'}_{T}}\|\P_{N_1}f_1\|_{X^{0,b}_{T}} \|\P_{\ges N_0} e^{i\be F[f_3]}\|_{L^{3}_{T,x}} \|\P_{N_2}f_2\|_{L^{\infty}_{T,x}} \|g\|_{L^{6}_{T,x}} \\
& \les N_{0}^{-\frac 12+10\dl} N_{2}^{\frac 12-s_0}  \| k\|_{X^{0,b'}_{T}}  (1+\|f_3\|_{X^{0,b}_{T}})^2 \|f_3\|_{X^{s,b}_{T}} \|f_1\|_{X^{0,b}_{T}}\|f_2\|_{X^{s_0,b}_{T}} \|g\|_{X^{0,b}_{T}}.
\end{align*}
Now since $N_{0}^{-\frac 12+10\dl} N_{2}^{\frac 12-s_0}\les N_{0}^{-\frac12 + 10 \dl + \max(\frac12 - s_0,0)}$ and $s_0>10\dl$, we have a negative power of the maximum dyadic which can be used to sum over all of the dyadics.

The estimate for \eqref{Ywbd2} is similar, instead using \eqref{DeLinfty2} and \eqref{L2posdiff}. We omit the details.
\end{proof}

\subsection{The lower order terms}

\begin{lemma} \label{LEM:cubics}
Let $2\dl < s_0 \le s \le 1$, $0\leq \g<\min(s_0,\frac 12)$, and $0< T \le 1$. Then,
\begin{align}
\| &e^{i\be F[g_1]}f_1 \QQ_{h}(\cj{f_2}f_3)\|_{X^{s+\g,-b'}_{T}}  \les \jb{\be} \big( \max_{p\in \{2,8\}}\|\QQ_{h}\|_{L^p\to L^p} \big)  C_1(\|g_1\|_{X^{0,b}_{T}}) \notag \\
&\qquad \times   \Big( \prod_{j=1}^{3} \|f_j\|_{X^{s_0,b}_{T}} \|g_1\|_{X^{s,b}_{T}}+\max_{\s\in S_3}   \|f_{\s(1)}\|_{X^{s,b}_{T}} \|f_{\s(2)}\|_{X^{s_0,b}_{T}} \|f_{\s(3)}\|_{X^{s_0,b}_{T}}\Big), 
\label{Gh0}\\
\| (&e^{i\be F[g_1]}-  e^{i\be F[g_2]} ) f_1 \QQ_{h}(\cj{f_2}f_3)\|_{X^{s+\g,-b'}_{T}} \notag  \\
& \les \jb{\be}^2 \big( \max_{p\in \{2,8\}}\|\QQ_{h}\|_{L^p\to L^p} \big)  C_{2}(\|g_1\|_{X^{s,b}_{T}}, \|g_2\|_{X^{s,b}_{T}}) \bigg(\prod_{j=1}^{3} \|f_j\|_{X^{s,b}_{T}} \bigg) \|g_1-g_2\|_{X^{s,b}_{T}}, 
 \label{Gh}
\end{align}
where $\QQ_h$ is as in \eqref{Qh}, for some non-negative polynomials $C_1:\R \to \R_+$ and $C_2: \R^2 \to \R_+$.
\end{lemma}
\begin{remark}\rm \label{RMK:Qbd}
Notice that by setting $g_1 =g_2 \equiv 0$ in \eqref{Gh}, we obtain an estimate for $f_1 \QQ_{h}(\cj{f_2}f_3)$ without any exponential factor. Similarly, by just setting $g_2\equiv 0$, we get an estimate with a single exponential factor $ e^{i\be F[g_1]} f_1 \QQ_{h}(\cj{f_2}f_3)$.
\end{remark}
\begin{proof}

To ease the notation, we set $E: = e^{i\be F[g_1]}-  e^{i\be F[g_2]}$. To show \eqref{Gh}, by duality, we reduce to bounding
\begin{align}
\int_{0}^{T} \int_{\R} \cj{\jb{\dx}^{s+\g} k}  \cdot E\cdot f_1 \QQ_{h}( \cj{f_2}f_3) dx dt \label{Gh1}
\end{align}
for $k\in X^{0,b'}_{T}$ with $\|k\|_{X^{0,b'}_{T}}\leq 1$. We then dyadically decompose
\begin{align*}
\eqref{Gh1} 
& = \sum_{N_0,N_1,N_2,N_3} \int_{0}^{T} \int_{\R} \cj{\jb{\dx}^{s+\g} \P_{N_0} k} \cdot   E \cdot \P_{N_1}f_1 \QQ_{h}( \cj{\P_{N_2}f_2} \P_{N_3}f_3) dxdt \\
& = \sum_{N_0,N_1,N_2,N_3} \big( \I_{\cj{N}}^{(1)}+\I_{\cj{N}}^{(2)}+\I_{\cj{N}}^{(3)}\big),
\end{align*}
where in $\I_{\cj{N}}^{(1)}$, $E$ is replaced by $\Pblo E$, in $\I_{\cj{N}}^{(2)}$ it is replaced by $\Pbhi \P_{\ll N_0}E$ and in $\I_{\cj{N}}^{(3)}$, it is replaced by $\P_{\ges N_0}E$.

Consider first the contribution from $\I_{\cj{N}}^{(1)}$. By impossible frequency interactions, we see that $N_{(1)}\sim N_{(2)}$ and $N_0 \les \max(N_1,N_2,N_3)$. If $N_0 \leq N_{(3)}$, then by \eqref{L4} we have
\begin{align*}
|\I_{\cj{N}}^{(1)}| &\leq \|\QQ_{h}\|_{\text{op}} \|E\|_{L^{\infty}_{T,x}} N_0^{s+\g}\|k\|_{L^{4}_{T,x}}\|\P_{N_1}f_1\|_{L^{4}_{T,x}} \|\P_{N_2}f_2\|_{L^{4}_{T,x}}  \|\P_{N_3}f_3\|_{L^{4}_{T,x}}  \\
& \les \|\QQ_{h}\|_{\text{op}} \|E\|_{L^{\infty}_{T,x}}  N_{(1)}^{\g-s_0} \|k\|_{X^{0,b'}_{T}} \max_{\s\in S_{3}} \|f_{\s(1)}\|_{X^{s,b}_{T}}\|f_{\s(2)}\|_{X^{s_0,b}_{T}}\|f_{\s(3)}\|_{X^{0,b}_{T}},
\end{align*}
where we obtain a negative power of $N_{(1)}$ as long as $\g<s_0$. We may now assume that $N_{0}\geq N_{(2)}$. If $N_{(3)} \ges N_{(2)}$, then the same argument as above applies as we have two functions with high frequency input.  Thus, it remains to consider when $N_{(3)}\ll N_{(2)}$. We have two further cases. First, if $N_0 \sim N_1$, then by Cauchy-Schwarz, H\"{o}lder's inequality with $\frac{1}{p_{\dl}}+\frac{1}{q_{\dl}}=\frac 12$ and $q_{\dl}:=\frac{3+\dl}{1-4\dl}$, \eqref{localsmooth0}, \eqref{L4}, and \eqref{L6hi}, we have 
\begin{align*}
|\I_{\cj{N}}^{(1)}| &\leq  \|E\|_{L^{\infty}_{T,x}} \|\P_{N_1}f_1\|_{L^{2}_x L^{\infty}_{T}} \| \jb{\dx}^{s+\g}\P_{N_0}k \cdot \QQ_h( \cj{\P_{N_2}f_2} \P_{N_3}f_3)\|_{L^2_x L^{1}_{T}} \\
& \les  \|E\|_{L^{\infty}_{T,x}} N_{1}^{-\frac 12+\g}\|f_1\|_{X^{s,b}_{T}} \|\P_{N_0}k\|_{L^{p_{\dl}}_{T,x}} \|\QQ_{h}( \cj{\P_{N_2}f_2} \P_{N_3}f_3)\|_{L^{q_{\dl}}_{T,x}} \\
&\les \|\QQ_{h}\|_{L^{q_{\dl}}\to L^{q_{\dl}}} \|E\|_{L^{\infty}_{T,x}} N_{1}^{-\frac 12+\g}\|f_1\|_{X^{s,b}_{T}} \|k\|_{X^{0,b'}_{T}} \|\P_{N_2}f_2\|_{L^{2q_{\dl}}_{T,x}}\|\P_{N_3}f_3\|_{L^{2q_{\dl}}_{T,x}} \\
& \les \|\QQ_{h}\|_{L^{q_{\dl}}\to L^{q_{\dl}}} \|E\|_{L^{\infty}_{T,x}} N_{1}^{-\frac 12+\g+\frac{13\dl}{3+\dl} }\|f_1\|_{X^{s,b}_{T}} \|f_2\|_{X^{0,b}_{T}} \|f_3\|_{X^{0,b}_{T}},
\end{align*}
which is a negative power of $N_{(1)}$ provided that $\g< \frac{1}{2}-\frac{13\dl}{3+\dl}$.
Now, we consider the case $N_0 \sim N_{2}\vee N_3$. Without loss of generality, we assume that $N_0 \sim N_2 \gg N_1 \vee N_3$. In this case by \eqref{L4} and \eqref{bilin1}, we have 
\begin{align*}
|\I_{\cj{N}}^{(1)}| & \les  N_0^{s+\g} \|E\|_{L^{\infty}_{T,x}}  \|\P_{N_0}k\|_{L^4_{T,x}}\|\P_{N_1}f_1\|_{L^{4}_{T,x}} \| \QQ_{h}( \cj{\P_{N_2}f_2} \P_{N_3}f_3)\|_{L^2_{T,x}} \\
& \les \|\QQ_{h}\|_{\text{op}} N_{0}^{-\frac 12+\g} \|k\|_{X^{0,b'}_{T}} \|f_1\|_{X^{0,b}_{T}} \|f_2\|_{X^{s,b}_{T}} \|f_3\|_{X^{0,b}_{T}},
\end{align*}
which is acceptable provided that $\g<\frac 12$. This completes the estimates for $\I_{\cj{N}}^{(1)}$.

For $\I_{\cj{N}}^{(2)}$, we follow the same ideas as we did for $\I_{\cj{N}}^{(1)}$ as we still have the conditions $N_0 \les \max(N_1,N_2,N_3)$ and $N_{(1)}\sim N_{(2)}$ and we always place 
$\Pbhi \P_{\ll N_0}E$ into $L^{\infty}_{T,x}$.

Lastly, we consider $\I_{\cj{N}}^{(3)}$. Let $\wt{r}_{\dl}:=\frac{6+2\dl}{1-\frac{23}{9}\dl}$ be such that $\frac{1}{p_{\dl}}+\frac{3}{\wt{r}_{\dl}}+\frac 13 = 1$. Then, by H\"{o}lder's inequality, \eqref{L2posdiff}, \eqref{L4} and \eqref{L6hi}, we have 
\begin{align*}
|\I_{\cj{N}}^{(3)}| & \les N_{0}^{s+\g} \| \QQ_h\|_{L^{ \frac{ \wt{r}_{\dl} }{2}   }   \to L^{ \frac{ \wt{r}_{\dl} }{2}}} \|\P_{\ges N_0}E\|_{L^{3}_{T,x}} \|\P_{N_0} k \|_{L^{p_{\dl}}_{T,x}} \|\P_{N_1}f_1\|_{L^{\wt{r}_{\dl}}_{T,x}}   \|\P_{N_2}f_2\|_{L^{\wt{r}_{\dl}}_{T,x}}  \|\P_{N_3}f_3\|_{L^{\wt{r}_{\dl}}_{T,x}}  \\
&\les N_{0}^{-1+\g} \| \QQ_h\|_{L^{ \frac{ \wt{r}_{\dl} }{2}   }   \to L^{ \frac{ \wt{r}_{\dl} }{2}}} \|k\|_{X^{0,b'}_{T}} C_{2}(\|g_1\|_{X^{s,b}_{T}}, \|g_2\|_{X^{s,b}_{T}}) \|g_1-g_2\|_{X^{s,b}_{T}} \\
& \hphantom{XXXX} \times (N_1 N_2 N_3)^{\frac{13\dl}{9+3\dl}-s_0} \|f_1\|_{X^{s_0,b}_{T}}\|f_2\|_{X^{s_0,b}_{T}} \|f_3\|_{X^{s_0,b}_{T}}.
\end{align*}
Now this bound is acceptable provided that $\g\leq 1$ and $s_0>\frac{13\dl}{9+3\dl}$.  

By examining the cases, we see that we need control on $\|\QQ_{h}\|_{L^p\to L^p}$ when $p\in \{2,p_{\dl}, \frac{\wt{r}_{\dl}}{2}\}$. To control these operator norms, it suffices, by interpolation, to control the $L^p\to L^p$ operator norms only for $p\in \{2,8\}$, with $8$ chosen to be large enough and independent of $s$ and $\dl$. This completes the proof of \eqref{Gh}.

The estimate for \eqref{Gh0} follows from the strategy above, with the only change being in the estimate of $\I^{(3)}_{\cj{N}}$, where we use \eqref{L2pos2} instead of \eqref{L2posdiff} for the $L^3_{T,x}$-norm of $e^{i \be F[g_1]}$. We omit details.
\end{proof}

\begin{lemma}
\label{LEM:N0}
Let $s\geq s_0>10\dl$, $0\leq \g<s$, and $0< T \le 1$. Then, it holds that
\begin{align}
\| \NN_{0}(f_1,f_2,f_3)\|_{X^{s+\g,-b'}_{T}} & \les  \max_{\s \in S_3}\|f_{\s(1)}\|_{X^{s,b}_{T}} \|f_{\s(2)}\|_{X^{s_0,b}_{T}}\|f_{\s(3)}\|_{X^{s_0,b}_{T}},  \label{N0est}
\end{align}
where $\NN_0$ is as in \eqref{N0}.
\end{lemma}
\begin{proof}
Note that the operator $\PbHI-\Pbhi$ is a frequency projection to frequencies $\{|\xi|\sim 1\}$.
From \eqref{N0}, we see that we have two kinds of contributions to $\NN_0$: (i) one with an outer factor of $\PbHI-\Pbhi$, and (ii) one with a factor $\Id - \P_\HI = \P_\LO$ on the derivative. For contribution (i), we use  $\PbHI-\Pbhi$ to control the outer derivatives $s+\g$ and then can apply the bilinear Strichartz estimate \eqref{bilin1} twice, similar to as we did in \eqref{doublebilin}, which weakens the derivative to $N_{23}^{10\dl}$ and which can be absorbed by $N_2\vee N_3$ and $s>10\dl$. For the latter contribution (ii), the projection $\P_\LO$ completely controls the derivative and we can simply apply \eqref{Gh} with Remark~\ref{RMK:Qbd}. 
\end{proof}

\section{Local well-posedness}\label{SEC:LWP}

We consider now the system of equations for $(v,y,z,w)$, which we write in the Duhamel formulation as:
\begin{align}
\begin{split}
 v(t) & = S(t) v_0 + \mathcal{I}\big[ \mathfrak{N}_{v}[v,v+y+z]\big] ,\\
  y(t) & = S(t) y_0 + \mathcal{I}\big[ \mathfrak{N}_{y}[v,y,z,w]\big] , \\
  z(t) & =S(t) z_0 + \mathcal{I}\big[ \mathfrak{N}_{z}[v+y+z]\big], \\
  w(t) & = S(t)w_0 +\mathcal{I}\big[ \mathfrak{N}_{w}[w,v+y+z]\big],
  \end{split} \label{system}
\end{align}
where $\mathfrak{N}_f$ is as in \eqref{veq}, \eqref{Yeq2}, \eqref{Zeq}, and \eqref{weq} for $f\in\{v,y,z,w\}$, and 
\begin{align}
\mathcal{I}[G](t) : = \int_{0}^{t} S(t-t')G(t')dt'.
\notag
\end{align}

\begin{proposition}\label{PROP:system}
Let $0<s <1$, 
$\dl>0$ be sufficiently small so that the results of Section~\ref{SEC:ests} all hold true and $50\dl \leq s$, and set $\s:=s-11\dl$.
Then, the system in \eqref{system} for $(v,y,z,w)$ is locally well posed in 
\begin{align*}
\mathfrak{H}^{s,\s}(\R): = H^{s}_{+}(\R)\times H^{\s}_{-}(\R)\times H^{s}(\R)\times H^{s}_{-}(\R)
.
\end{align*}
More precisely, given $(v_0,y_0, z_0,w_0)\in \mathfrak{H}^{s,\s}(\R)$, and
\begin{align}\notag
T=T(\|(v_0,y_0,z_0,w_0)\|_{\mathfrak{H}^{s,\s}(\R)}) \sim (1+ \|(v_0,y_0,z_0,w_0)\|_{\mathfrak{H}^{s,\s}(\R)})^{-A},
\end{align}
for some $A=A(\dl,s)\geq 1$,
there is a unique solution 
\begin{align*}
(v,y,z,w)\in \mathfrak{X}^{s,\s}_{T}:= X^{s,\frac 12+\dl}_{T}\times X^{\s,\frac 12+\dl}_{T}\times X^{s,\frac 12+\dl}_{T}\times X^{s,\frac{1}{2}+\dl}_{T} \embeds C([-T,T];\Hf^{s,\s}(\R))
\end{align*}
 to the system \eqref{system} such that $(v,y,z,w)\vert_{t=0}=(v_0,y_0,w_0,z_0)$ and for which
 the data-to-solution map $\Phi: (v_0,y_0,w_0,z_0)\mapsto (v,y,z,w)$ is locally Lipschitz continuous. Moreover, we have the following nonlinear smoothing properties, for any $s_0>0$:
 \begin{align}
 \begin{split}
\| v- S(t)v_0\|_{C_{T}H^{s+\dl}_x}+\| z- S(t)z_0\|_{C_{T}H^{s_0}_x}
+& \| w- S(t)w_0\|_{C_{T}H^{s+\dl}_x} \\
&  \leq C(T, s_0, \|(v_0,y_0,z_0,w_0)\|_{\Hf^{s,\s}(\R)})
.
\end{split} \label{nonlinsmooth}
\end{align}
 \end{proposition}

 \begin{proof}
 To show that the system \eqref{system} is locally well-posed, we run a contraction mapping argument in a ball in $\mathfrak{X}^{s,\s}_{T}$, combining the nonlinear estimates in Lemmas~\ref{LEM:vwXsb}, \ref{LEM:Zbd}, \ref{LEM:N1est}, \ref{LEM:N2perp}, \ref{LEM:N2pos}, \ref{LEM:N2rem}, \ref{LEM:cubics}, and \ref{LEM:N0}, 
 and we always use \eqref{timeloc} to gain a factor of $T^{\dl}$ in the nonlinear estimates.
 Notice that the uniqueness statement here is for functions in the given ball in $\mathfrak{X}^{s,\s}_{T}$. However, a standard argument then extends this uniqueness to the entire space $\mathfrak{X}^{s,\s}_{T}$.
 The nonlinear smoothing follows from the nonlinear smoothing estimates for $(v,z,w)$ in Lemma~\ref{LEM:vwXsb}, Lemma~\ref{LEM:Zbd}, Lemma~\ref{LEM:cubics}, and Remark~\ref{RMK:Qbd}. Note that the nonlinear smoothing for $z$ is trivial since the nonlinear term is compactly supported in frequency space.
 \end{proof}

\subsection{Proof of Theorem~\ref{THM:LWP}}

Fix $(\dl,s,\s)$ as in Proposition~\ref{PROP:system} and such that $s\leq \frac 12$, and let $u_0\in H^{s}(\R)$. Given $j\in \N$, we let $u_{0,j}:= \P_{\leq j}u_0$, which is a smooth projection onto frequencies $\{|\xi|\leq j\}$.  By the local well-posedness result for \eqref{INLS} in \cite{MP}, there exists  a maximal time of existence $T_{\max}(j)>0$ and a unique classical solution $u_{j}\in C([0,T_{\max}(j));H^{\infty}(\R))$ to \eqref{INLS} such that $u_{j}(0)=u_{0,j}$. 
Moreover, $T_{\max}(j)$ satisfies a blow-up criterion in the sense that if $T_{\max}(j)<\infty$, then 
\begin{align}
\lim_{t\to T_{\max}(j)}\|u_{j}(t)\|_{H^{\frac 12+9\dl}(\R)} =+\infty. \label{blowup}
\end{align} 
The first step is to show that $T_{\max}(j)$ can be bounded below uniformly in $j\in \N$. Note that this is necessary since we do not, in the focusing case, have smooth global solutions for large data.

For $j\in \N$ and for the smooth solutions $u_j$ as above on $[0,T_{\max}(j))$, we define 
\begin{align}
\begin{split}
\wt{v}_{j}(t) : = \Pbhip[ e^{i\be F[u_j(t)]}u_j(t)], &\quad  \wt{y}_{j}(t) : = \P_{-,\text{hi}}[ e^{i\be F[u_j(t)]}u_j(t)], \\
 \wt{z}_{j}(t):= \Pblo[ e^{i\be F[u_j(t)]}u_j(t)],   &\quad \wt{w}_{j}(t):=\P_{-,\text{hi}}u_j(t).
\end{split} 
\label{vjs}
\end{align}
so that 
\begin{align}
u_{j} = e^{-i\be F[\wt{v}_j+\wt{y}_j+\wt{z}_j]}(\wt{v}_j+\wt{y}_j+\wt{z}_j).
\label{ujformula}
\end{align}
 Note that $(\wt{v}_j,\wt{y}_j,\wt{z}_j,\wt{w}_j)$ are smooth functions and thus satisfy the equations \eqref{veq}, \eqref{Yeq}, \eqref{Zeq}, and \eqref{weq}, respectively.
Moreover, $\wt{y}_j$ also satisfies 
\begin{align}
\wt{y}_{j} = \Y^{\pos}[ \wt{v}_j+\wt{y}_j+\wt{z}_j] + \Y^{\negg}[\wt{v}_j+\wt{y}_j+\wt{z}_j, \wt{w}_j] +e^{i\be F[\wt{v}_j+\wt{y}_j+\wt{z}_j]}\wt{w}_{j}
\label{yjdecomp}
\end{align}
and the equation \eqref{Yeq2}.

Meanwhile, for $j\in \N\cup\{\infty\}$, we define $v_{0,j}=\Pbhip[ e^{i\be F[u_{0,j}]}u_{0,j}]$, $y_{0,j}:= \P_{-,\text{hi}}[ e^{i\be F[u_{0,j}]}u_{0,j}]$, $z_{0,j}:= \Pblo[ e^{i\be F[u_{0,j}]}u_{0,j}]$, and $w_{0,j}:= \P_{-,\text{hi}}u_{0,j}$. We use the shorthand $v_{0,\infty}:=v_{0}$ and similarly for $y_{0,\infty}$, $z_{0,\infty}$, and $w_{0,\infty}$.
 It follows from Lemma~\ref{LEM:JsL2} that
\begin{align}
\sup_{j\in \N\cup \{\infty\}}\| (v_{0,j}, y_{0,j}, z_{0,j}, w_{0,j})\|_{\Hf^{s,\s}} \les (1+\|u_0\|_{L^2})^{2}\|u_0\|_{H^{s}}. \label{unifdata}
\end{align}
By applying Proposition~\ref{PROP:system} and \eqref{unifdata}, we find that there is $T_{\ast}\sim (1+\|u_0\|_{H^{s}})^{-A}$, uniform in $j\in \N$, and a
 unique solution $(v_j,y_j,z_j,w_j)\in \mathfrak{X}^{s,\s}_{T_{\ast}}$ to \eqref{system} satisfying $(v_j,y_j,z_j,w_j)\vert_{t=0}= (v_{0,j}, y_{0,j}, z_{0,j}, w_{0,j})$ and is such that
\begin{align}
\| (v_j,y_j,z_j,w_j)\|_{\mathfrak{X}^{s,\s}_{T_{\ast}}}  & \les (1+\|u_0\|_{L^2})^{2} \|u_0\|_{H^{s}} ,
\label{LWPbd}\\
\| (v_j,y_j,z_j,w_j)- (v_{\infty},y_{\infty},z_{\infty},w_{\infty})\|_{\mathfrak{X}^{s,\s}_{T_{\ast}}} &\les \| (v_{0,j}, y_{0,j}, z_{0,j}, w_{0,j}) - (v_{0}, y_{0}, z_{0}, w_{0})\|_{\Hf^{s,\s}}, \label{Lipconv}
\end{align}
for any $j\in \N\cup\{\infty\}$,
and for which
 \begin{align}
 \begin{split}
\sup_{j\in \N\cup\{\infty\}} \Big( \| v_j- S(t)v_{0,j}\|_{C_{T_{\ast}}H^{s+\dl}_x}+\| z_{j}- S(t)&z_{0,j}\|_{C_{T_{\ast}}H^{s_0}_x}
+ \| w_{j}- S(t)w_{0,j}\|_{C_{T_{\ast}}H^{s+\dl}_x} \Big) \\
&  \leq C(T_{\ast}, s_0, \|u_0\|_{H^s}),
\end{split} %\label{nonlinsmooth}
\end{align}
for any $s_0>0$.

We now wish to show that in fact 
\begin{align}
(v_{j},y_j, z_j,w_j)= (\wt{v}_{j},\wt{y}_j, \wt{z}_j,\wt{w}_j) \quad \text{for}\quad j\in \N. \label{vequal}
\end{align}
By the uniqueness statement in $\mathfrak{X}^{s,\s}_{T}$ in Proposition~\ref{PROP:system}, \eqref{vequal} follows once we show that $(\wt{v}_{j},\wt{y}_j, \wt{z}_j,\wt{w}_j)\in \mathfrak{X}^{s,\s}_{T_{\max}(j)}$, which is the content of the next lemma.

\begin{lemma}\label{LEM:Xsbreg}
Let $s_1\geq s$ where $s$ is as in Proposition~\ref{PROP:system} and let $(\wt{v}_j,\wt{y}_j,\wt{z}_j,\wt{w}_j)$ be as defined in \eqref{vjs}. Then, it holds that 
\begin{align}
\| (\wt{v}_j,\wt{y}_j,\wt{z}_j,\wt{w}_j)\|_{\mathfrak{X}^{s_1,s_1}_{T_{\max}(j)}} \leq C(\|u_j\|_{L^{\infty}_{T_{\max}(j)}H^{s_1+2}_x}) <\infty. \label{Xsbvj}
\end{align}
\end{lemma}
\noi
Note that the RHS of \eqref{Xsbvj} does not need to be uniform in $j$ in order to conclude \eqref{vequal}.

\begin{proof}
Since $(\wt{v}_j,\wt{y}_j,\wt{z}_j,\wt{w}_j)$ satisfy \eqref{veq}, \eqref{Yeq}, \eqref{Zeq}, and \eqref{weq}, by following the argument in \cite[Proposition 3.2]{MP}, we see that for any $0<T\leq T_{\max}(j)$, it holds that
\begin{align}
\begin{split}
\| \wt{v}_j\|_{X^{s_1,1}_{T}}& \les \|\wt{v}_j\|_{L^{\infty}_{T}H^{s_1}_x} + \|\mathfrak{N}_{v}[\wt{v}_j, \wt{v}_j+\wt{y}_j+\wt{z}_j]\|_{L^{2}_{T}H^{s_1}_x}, \\
\| \wt{y}_j\|_{X^{s_1,1}_{T}} &\les \|\wt{y}_j\|_{L^{\infty}_{T}H^{s_1}_x} + \|\mathfrak{N}^{(0)}_{y}[\wt{v}_j+\wt{y}_j+\wt{z}_j]\|_{L^{2}_{T}H^{s_1}_x},  \\
\| \wt{z}_j\|_{X^{s_1,1}_{T}}& \les \|\wt{v}_j\|_{L^{\infty}_{T}H^{s_1}_x} + \|\mathfrak{N}_{z}[\wt{v}_j+\wt{y}_j+\wt{z}_j]\|_{L^{2}_{T}H^{s_1}_x}, \\
\| \wt{w}_j\|_{X^{s_1,1}_{T}}& \les \|\wt{w}_j\|_{L^{\infty}_{T}H^{s_1}_x} + \|\mathfrak{N}_{w}[\wt{w}_j, \wt{v}_j+\wt{y}_j+\wt{z}_j]\|_{L^{2}_{T}H^{s_1}_x}.
\end{split} \label{Xs1bds}
\end{align} 
Now the bounds for the right-hand sides in \eqref{Xs1bds} follow simply from the fractional Leibniz rule and using the fact that $(\wt{v}_j,\wt{y}_j,\wt{z}_j,\wt{w}_j)$ are all sufficiently smooth functions at least belonging to $L^{\infty}_{T}H^{s_1+2}_x$, which is sufficient to absorb the derivative losses in $\mathfrak{N}_{v}$, $\mathfrak{N}_{y}^{(0)}$, $\mathfrak{N}_{z}$, and $\mathfrak{N}_{w}$.
We then obtain \eqref{Xsbvj} from the embedding $X^{s_1,1}_{T}\embeds X^{s_1,b}_{T}$.
\end{proof}

It now follows from \eqref{Xsbvj}, \eqref{ujformula}, and \eqref{yjdecomp} that 
\begin{align}
u_{j} &= e^{-i\be F[v_j+y_j+z_j]}(v_j+y_j+z_j), \label{ujformula2}\\
y_{j}& =\Y^{\pos}[v_j+y_j+z_j] + \Y^{\negg}[ v_j+y_j+z_j, w_j]+ e^{i\be F[v_j+y_j+z_j]}w_{j}. \label{yjdecomp2}
\end{align}

At this stage, the equality \eqref{vequal} only holds for the minimum amount of time $T_{\max}(j) \wedge T_{\ast}$,  still depending upon $j$. 
We now establish a bound on $u_j$ in $L^{\infty}_{T}H_x^{\frac 12+}$ for a short time which is uniform in $j$, at which point, \eqref{blowup} then allows us to show that $T_{\max}(j)$ is lower bounded by some strictly positive number which is independent of $j$; see \eqref{Tlowerbd} below.
First, we note that using the Duhamel formulation \eqref{system}, Lemma~\ref{LEM:linXsb}, and the estimates \eqref{LWPbd}, \eqref{vtrilin}, \eqref{ztrilin}, \eqref{N1est}, \eqref{N2est}, \eqref{Yposest0}, \eqref{Ywbd1}, \eqref{Gh0}, \eqref{N0est}, we have
\begin{align*}
\| (v_j,y_j,z_j,w_j)\|_{\mathfrak{X}^{\frac 12+20\dl,\frac 12+9\dl}_{T}} 
&\les \|(v_{0,j},y_{0,j},z_{0,j},w_{0,j})\|_{\Hf^{\frac 12+20\dl,\frac 12+9\dl}}  \\
&\qquad +C( \|u_0\|_{H^{s}}) T^{\dl} \| (v_j,y_j,z_j,w_j)\|_{\mathfrak{X}^{\frac 12+20\dl,\frac 12+9\dl}_{T}},
\end{align*} 
and thus by choosing $T_{\ast \ast}=T_{\ast \ast}(\|u_{0}\|_{H^s})>0$ 
sufficiently small it follows from \eqref{Xsbvj} that 
\begin{align*}
\| (v_j,y_j,z_j,w_j)\|_{\mathfrak{X}^{\frac 12+20\dl,\frac 12+9\dl}_{T_{\ast} \wedge T_{\ast \ast} \wedge T_{\max}(j) }} \les  \|(v_{0,j},y_{0,j},z_{0,j},w_{0,j})\|_{\Hf^{\frac 12+20\dl,\frac 12+9\dl}}.
\end{align*}
Consequently, together with \eqref{ujformula2}, \eqref{L2F1}, \eqref{YsCTHs}, and \eqref{L2F2}, we deduce that
\begin{align*}
\| u_{j}\|_{L^{\infty}_{T_{\ast} \wedge T_{\ast \ast} \wedge T_{\max}(j) } H_x^{\frac 12+9\dl}}  \les \|(v_{0,j},y_{0,j},z_{0,j},w_{0,j})\|_{\Hf^{\frac 12+20\dl,\frac 12+9\dl}}.
\end{align*} 
Comparing this with \eqref{blowup} we then conclude that 
\begin{align}
\inf_{j\in \N}T_{\max}(j) \geq T_{\ast} \wedge T_{\ast \ast} =: T_{0}, \label{Tlowerbd}
\end{align}
with $T_0$ only depending on $\|u_0\|_{H^{s}}$ and not on $j\in \N$.

Now we can pass to the limit in $j$ on the time interval $[0,T_0]$ and construct low-regularity solutions to \eqref{INLS}.
We define 
\begin{align}
u :  =e^{-i\be F[v_{\infty}+y_{\infty}+z_{\infty}]}(v_{\infty}+y_{\infty}+z_{\infty}), \quad u(0)=e^{-i\be F[v_{0}+y_{0}+z_{0}]}(v_{0}+y_{0}+z_{0})=u_0. \label{udefn}
\end{align}
By using \eqref{ujformula2}, \eqref{L2F2}, and \eqref{Lipconv}, we deduce that
\begin{align}\notag
\| u_{j} - u\|_{L^{\infty}_{T_0}H^{\s}_x} \les  \| (v_{0,j}, y_{0,j}, z_{0,j}, w_{0,j}) - (v_{0}, y_{0}, z_{0}, w_{0})\|_{\Hf^{s,\s}}  
\end{align}
and hence $u_j \to u$ in $C_{T_0}H^{\s}_x$. Moreover, since $\P_{-,\text{hi}}u_j=w_{j}$, it follows that $\P_{-,\text{hi}}u=w$ with the convergence taking place in the more regular space $C_{T_0}H^{s}_{x}$. In the following, we ease the notation by defining $(v,y,z,w): = (v_{\infty}, y_{\infty}, z_{\infty}, w_{\infty})$.

We now upgrade the regularity and convergence property of $u$ to $C_{T_0}H^{s}_x$. Notice that this reduces to showing that $\P_{+}u\in L^{\infty}_{T_0}H^s_x$ and that $\P_{+}u_j$ converges to $\P_+ u$ in $L^{\infty}_{T_0}H^s_x$. It is clear that we need only justify these statements for $\P_{+,\text{HI}}u$.
For this purpose, we need to ensure that the structural decomposition \eqref{yjdecomp} holds in the limit.
By using \eqref{Lipconv} with \eqref{YPOS2} and \eqref{YNEG22}, we may pass to the limit (say within the norm $L^{\frac 32}_{T_0,x}$) in both sides of \eqref{yjdecomp2} to find that  $y$ satisfies
\begin{align}
y& =\Y^{\pos}[v+y+z] + \Y^{\negg}[ v+y+z, w]+ e^{i\be F[v+y+z]}w.  \label{yjdecomp3}
\end{align}

Inserting \eqref{yjdecomp3} into \eqref{udefn} we find
\begin{align*}
u   =e^{-i\be F[v+y+z]}(v+\Y^{\pos}[v+y+z] + \Y^{\negg}[ v+y+z, w]+z)  + w.
\end{align*}
Applying $\PbHIp$ to both sides and using that $\P_{+}\P_{-}=0$ we find
\begin{align}
\PbHIp u = \PbHIp\big[ e^{-i\be F[v+y+z]}(v+\Y^{\pos}[v+y+z] + \Y^{\negg}[w, v+y+z]+z)\big]. \label{P+u}
\end{align}
Since $u_0\in H^s$, we have that $y_{0}\in H^{s}_{-}$. Then, by \eqref{YposLTHs1} and \eqref{YnegLTHs1}, we have 
\begin{align}
& 
\|\Y^{\pos}[v+y+z]\|_{L^{\infty}_{T_0}H^s_x} +\|\Y^{\negg}[v+y+z,w]\|_{L^{\infty}_{T_0}H^s_x}  
\notag
\\
&\qquad \les (1+\|u\|_{L^{\infty}_{T_0}L^2_x})^2 \big\{ \|v+y+z\|_{L^{\infty}_{T_0}H^{\s}_x}^3+ \|w\|_{L^{\infty}_{T_0}H^{s}_x}
\big\}
.
\label{uHsup1}
\end{align}
Therefore, by \eqref{P+u}, \eqref{L2F1}, and \eqref{uHsup1}, we have
\begin{align*}
\| \PbHIp u\|_{L^{\infty}_{T_0}H^s_x} & \les C(\|v+y+z\|_{L^{\infty}_{T_0}H^s_x}) \big[ \|v\|_{L^{\infty}_{T_0}H^{s}_x} + \|\Y^{\pos}[v+y+z]\|_{L^{\infty}_{T_0}H^s_x} 
\\
&
\hspace{4cm} + \|\Y^{\negg}[v+y+z,w]\|_{L^{\infty}_{T_0}H^s_x} +\|z\|_{L^{\infty}_{T_0}L^2_x}\big]
 <\infty,
\end{align*}
for some positive polynomial $C$, 
showing that in fact $u\in L^{\infty}_{T_0}H^{s}_x$.  Similar arguments justify that $\PbHIp u_j\to \PbHIp u$ in $C_{T_0}H^s_x$.

We now show that $u$ defined in \eqref{udefn}  solves \eqref{INLS}.
More precisely, we will show that $u\in C([0,T_0];H^{s}(\R))$ satisfies the Duhamel formulation
\begin{align}
u(t) = S(t)u_0 + \int_{0}^{t} S(t-t') [ 2\be u\P_{+}\dx(| u|^2) + u\QQ_{h}(|u|^2)] dt' \label{distDuh}
\end{align}
in the space $X^{-3,b}_{T}$. We know that 
\begin{align}\notag
u_j(t) = S(t)u_{0,j} + \int_{0}^{t} S(t-t') [ 2\be u_j\P_{+}\dx(| u_j|^2) + u_j\QQ_{h}(|u_j|^2)] dt'
\end{align}
and we will pass to the limit as $j\to \infty$ in $X^{0,b}_{T}$.
Indeed, by \eqref{ujformula2}, \eqref{udefn}, \eqref{Expdiffs}, \eqref{L4}, and writing $F_{j}:= F[v_j+y_j+z_j]$ and $F:= F[v+y+z]$, we have
\begin{align}
\big\| u_j &\QQ_{h}( |u_j|^2) - u\QQ_{h}( |u|)^2\big\|_{L^2_{T_0,x}}
\notag\\
& = \big\|  e^{i\be F_j}(v_j + y_j+z_j) \QQ_{h}( |v_j+y_j+z_j|^2) -e^{i\be F}(v+y+z)\QQ_{h}( |v+y+z|^2)\big\|_{L^2_{T_0,x}} 
\notag\\
& \les \|e^{i\be F_j}-e^{i\be F}\|_{L^{\infty}_{T_0,x}} \big( \|v_j+y_j+z_j\|_{L^{6}_{T_0,x}} + \|v+y+z\|_{L^{6}_{T_0,x}}\big)^{3} 
\notag\\
& \hphantom{XX} + \big(\|v_j+y_j+z_j\|_{L^{6}_{T_0,x}}+ \|v+y+z\|_{L^{6}_{T_0,x}}\big)^{2}   \|v_j+y_j+z_j-v-y-z\|_{L^6_{T_0,x}} 
\notag\\
& \to 0,
\label{cubicdist}
\end{align}
as $j\to \infty$ by the strong convergences of $v_j,y_j,z_j$ to $v,y,z$ in $X^{\s,b}_{T_0}$.
Therefore, by Lemma~\ref{LEM:linXsb}, we have
\begin{align*}
\int_{0}^{t}S(t-t') [ u_{j}\QQ_{h}(|u_j|^2)dt' \to \int_{0}^{t}S(t-t') [ u\QQ_{h}(|u|^2)dt' 
\end{align*}
in $X^{0,b}_{T}$ as $j\to\infty$.

For the derivative nonlinear term in \eqref{INLS2}, we write
\begin{align}
u_j \P_+ \dx( |u_j|^2)-u \P_{+}\dx(|u|^2) =&  (\P_{+}u_j)\P_+ \dx( |u_j|^2)-(\P_{+}u)\P_+ \dx( |u|^2) \label{dist1}\\
&+ (\Pblo u_j)\P_+ \dx( |u_j|^2)- (\Pblo u)\P_{+}\dx(|u|^2) \label{dist2} \\
& + w_j \P_{+}\dx(|v_j+y_j+z_j|^2) - w\P_{+}\dx(|v+y+z|^2). \label{dist3}
\end{align}
First, the estimate 
\begin{align*}
\| (\P_{+}f_1)\P_{+}\dx( f_2 f_3)\|_{H^{-3}_{x}} \les \|f_1\|_{L^2_x}\|f_2\|_{L^2_x}\|f_3\|_{L^2_x}
\end{align*}
combined with the convergence $u_j\to u$ in $L^{\infty}_{T_0}H^s_x$ implies that the term \eqref{dist1} converges to zero in $L^{2}_{T_0}H^{-3}_x$.
Next, for \eqref{dist2}, we have
\begin{align*}
(\Pblo u) \P_{+}\dx( |u|^2) = \dx\big[ (\Pblo u) \P_{+}( |u|^2)\big] - (\dx \Pblo u) \P_{+}( |u|^2), 
\end{align*} 
and we write the same for the $u_j$-terms. For the first term, the derivative is absorbed in $X^{-3,b}_{T}$ and the convergence claim reduces to similar computations as in \eqref{cubicdist}. Similarly the contributions from the second term are handled similarly as there is no longer a derivative loss to overcome. 
Thus, the term \eqref{dist2} converges to zero in $X^{-3,b}_{T}$. Lastly, we consider the term \eqref{dist3}. 
We focus on the contribution when $\P_{+}$ is replaced by $\P_\hi$ since otherwise a similar argument as in \eqref{cubicdist} establishes the convergence in that case. We then use the following estimate: 
\begin{align}
\| f_1 \Pbhip \dx( \cj{f_2}f_3)\|_{X^{0,-b'}_{T}} \les \| f_1\|_{X^{\s,b}_{T}}\|f_2\|_{X^{\s,b}_{T}}\|f_3\|_{X^{\s,b}_{T}}. \label{cubicdist2}
\end{align}
The estimate \eqref{cubicdist2} follows from similar arguments as in \eqref{doublebilin} using the bilinear Strichartz estimate \eqref{bilin1} twice and using the derivatives $\s$ to overcome the loss and to perform the dyadic summations. This completes the verification of \eqref{distDuh}.

We now prove the local Lipschitz continuity of the map
\begin{align}
u_0 \in H^s (\R) \mapsto \Phi_{t}(u_0)=: u(t) \in C([0,T];H^s(\R)). \label{solmap}
\end{align}
Let $u_0^{(1)}$ and $u_{0}^{(2)}$ be in $H^s(\R)$ and consider the corresponding solutions $\Phi_{t}(u_0^{(k)})$, $k=1,2$, which belong to $C_{T}H^{s}_x$ for some common $T>0$. 
We have the decomposition
\begin{align*}
u^{(k)} = e^{i\be F[ v^{(k)}+y^{(k)}+z^{(k)}]} (v^{(k)}+\Y^{\pos}[v^{(k)}+y^{(k)}+z^{(k)}]+\Y^{\negg}[v^{(k)}+y^{(k)}+z^{(k)}, w^{(k)}]+z^{(k)})+w^{(k)}
\end{align*}
where $(v^{(k)}, y^{(k)}, z^{(k)}, w^{(k)})\in \mathfrak{X}^{s,\s}_{T}$, $y^{(k)}$ satisfies \eqref{yjdecomp3}, and
\begin{align*}
\|(v^{(k)}, y^{(k)}, z^{(k)}, w^{(k)})\|_{\mathfrak{X}^{s,\s}_{T}} \leq K
\end{align*} 
for some $K>0$, and
satisfy \eqref{nonlinsmooth}, for $k=1,2$. Then, by \eqref{L2F1}, \eqref{L2F2}, \eqref{YposLTHs1}, \eqref{YnegLTHs1}, we have 
\begin{align*}
\|& u^{(1)}-u^{(2)}\|_{C_{T}H^{s}_x}  \\
& \leq C(K) 
\Big[ \| v^{(1)}- v^{(2)}\|_{X^{s,b}_{T}} + \| \Y^{\pos}[v^{(1)}+y^{(1)}+z^{(1)}]-\Y^{\pos}[v^{(2)}+y^{(2)}+z^{(2)}]\|_{L^{\infty}_{T}H^{s}_x} \\
&\hphantom{XXXXXX}+ \| \Y^{\negg}[v^{(1)}+y^{(1)}+z^{(1)}, w^{(1)}]-\Y^{\pos}[v^{(2)}+y^{(2)}+z^{(2)},w^{(2)}]\|_{L^{\infty}_{T}H^{s}_x} \\
&\hphantom{XXXXXX}+  \| z^{(1)}-z^{(2)}\|_{X^{s,b}_{T}} + \|w^{(1)}-w^{(2)}\|_{X^{s,b}_{T}}\Big] \\
& \leq C(K,T) \| u_0^{(1)}-u_{0}^{(2)}\|_{H^{s}_x}.
\end{align*}
This establishes the local Lipschitz continuity of the solution map \eqref{solmap}. This completes the proof of Theorem~\ref{THM:LWP}.

\begin{remark}\rm \label{RMK:smoothness}
The solution map \eqref{solmap} can be decomposed as the composition
\begin{align*}
    u_0 \stackrel{(\textup{i})}{\longmapsto} (v_0,y_0,z_0,w_0) \stackrel{(\textup{ii})}{\longmapsto} (v,y,z,w) \stackrel{(\textup{iii})}{\longmapsto} u,
\end{align*}
where $\text{(i)}$ is given in \eqref{systdata}, $\text{(ii)}$ is constructed in Proposition~\ref{PROP:system}, and $\text{(iii)}$ is the reconstruction by \eqref{udefn}. The maps  $\text{(i)}$ and $\text{(iii)}$ are smooth, and the map $\text{(ii)}$ is analytic as a consequence of the contraction mapping argument used in Proposition~\ref{PROP:system}.
\end{remark}

\section{Conservation laws}
\label{s:consv}

The goal of this section is to prove that the series \eqref{A series} converges and is conserved under the flow of \eqref{INLSG}. 
To this end, we decompose
\begin{equation}
\Pih f = J_h f + K_hf , 
\label{Pihdecomp}
\end{equation}
where $J_{h}$ and $K_{h}$ are the integral operators
\begin{align}
J_{h}f(x) &= \frac{i}{4h} \text{p.v.} \int_{\R} \sgn(x-y) f(y)dy,
\label{Jh} \\
K_{h}f(x) &= \frac{1}{2}f(x) + \frac{i}{4h} \text{p.v.} \int_{\R} \bigg[ \coth\bigg( \frac{\pi(x-y)}{2h} \bigg) - \sgn(x-y)\bigg] f(y)dy.
\notag 
% \label{Kh} 
\end{align}

The next two lemmas contain useful Hilbert--Schmidt and operator-norm estimates, respectively, involving these operators.
Given a dyadic number $N$, it will be convenient in the following to use the shorthand notation $u_{N}:=\P_{N}u$.

\begin{lemma}
For $u,v\in L^2$, $0<h\leq \infty$, $0 < \kappa < \infty$, and $M,N\in 2^{\Z}$, we have
\begin{align}
\big\| u J_h v \big\|_{\hs} 
&\leq \tfrac{1}{4} h^{-1} \|u\|_{L^2} \|v\|_{L^2} ,
\label{hs J} \\
\big\| (\kk-\dx)^{-1} u \big\|_{\hs} = \big\| \cj{u} (\kappa-\dx)^{-1}\big\|_{\hs} 
&= \tfrac12 \kappa^{-\frac12} \|u\|_{L^2} ,
\label{hs} \\
\big\|\cj{u} (\kk-\dx)^{-1} v \big\|_{\hs} &\leq \|u\|_{L^2}\|v\|_{L^2} , 
\label{hs 3} \\
\big\|\cj{u}_M (\kk-\dx)^{-1} v_N \big\|_{\hs} &\lesssim \min\Big( 1, \sqrt{ \tfrac{M}{\kk} }, \sqrt{ \tfrac{N}{\kk} } \Big) \|u_M\|_{L^2}\|v_N\|_{L^2} .
\label{hs 2} 
\end{align}
\end{lemma}
\begin{proof}
The estimate \eqref{hs J} is elementary, once one realizes that the operator in question is an integral operator with a square-integrable kernel:
\[
\big\| u J_h v \big\|_{\hs}^2
= \iint \big| u(x) \tfrac{i}{4h} \sgn( x-y ) v(y) \big|^2\,dx\,dy
\leq \tfrac{1}{16} h^{-2} \|u\|_{L^2}^2 \|v\|_{L^2}^2 . 
\]

Next, we recall that $(\kk\pm\dx)^{-1}$ is also an integral operator, given by
\begin{equation}
[(\kk\pm\dx)^{-1}f](x) 
= \int e^{-\kappa|x-y|}\ind_{\pm (x-y) > 0} f(y)\,dy . 
\label{R kernel}
\end{equation}
A direct computation then shows that
\begin{align*}
\big\| (\kk-\dx)^{-1} u \big\|_{\hs}^2
&= \tr\{ \cj{u} (\kk+\dx)^{-1} (\kk-\dx)^{-1} u \} \\
&= \iint |u(x)|^2 e^{-2\kappa|x-y|} \ind_{x>y} \,dx\,dy \\
&= \tfrac12\kappa^{-1}\|u\|_{L^2}^2 .
\end{align*}
This establishes the first term of \eqref{hs}, and the second term follows a parallel argument.

The operator in \eqref{hs 2} also has a square-integrable kernel, with
\begin{equation*}
\big\| \cj{u} (\kappa-\dx)^{-1} u \big\|_{\hs}^2
=\iint \big| \cj{u}(x) e^{-\kappa|x-y|}\ind_{ x < y} u(y) \big|^2\,dx\,dy
\leq \|u\|_{L^2}^4 .
\end{equation*}
This proves \eqref{hs 3}, as well as \eqref{hs 2} when the minimum on the right-hand side is equal to $1$.

In a similar vein, consideration of the kernel in Fourier variables leads us to compute
\begin{align*}
\big\|\cj{u}_M &(\kk-\dx)^{-1} v_N \big\|_{\hs}^2 
= \tr\{ \cj{v}_N (\kk+\dx)^{-1} u_M\cj{u}_M (\kk-\dx)^{-1} v_N \} \\
&= \frac{1}{(2\pi)^2} \int_{\R^4} \frac{\ft{\cj{v}}_N(\xi_1-\xi_4) \ft{u}_M(\xi_4-\xi_3) \ft{\cj{u}}_M(\xi_3-\xi_2) \ft{v}_N(\xi_2-\xi_1)}{(\kk + i\xi_4) (\kk-i\xi_2)}\,d\xi_1 d\xi_2  d\xi_3 d\xi_4 \\
&= \frac{1}{(2\pi)^2} \int_{\R^4} \frac{\cj{\ft{v}}_N(-\eta_1-\eta_2-\eta_3) \ft{u}_M(\eta_3) \cj{\ft{u}}_M(-\eta_2) \ft{v}_N(\eta_1)}{(\kk - i\eta_4) (\kk+i(\eta_2+\eta_3+\eta_4))}\,d\eta_1 d\eta_2 d\eta_3 d\eta_4 \\
&= \frac{1}{2\pi} \int_{\R^3} \frac{\cj{\ft{v}}_N(-\eta_1-\eta_2-\eta_3) \ft{u}_M(\eta_3) \cj{\ft{u}}_M(-\eta_2) \ft{v}_N(\eta_1)}{2\kk + i(\eta_2+\eta_3)}\,d\eta_1 d\eta_2 d\eta_3 \\
&\lesssim \|u_M\|_{L^2}^2 \| v_N\|_{L^2}  \bigg( \int_{\R^3} \frac{|\ft{v}_N(-\eta_1-\eta_2-\eta_3)|^2}{4\kk^2 + (\eta_2+\eta_3)^2}\ind_{|\eta_2|\sim M\sim|\eta_3|}\,d\eta_1 d\eta_2 d\eta_3 \bigg)^{\frac12} \\
&\lesssim \tfrac{M}{\kappa} \|u_M\|_{L^2}^2 \| v_N\|_{L^2}^2 .
\end{align*}
In passing to the second to last line, we used Cauchy--Schwarz.  This yields the second term in the minimum in RHS\eqref{hs 2}.  To obtain the third term, we replace $\eta_2 + \eta_3 = (\eta_1+\eta_2+\eta_3) - \eta_1$ in the computation above and observe that the integrand is supported in the region $|\eta_1+\eta_2+\eta_3| \sim N \sim |\eta_1|$.
\end{proof}

\begin{lemma}
Fix $0<r<\frac12$ and set $p_r = \frac{2}{1-2r}$.
Then for any $0<h\leq \infty$, $0<\kappa<\infty$, and $M,N\in 2^{\Z}$, we have
\begin{align}
\| K_h\|_{\op} &\lesssim 1,
\label{op K} \\
\big\| (\kappa-\dx)^{-1} u_{N} \big\|_{\op}
\leq \big\| (\kappa-\dx)^{-1} u_{ N} \big\|_{\mathfrak{I}_{p_r}}
&\lesssim  \kappa^{-\frac12-r}N^r \|u_{ N}\|_{L^2} ,
\label{op 3} \\
\big\|\cj{u}_M (\kk-\dx)^{-1} v_N \big\|_{\op} &\lesssim \min\big( \tfrac{M^r}{\kk^r},1 \big) \min\big( \tfrac{N^r}{\kk^r},1 \big) \|u_M\|_{L^2}\|v_N\|_{L^2},
\label{op 2} \\
\big\|\cj{u} (\kk-\dx)^{-1} u \big\|_{\op} &\lesssim \kappa^{-2r} N^{2r} \|u_{\leq N}\|_{L^2}^2 + \|u_{>N}\|_{L^2}^2 .
\label{op 6}
\end{align}
\end{lemma}
\begin{proof}
The estimate \eqref{op K} is trivial once one considers the Fourier symbol of $K_h$, since
\[
\sup_{\xi\in\R} \tfrac{1}{2}\big|  1+ \coth(h\xi) - \tfrac{1}{h\xi} \big| = \sup_{\xi\in\R} \tfrac{1}{2}\big|  1+ \coth\xi - \tfrac{1}{\xi} \big| <\infty 
\]
for all $h>0$.

The first inequality in \eqref{op 3} follows immediately from the embedding $\ell^{p_r} \subseteq\ell^\infty$.  As $2\leq p_r\leq \infty$, then by \cite[Theorem 4.1]{Simon} we have
\[
\big\| (\kappa-\dx)^{-1} u_{ N} \big\|_{\mathfrak{I}_{p_r}}
\lesssim \big\| (\kappa-i\xi)^{-1} \big\|_{L^{p_r}_{\xi}} \| u_{ N} \|_{L^{p_r}} 
\lesssim \kappa^{\frac{1}{p_r} - 1} N^{\frac12 - \frac{1}{p_r}} \|u_{ N} \|_{L^2} .
\]
Inserting $p_r = \frac{2}{1-2r}$ then yields \eqref{op 3}.

Next, we turn to \eqref{op 2}.  When both $M>\kk$ and $N> \kk$, we use \eqref{hs 3} to bound
\begin{align*}
\big\|\cj{u}_{M} (\kk-\dx)^{-1} v_N \big\|_{\op} 
\leq \|u_M\|_{L^2} \|v_N\|_{L^2} .
\end{align*}

Recall that $(\kappa-\dx)^{-1}$ is an integral operator for the kernel~\eqref{R kernel}.  A direct computation shows that for any $1\leq p\leq\infty$, we have
\begin{equation}
\big\| e^{-\kappa|x|}\ind_{x<0} \big\|_{L^{p}}\lesssim_p \kappa^{-\frac1p}  .
\label{R Lp}
\end{equation}

In the case $M\leq \kk$ and $N>\kk$, we apply \eqref{R Lp} to $p=\frac1r$ and use Young's inequality to bound
\[
\big\| (\kappa-\dx)^{-1} v_N f \big\|_{L^{\frac1r}} 
\lesssim \kappa^{-r} \| v_N f \|_{L^1}
\lesssim \kappa^{-r} \|v_N\|_{L^2} \|f\|_{L^2} ,
\]
Together with the H\"older and Bernstein inequalities, this yields
\begin{align*}
\big\| \cj{u}_{M} (\kk-\dx)^{-1} v_N f \big\|_{L^2}
&\leq \| u_{M} \|_{L^{\frac{2}{1-2r}}} \big\| (\kappa-\dx)^{-1} v_N f \big\|_{L^{\frac1r}}  \\
&\lesssim \kappa^{-r} M^r \|u_{M} \|_{L^2} \|v_N\|_{L^2} \|f\|_{L^2} .
\end{align*}

Similarly, for the case $M>\kk$ and $N\leq \kk$ we have
\[
\big\| (\kappa-\dx)^{-1} v_N f \big\|_{L^{\infty}} 
\lesssim \kappa^{-r} \| v_N f\|_{L^{\frac{1}{1-r}}}
\lesssim \kappa^{-r} \|v_N\|_{L^{\frac{2}{1-2r}}} \|f\|_{L^2} ,
\]
and thus
\begin{align*}
\big\| \cj{u}_{M} (\kk-\dx)^{-1} v_N f \big\|_{L^2}
&\leq \| u_{M} \|_{L^{2}} \big\| (\kappa-\dx)^{-1} v_N f \big\|_{L^{\infty}}  \\
&\lesssim \kappa^{-r} N^r \|u_{M} \|_{L^2} \|v_N\|_{L^2} \|f\|_{L^2} .
\end{align*}

Finally, when $M,N\leq \kk$, we employ \eqref{R Lp} with $p=\frac{1}{2r}$ to obtain
\[
\big\| (\kappa-\dx)^{-1} v_N f \big\|_{L^{\frac1r}} 
\lesssim \kappa^{-2r} \| v_N f \|_{L^{\frac{1}{1-r}}}
\lesssim \kappa^{-2r} \|v_N\|_{L^{\frac{2}{1-2r}}} \|f\|_{L^2} ,
\]
from whence
\begin{align*}
\big\| \cj{u}_{M} (\kk-\dx)^{-1} v_N f \big\|_{L^2}
&\leq \| u_{M} \|_{L^{\frac{2}{1-2r}}} \big\| (\kappa-\dx)^{-1} v_N f \big\|_{L^{\frac1r}}  \\
&\lesssim \kappa^{-2r} \|u_{M} \|_{L^{\frac{2}{1-2r}}} \|v_N \|_{L^{\frac{2}{1-2r}}} \|f\|_{L^2} \\
&\lesssim \kappa^{-2r} M^r N^{r} \|u_{M} \|_{L^2} \|v_N\|_{L^2} \|f\|_{L^2} .
\end{align*}

Collecting the previous steps, we conclude that \eqref{op 2} holds in all cases.

Lastly, we turn to \eqref{op 6}.  Decomposing each $u = u_{\leq N} + u_{>N}$ produces four terms, which follow arguments parallel to the previous four cases:
\begin{align*}
\big\| \cj{u}_{>N} (\kk-\dx)^{-1} u_{>N} \big\|_{\op}
&\leq \| u_{>N}\|_{L^2}^2 ,\\
\big\| \cj{u}_{\leq N} (\kk-\dx)^{-1} u_{> N} \big\|_{\op}
&\lesssim \kappa^{-r}N^r \|u_{\leq N}\|_{L^2} \|u_{>N}\|_{L^2} ,\\
\big\| \cj{u}_{> N} (\kk-\dx)^{-1} u_{\leq N} \big\|_{\op}
&\lesssim \kappa^{-r}N^r \|u_{> N}\|_{L^2} \|u_{\leq N}\|_{L^2} ,\\
\big\| \cj{u}_{\leq N} (\kk-\dx)^{-1} u_{\leq N} \big\|_{\op}
&\lesssim \kappa^{-2r}N^{2r} \|u_{\leq N}\|_{L^2}^2 .
\end{align*}
Adding these together yields \eqref{op 6}.
\end{proof}

Unfortunately, we cannot make sense of the first term of the series \eqref{A series} when $u\in L^2$.  Specifically, the contribution of $J_h$:
\[
(\kappa-\dx)^{-1} uJ_h\cj{u}(\kappa-\dx)^{-1}
\]
is a trace-class operator if, say, $u\in L^1$ (see \eqref{tr J} below for details).  However, for $u\in L^2$, it is unclear if this operator belongs to the trace class.

Instead, we define the terms of $A$ that are quadratic in $u$ via the following lemma.  This formulation has the advantage that it continues to make sense for $u\in L^2$.
\begin{lemma}[Quadratic energy]
For $0<h\leq \infty$ and $0 < \kappa < \infty$, define
\begin{equation}
A^{[2]}(\kappa;u,h) := 
\int F(\xi;\kappa,h) |\ft{u}(\xi)|^2\,d\xi 
+ \int \cj{u}(x) (G*u)(x)\,dx,
\label{A F}
\end{equation}
where 
\begin{align}
F(\xi;\kappa,h) &:= \begin{cases} 
\displaystyle \tfrac{\beta}{4\pi} \int \tfrac{1+\coth(h\eta)-\frac{1}{h\eta}}{(\kappa + i(\xi+\eta))^2} \,d\eta & \text{if }h<\infty, \\
\displaystyle \tfrac{\beta}{2\pi} \tfrac{1}{i\kappa-\xi} & \text{if }h=\infty ,
\end{cases}
\notag 
% \label{F} 
\\
G(x;\kappa,h) &:= \begin{cases} 
\frac{i\beta}{4h} xe^{-\kappa|x|} \ind_{x>0} & \text{if }h<\infty, \\
0 & \text{if }h=\infty .
\end{cases}
\notag 
% \label{G}
\end{align}
Then we have
\begin{align}
\bigg| F(\xi;\kappa,h) - \frac{\beta}{2\pi} \frac{1}{i\kappa-\xi} \bigg| 
&\lesssim h^{-1}\kappa^{-2} ,
\label{F 2} \\
\bigg| \int \cj{u}(x) (G*u)(x)\,dx \bigg|
&\leq \tfrac14 h^{-1} \kappa^{-2} \|u\|_{L^2}^2
\label{G 2}
\end{align}
uniformly for $\xi\in\R$ and $h,\kappa>0$.
As a consequence,
\begin{align*}
|A^{[2]}(\kappa;u,h) | 
&\lesssim (\kappa^{-1}+h^{-1}\kappa^{-2}) \|u\|_{L^2}^2 .
\end{align*}
Additionally, when $u\in L^1\cap L^2$,
\begin{equation}
A^{[2]}(\kappa;u,h) = \beta \tr\{ (\kappa-\dx)^{-1}u\Pih\cj{u}(\kappa-\dx)^{-1} \} .
\label{A2}
\end{equation}
\end{lemma}
\begin{proof}
Let us begin with \eqref{A2}, in order to motivate the definition \eqref{A F}.
Considering the integral kernel of the operator in Fourier variables, we find
\begin{align*}
\tr\{ (\kappa-\dx)^{-1}uK_h\cj{u}(\kappa-\dx)^{-1} \}
&= \frac{1}{4\pi} \iint \frac{1+\coth(h\eta)-\frac{1}{h\eta}}{(\kappa+i\xi)^2} \ft{u}(\xi-\eta) \ft{\cj{u}}(\eta-\xi)\,d\eta\,d\xi \\
&= \frac{1}{4\pi} \iint \frac{1+\coth(h\eta)-\frac{1}{h\eta}}{(\kappa+i(\xi+\eta))^2} |\ft{u}(\xi)|^2\,d\eta\,d\xi .
\end{align*}
This explains the contribution of $F$ to \eqref{A2}.

Next, we turn to the contribution of $G$.  Let $k(x) = e^{-\kappa |x|} \ind_{x<0}$
denote the convolution kernel of $(\kk-\dx)^{-1}$, so that
\[
[(\kk-\dx)^{-1}f](x) = \int k(x-y)f(y)\,dy.
\]
Similarly, let $j_h(x)$ denote the convolution kernel of $J_h$ in \eqref{Jh}.
We compute
\begin{align*}
\tr\{(\kk-\dx)^{-1}u J_{h}\cj{u}(\kk-\dx)^{-1} \} 
= \int k(x-y)u(y)j_{h}(y-z) \cj{u}(z) k(z-x)\,dx\,dy\,dz ,
\end{align*}
where
\begin{align*}
\int  k(x-y)k(z-x)\,dx 
= \int e^{-\kappa(y-x)-\kappa(x-z)}\ind_{z<x<y}\,dx
=  (y-z) e^{-\kappa(y-z)} \ind_{y>z}.
\end{align*}
Inserting this and the kernel~\eqref{Jh} for $J_h$, we find
\begin{align*}
\tr\{ (\kk-\dx)^{-1}u J_{h}\cj{u}(\kk-\dx)^{-1} \} 
= \frac{i}{4h} \iint u(y)\cj{u}(z) \ind_{y>z} (y-z) e^{-\kappa(y-z)} \,dy\,dz .
\end{align*}
This accounts for the contribution of $G$, and finishes the proof of \eqref{A2}.  Note that the assumption $u\in L^1$ guarantees that the left-hand side above is well-defined by \eqref{hs} and \eqref{hs J}:
\begin{align}
\big| \tr\{ (\kk-\dx)^{-1}u J_{h}\cj{u}(\kk-\dx)^{-1} \} \big|
&\leq \big\| (\kk-\dx)^{-1} \sqrt{|u|} \big\|_{\hs} \big\| \tfrac{u}{\sqrt{|u|}} J_h \sqrt{|u|} \big\|_{\hs} \big\| \tfrac{\cj{u}}{\sqrt{|u|}}(\kk-\dx)^{-1} \big\|_{\hs} 
\nonumber \\
&\lesssim h^{-1}\kappa^{-1} \|u\|_{L^1}^2 .
\label{tr J}
\end{align}

Next, we turn to \eqref{G 2}.  By Cauchy--Schwarz and Young's inequality,
\begin{align*}
\bigg| \int \cj{u}(x) (G*u)(x)\,dx \bigg|
&\leq \frac{1}{4h} \iint |u(x)||u(y)| (x-y) e^{-\kappa(x-y)} \ind_{x>y} \,dx\,dy \\
&\leq  \frac{1}{4h} \| u\|_{L^2} \Big\| \int u(y) (x-y)e^{-\kk (x-y)} \ind_{x>y}dy \Big\|_{L^2_x} \\
& \leq  \frac{1}{4h} \| u\|_{L^2}^{2} \| xe^{-\kk x}\ind_{x>0}\|_{L^1_x} = \frac{1}{4h\kappa^2}\|u\|_{L^2}^{2} .
\end{align*}

Lastly, we turn to \eqref{F 2}.  When $h=\infty$, $\coth(h\xi)$ is replaced by $\sgn(\xi)$, and we can compute $F$ exactly:
\[
F(\xi;\kappa,\infty) 
= \frac{\beta}{2\pi} \int_0^\infty \frac{1}{(\kappa + i(\xi+\eta))^2} \,d\eta 
= \frac{\beta}{2\pi} \frac{1}{i\kappa-\xi} .
\]
This corresponds to the left-hand side of \eqref{F 2}.

In the case $h<\infty$, $F$ can still be computed exactly (although rather tediously) using residues and a semicircular contour.  Ultimately, however, we do not find this formula as useful as direct estimation:
\begin{align*}
\bigg| \int \frac{1}{(\kappa + i(\xi+\eta))^2} [\coth(h\eta) -\tfrac{1}{h\eta} - \sgn(\eta)]\,d\eta \bigg|
&\lesssim \kappa^{-2} \int \big| \coth(h\eta) -\tfrac{1}{h\eta} - \sgn(\eta)\big|\,d\eta \lesssim \frac{1}{h \kk^2} .
\end{align*}
This proves \eqref{F 2}.
\end{proof}

We are now equipped to establish the convergence of the quartic-and-higher terms of the series \eqref{A series}:
\begin{proposition}\label{t:A series}
Fix $h>0$.  If $Q\subseteq L^2$ is bounded and equicontinuous, then there exists $\kappa_0 = \kappa_0(Q)\geq 1$ so that for $\kappa\geq\kappa_0$ the series 
\begin{equation}
A^{[\geq 4]}(\kappa;u,h) := -i \sum_{n\geq 2} (i\beta)^n \tr\big\{ \big[ (\kappa-\dx)^{-1} u\Pih\cj{u} \big]^n (\kappa-\dx)^{-1} \big\} 
\label{A4}
\end{equation}
converges uniformly for $u\in Q$.
\end{proposition}
\begin{proof}
For $n\geq 2$, we will bound
\[
\big| \tr\big\{ \big[ (\kappa-\dx)^{-1} u\Pih\cj{u} \big]^n (\kappa-\dx)^{-1} \big\} \big| .
\]
We insert the $\Pih$ decomposition~\eqref{Pihdecomp} and expand.  

For the term with $n$ copies of $J_h$, we use \eqref{hs J} to bound
\begin{align*}
\big| \tr\big\{ \big[ (\kappa-\dx)^{-1} u J_h\cj{u} \big]^n (\kappa-\dx)^{-1} \big\} \big|  
&\leq \| uJ_h\cj{u} \|_{\hs}^{n} \|(\kk-\dx)^{-1} \|_{\op}^{n+1} \\
&\leq \kappa^{-1} ( h^{-1} \kk^{-1} \|u\|_{L^2}^2 )^n .
\end{align*}
Note that this requires $n\geq 2$.

The remaining contributions have at least one factor of $K_h$.  When there is just one isolated copy of $K_h$, we use \eqref{hs} and \eqref{op K} to bound
\[
\big\| (\kk-\dx)^{-1} uK_h \cj{u} (\kk-\dx)^{-1} \|_{\tc}
\leq \big\| (\kk-\dx)^{-1} u \big\|_{\hs} \|K_h\|_{\op} \big\| \cj{u} (\kk-\dx)^{-1} \big\|_{\hs}
\lesssim \kappa^{-1} \|u\|_{L^2}^2 .
\]

When two or more copies of $K_h$ appear consecutively, we bound
\begin{align}
\big\| [(\kk-\dx)^{-1} uK_h \cj{u} ]^m (\kk-\dx)^{-1}  \|_{\tc} 
&\leq \big\| (\kk-\dx)^{-1} u \big\|_{\hs}^2 \|K_h\|_{\op}^m \|\cj{u} (\kk-\dx)^{-1} u\|_{\op}^{m-1} 
\nonumber \\
&\lesssim \kappa^{-1} \|u\|_{L^2}^2 \|\cj{u} (\kk-\dx)^{-1} u\|_{\op}^{m-1} .
\label{K run}
\end{align}
In other words, for each run of $K_h$, the left- and right-most copies of $(\kk-\dx)^{-1}u$ and $\cj{u} (\kk-\dx)^{-1}$ are estimated in $\hs$-norm, and everything else is estimated in operator norm.  In particular, as there is at least one copy of $K_h$, there are at least two operators in $\hs$, and so the composition is in $\tc$.  The remaining copies of $J_h$ are then simply dealt with using \eqref{hs J}; e.g.,
\[
\big\| (\kk-\dx)^{-1} uJ_h \cj{u}  \big\|_{\op} + \big\| uJ_h \cj{u} (\kk-\dx)^{-1} \big\|_{\op}
\leq h^{-1} \kk^{-1} \|u\|_{L^2}^2 .
\]

Altogether, this yields
\begin{align}
\big| \tr&\big\{ \big[ (\kappa-\dx)^{-1} u\Pih\cj{u} \big]^n (\kappa-\dx)^{-1} \big\} \big| 
\nonumber \\
&\lesssim (\kappa^{-1} + h^{-1}\kappa^{-2}) \|u\|_{L^2}^2 \big( h^{-1}\kappa^{-1}\|u\|_{L^2}^2 + C \|\cj{u} (\kk-\dx)^{-1} u\|_{\op} \big)^{n-1} 
\label{A4 2}
\end{align}
uniformly for $n\geq 2$, for some constant $C>0$.   
For a parameter $0<\eta<1$ to be chosen, we apply \eqref{op 6} with $r=\frac14$ and $N=\eta\kappa$ to obtain
\begin{align}
\|\cj{u} (\kk-\dx)^{-1} u\|_{\op} 
\lesssim \sqrt{\eta} \|u\|_{L^2}^2 + \|u_{>\eta\kappa}\|_{L^2}^2 .
\label{equicty 3}
\end{align}
Therefore, given a bounded and equicontinuous set $Q\subseteq L^2$, we may choose $\eta>0$ small and then $\kappa_0 \geq 1$ large enough so that
\[
h^{-1}\kappa^{-1}\|u\|_{L^2}^2 + C \|\cj{u} (\kk-\dx)^{-1} u\|_{\op} \leq \tfrac12
\]
for all $\kappa\geq\kappa_0$ and $u\in Q$.  Consequently, the right-hand side of \eqref{A4 2} can be summed in $n$ for all such $\kappa$ and $u$.
\end{proof}

We now arrive at the climax of this section, where we combine the two pieces $A^{[2]}$ and $A^{[\geq 4]}$ to construct our key quantity:

\begin{proposition}[Conservation laws]\label{t:A dot}
Define the quantity
\[
A(\kappa;u,h) := A^{[2]}(\kappa;u,h) + A^{[\geq 4]}(\kappa;u,h) .
\]
If $u(t)$ is an $H^\infty(\R)$ solution to \eqref{INLSG} with initial data $u_0$ satisfying $\langle x \rangle u_0\in L^2(\R)$, then for any time $t$ we have
\begin{equation}
\tfrac{d}{dt} A(\kappa;u(t)) = 0 
\label{A dot}
\end{equation}
for all $\kappa$ sufficiently large.
\end{proposition}
\begin{proof}
Let $u(t)$ be an $H^\infty$ solution to \eqref{INLSG} with initial data $u_0$ satisfying $\langle x \rangle u_0\in L^2$.  First, we claim that the condition $\langle x\rangle u \in L^2$ is preserved under the dynamics.  Given a smooth weight $w\in L^\infty$, integration by parts shows that
\[
\frac{d}{dt} \int w(x) |u(t,x)|^2\,dx
= \int w'(x) \big[ 2\Im (\cj{u}u') - \beta |u|^4 \big](t,x)\,dx .
\]
As $u(t)$ is in $H^1$, a standard Gronwall argument involving a sequence of bounded weights converging pointwise to $x^2$ shows that $\langle x\rangle u(t)$ belongs to $L^2$ for all time.  In particular, $u(t)$ is in $L^1\cap L^2$, and so the representation \eqref{A2} of $A^{[2]}$ as a trace holds.

Fix a time $t_0\in\R$, at which we will verify \eqref{A dot}.  On any bounded time interval containing $t_0$, the functions $u(t)$ are bounded in $H^1$, and so are bounded and equicontinuous in $L^2$.  Therefore, by Proposition~\ref{t:A series}, we may choose $\kappa$ large enough so that the series \eqref{A4} converges uniformly on this time interval.  Uniform convergence guarantees that $t\mapsto A(\kappa;u(t))$ is differentiable at $t_0$, and that the series differentiated term-by-term converges to $\frac{d}{dt} A(\kappa;u(t))$.

Differentiating \eqref{A2} and \eqref{A4}, one sees that it suffices to show
\begin{equation}
\tr\big\{ (\kk-\dx)^{-1} \big[ u''\Pih\cj{u} - u\Pih\cj{u}'' \big] (\kk-\dx)^{-1} \big\} = 0,
\label{A dot 1} 
\end{equation}
and that for each $n\geq 1$,
\begin{equation}
\begin{aligned}
&\sum_{j=0}^{n} \tr\Big\{ \big[ (\kk-\dx)^{-1}u\Pih\cj{u} \big]^j (\kk-\dx)^{-1} \big[ u''\Pih\cj{u} - u\Pih\cj{u}'' \big] \\[-10pt]
&\qquad\qquad\qquad\qquad\qquad\qquad\qquad\qquad\qquad\qquad\circ \big[ (\kk-\dx)^{-1}u\Pih\cj{u} \big]^{n-j} (\kk-\dx)^{-1} \Big\} \\
&+2\sum_{j=0}^{n-1} \tr\Big\{ \big[ (\kk-\dx)^{-1}u\Pih\cj{u} \big]^j (\kk-\dx)^{-1} \big[ u\big(|u|^2\big)'_{+,h}\Pih\cj{u} + u\Pih\cj{u}\big(|u|^2\big)'_{-,h} \big] \\[-10pt]
&\qquad\qquad\qquad\qquad\qquad\qquad\qquad\qquad\qquad\quad\circ \big[ (\kk-\dx)^{-1}u\Pih\cj{u} \big]^{n-j-1} (\kk-\dx)^{-1} \Big\} \\
&=0 .
\end{aligned}
\label{A dot 2} 
\end{equation}
Here and throughout the proof, we use the shorthand
\[
f' = \dx (f)
\quad\text{and}\quad
f_{\pm,h} = \Pi_{\pm,h}(f) .
\]
For the definition of $\Pi_{\pm,h}$, see \eqref{Pi+h}.  This notation has the advantage that it allows us to easily distinguish the operator corresponding to multiplication by $f'$ from the operator $\dx f$ which maps $g\mapsto \dx(fg)$.

We begin with \eqref{A dot 1}.  Writing $u'' = [\dx,u']$, cycling the trace (i.e.\ $\tr(AB) = \tr(BA)$), and using the commutativity of Fourier multipliers, we find
\begin{align*}
\tr\big\{ &(\kk-\dx)^{-1} \big[ u''\Pih\cj{u} - u\Pih\cj{u}'' \big] (\kk-\dx)^{-1} \big\} \\
&= - \tr\big\{ (\kk-\dx)^{-1} \big[u'\Pih\cj{u}' - u'\Pih\cj{u}' \big] (\kk-\dx)^{-1} \big\} \\
&= 0 .
\end{align*}
This proves \eqref{A dot 1}.  Note that the left-hand side is finite as $u\in H^\infty$ and $u\in L^1$; for example, by \eqref{hs} and \eqref{hs J},
\begin{align*}
\big| \tr\{ (\kk-\dx)^{-1}u'' J_{h}\cj{u}(\kk-\dx)^{-1} \} \big|
&\leq \big\| (\kk-\dx)^{-1} \big\|_{\op} \big\| u'' J_h \sqrt{|u|} \big\|_{\hs} \big\| \tfrac{\cj{u}}{\sqrt{|u|}}(\kk-\dx)^{-1} \big\|_{\hs} \\
&\lesssim h^{-1}\kappa^{-\frac32} \|u\|_{L^1} \|u''\|_{L^2} .
\end{align*}

Next, we turn to \eqref{A dot 2}.  Arguing as in Proposition~\ref{t:A series}, it is clear that any trace containing two free $\Pih$ operators (i.e.\ not counting $(|u|^2)'_{+,h}$) is well-defined.  The only term which does not fit this description is
\begin{align*}
\tr\Big\{ (\kk-\dx)^{-1} \big[ u\big(|u|^2\big)'_{+,h}\Pih\cj{u} + u\Pih\cj{u}\big(|u|^2\big)'_{-,h} \big] (\kk-\dx)^{-1} \Big\} .
\end{align*}
However, this is well-defined for $u\in H^\infty$ by \eqref{hs}, \eqref{hs J}, and \eqref{op K}; for example,
\begin{align*}
\big| \tr\{ (\kk-\dx)^{-1}  u\big(|u|^2\big)'_{+,h} J_{h}\cj{u}(\kk-\dx)^{-1} \} \big|
&\leq \big\| (\kk-\dx)^{-1} u \big\|_{\hs} \big\| \big(|u|^2\big)'_{+,h} J_h \cj{u} \big\|_{\hs} \big\| (\kk-\dx)^{-1} \big\|_{\op} \\
&\lesssim h^{-1}\kappa^{-\frac32} \|u\|_{L^2}^2 \|u\|_{H^1}^2 .
\end{align*}
To prove \eqref{A dot 2}, we will employ the operator identity
\begin{equation}
\cj{u}(\kk-\dx)^{-1} u'' - \cj{u}'' (\kk-\dx)^{-1} u
= - 2\big(|u|^2\big)' + \big[ 2\kk\dx - \dx^2,\cj{u}(\kk-\dx)^{-1}u \big] .
\label{comm}
\end{equation}

When we insert \eqref{comm} into \eqref{A dot 2}, it will be the first term in RHS\eqref{comm} that exhibits the desired cancellation.  Specifically, we claim that
\begin{equation}
\Pih\big(|u|^2\big)'\Pih \cj{u} = \big(|u|^2\big)'_{+,h} \Pih\cj{u} + \Pih\cj{u}\big(|u|^2\big)'_{-,h}
\notag 
% \label{cotlar 2}
\end{equation}
as operators on $L^2$.  Indeed, using the definition \eqref{Pi+h} of $\Pi_{\pm,h}$, we find
\begin{align*}
4\Pih&\big(|u|^2\big)'\Pih \cj{u} - 4\big(|u|^2\big)'_{+,h} \Pih\cj{u} - 4 \Pih\cj{u}\big(|u|^2\big)'_{-,h} \\
&= \big(|u|^2\big)'_{\TT_h} \TT_h \cj{u} - \big(|u|^2\big)'\cj{u} - \TT_h \big(|u|^2\big)'_{\TT_h}\cj{u} - \TT_h \big(|u|^2\big)' \TT_h \cj{u}  = 0 .
\end{align*}
In the last step, we used a Cotlar-type identity for the operator $\TT_h$: for any $f,g\in H^{\infty}\cap L^1$,
\[
f_{\TT_h}\cdot g_{\TT_h} = fg + \big( f_{\TT_h}\cdot g + f\cdot g_{\TT_h} \big)_{\TT_h} + \tfrac{1}{4h^2} {\textstyle\int}f\cdot {\textstyle\int}g .
\]

The contributions from the second term in RHS\eqref{comm} to \eqref{A dot 2} can be combined as follows:
\begin{align*}
&\sum_{\ell=0}^{n-1} \tr\Big\{ \big( (\kk-\dx)^{-1}u\Pih\cj{u} \big)^{\ell} (\kk-\dx)^{-1} u\Pih \big[ 2\kk\dx - \dx^2 ,\, \cj{u}(\kk-\dx)^{-1}u\big] \Pih\cj{u} \\[-10pt]
&\qquad\qquad\qquad\qquad\qquad\qquad\qquad\qquad\qquad\circ\big( (\kk-\dx)^{-1}u\Pih\cj{u} \big)^{n-1-\ell} (\kk-\dx)^{-1} \Big\} \\
&= \tr\Big\{ (\kk-\dx)^{-1}u \Big[ 2\kk\dx - \dx^2 ,\, \big( \Pih\cj{u}(\kk-\dx)^{-1} u  \big)^{n} \Big] \Pih\cj{u} (\kk-\dx)^{-1} \Big\} \\
&= - \tr\Big\{ \Pih \big[ 2\kk\dx - \dx^2 ,\, \cj{u}(\kk-\dx)^{-2}u \big]  \big( \Pih\cj{u}(\kk-\dx)^{-1} u  \big)^{n}  \Big\} .
\end{align*}
In passing to the last line, we cycled the trace.  Each term above is still well-defined for $u\in H^\infty\cap L^1$, even when $n=1$.  For example,
\begin{align*}
\big| \tr\{ &u J_{h}\dx^2 \cj{u}(\kk-\dx)^{-2} u J_h \cj{u} (\kk-\dx)^{-1} \} \big| \\
&\leq \|u\|_{L^\infty} \big\| J_{h}\dx^2 \cj{u}(\kk-\dx)^{-1} \big\|_{\op}  \big\| (\kk-\dx)^{-1} \sqrt{|u|} \big\|_{\hs}  \big\| \tfrac{u}{\sqrt{|u|}} J_h \cj{u} \big\|_{\hs} \big\| (\kk-\dx)^{-1} \big\|_{\op} \\
&\lesssim h^{-2}\kappa^{-\frac32} \|u\|_{L^1} \|u\|_{L^2}  \|u\|_{H^1}^2 .
\end{align*}

Collecting the previous steps, we arrive at
\begin{align*}
\text{LHS}\eqref{A dot 2}
={} &\tr\Big\{ \Pih \cj{u} (\kk-\dx)^{-2} u'' \big( \Pih\cj{u}(\kk-\dx)^{-1} u  \big)^{n} \Big\} \\
&- \tr\Big\{ \Pih \cj{u}'' (\kk-\dx)^{-2} u \big( \Pih\cj{u}(\kk-\dx)^{-1} u  \big)^{n} \Big\} \\
&- \tr\Big\{ \Pih \big[ 2\kk\dx - \dx^2 ,\, \cj{u}(\kk-\dx)^{-2}u \big]  \big( \Pih\cj{u}(\kk-\dx)^{-1} u  \big)^{n}  \Big\} .
\end{align*}
Employing the operator identity
\[
\cj{u} (\kk-\dx)^{-2} u'' - \cj{u}'' (\kk-\dx)^{-2} u - \big[ 2\kk\dx - \dx^2 ,\, \cj{u}(\kk-\dx)^{-2}u \big]
= - 2\big[ \dx, \cj{u}(\kk-\dx)^{-1}u \big] ,
\]
we conclude
\begin{align*}
\text{LHS}\eqref{A dot 2}
&=- 2\tr\Big\{ \Pih \big[ \dx, \cj{u}(\kk-\dx)^{-1}u \big] \big( \Pih\cj{u}(\kk-\dx)^{-1} u  \big)^{n} \Big\} = 0 ,
\end{align*}
as desired.  In passing to the last line, we again cycled the trace.
\end{proof}

\section{A-priori estimates}
\label{s:a-priori}

In this section, we will prove Theorems~\ref{t:equicty} and \ref{t:a-priori}, beginning with the former.

Our key quantity for measuring equicontinuity is
\[
\alpha(\kappa;u,h) := \|u\|_{L^2}^2 + 2\pi\beta\kappa \Im A(\kappa;u,h) .
\]
As $A$ is conserved by Proposition~\ref{t:A dot}, then so too will be any combination of it and the mass.  We select this particular combination because by \eqref{F 2} and \eqref{G 2} the quadratic term is
\begin{equation*}
\alpha^{[2]}(\kappa;u,h) 
:= \|u\|_{L^2}^2 + 2\pi\beta\kappa \Im A^{[2]}(\kappa;u,h)
= \int \frac{\xi^2}{\kappa^2+\xi^2} |\ft{u}(\xi)|^2\,d\xi + O\big( h^{-1}\kappa^{-1} \|u\|_{L^2}^2 \big) .
\end{equation*}
The Fourier multiplier appearing above acts as a smooth cutoff to frequencies $|\xi| \gtrsim \kappa$, and so detects equicontinuity via \eqref{equicty 2}.

\begin{proof}[Proof of Theorem~\ref{t:equicty}]
Initially, we require $\delta\leq 1$ is chosen small enough so that, by \eqref{hs 3},
\[
h^{-1}\kappa^{-1}\|u\|_{L^2}^2 + C \|\cj{u} (\kk-\dx)^{-1} u\|_{\op} 
\leq h^{-1}\kappa^{-1}\delta^2 + C\delta^2 
\leq \tfrac12
\]
for all $u\in Q^*$ and $\kappa\geq \kappa_0 := h^{-1}$.  Here, $C>0$ is the constant appearing in \eqref{A4 2}, so that \eqref{A4 2} implies
\begin{align*}
|\alpha(\kappa;u,h) - \alpha^{[2]}(\kappa;u,h)| \lesssim \|u\|_{L^2}^2 \big( h^{-1}\kappa^{-1} \|u\|_{L^2}^2 + \| \cj{u}(\kappa-\dx)^{-1} u\|_{\op} \big) 
\end{align*}
uniformly for $u\in Q^*$ and $\kappa\geq \kappa_0$. 

Let $0<\eta<1$ be a small parameter, to be chosen later.  Applying \eqref{equicty 3} with this $\eta$, we obtain
\begin{align*}
|\alpha(\kappa;u,h) - \alpha^{[2]}(\kappa;u,h)| \lesssim \|u\|_{L^2}^2 \Big\{ ( h^{-1}\kappa^{-1} + \sqrt{\eta} ) \|u\|_{L^2}^2 +  \|u_{>\eta\kappa}\|_{L^2}^2  \Big\} .
\end{align*}

As $\alpha$ is conserved, then for each $u(0)\in Q$ we have
\[
\alpha^{[2]}(\kappa;u(t)) \leq \alpha^{[2]}(\kappa;u(0)) + 2\sup_{s\in[0,t]}|\alpha(\kappa;u(s)) - \alpha^{[2]}(\kappa;u(s))| .
\]
This yields
\[
\sup_{u\in Q^*} \alpha^{[2]}(\kappa;u,h) \leq \sup_{u\in Q} \alpha^{[2]}(\kappa;u,h) + C\delta^2  \Big\{ ( h^{-1}\kappa^{-1} + \sqrt{\eta} ) \delta^2 + \sup_{u\in Q^*} \|u_{>\eta\kappa}\|_{L^2}^2  \Big\}
\]
uniformly for $\kappa\geq \kappa_0$, for some constant $C\geq 1$ independent of $h$, $\kappa$, and $\delta$.

First, we choose $0<\delta\leq 1$ so that
\[
C \delta^2 \leq \tfrac{1}{8} .
\]
Then, given an arbitrary $0<\epsilon \leq 1$, we set $\eta=\frac14 \epsilon^2$, and choose $\kappa_0$ larger if necessary so that
\[
\sup_{u\in Q} \alpha^{[2]}(\kappa;u,h) \leq \tfrac{1}{8}\epsilon
\quad\text{and}\quad
h^{-1}\kappa^{-1} \leq \tfrac12 \epsilon
\]
for all $\kappa\geq \kappa_0$.
Altogether, this yields
\[
\sup_{u\in Q^*} \alpha^{[2]}(\kappa;u,h) \leq \tfrac{1}{4} \epsilon + \tfrac{1}{8} \sup_{u\in Q^*} \|u_{>\eta\kappa}\|_{L^2}^2 .
\]

Define $e_{\kk}(\xi):= 2\pi \be \kk \Im ( F(\xi;\kk,h)-\frac{\be}{2\pi}\frac{1}{i\kk-\xi})$ which by \eqref{F 2} satisfies $|e_{\kk}(\xi)|\les h^{-1}\kk^{-1}$ for all $\xi\in \R$. Then, by \eqref{F 2}, we have
\begin{align*}
\alpha^{[2]}(\kappa;u,h)
&= \int |\ft{u}(\xi)|^2\bigg( \frac{\xi^2}{\kappa^2+\xi^2} + e_{\kk}(\xi) \bigg) \,d\xi + 2\pi \be \kk \Im\int \cj{u} (G \ast u) dx \\
& \geq  \int_{|\xi|>\kk} |\ft{u}(\xi)|^2 \frac{\xi^2}{\kappa^2+\xi^2}  \,d\xi +\int |\ft{u}(\xi)|^2 e_{\kk}(\xi)\,d\xi  + 2\pi \be \kk \Im\int \cj{u} (G \ast u) dx.
\end{align*}
It then follows by \eqref{G 2} that
\begin{align}
 \tfrac12\| u_{>\kappa} \|_{L^2}^2 \leq \alpha^{[2]}(\kappa;u,h) + ch^{-1}\kk^{-1} \|u\|_{L^2}^2,  \label{higha2}
\end{align}
for some universal constant $c>0$.

Using the conservation of mass and increasing $\kappa_0$ further if necessary, \eqref{higha2} implies
\begin{align*}
    \tfrac12 \| u_{>\kappa} \|_{L^2}^2 \leq \alpha^{[2]}(\kappa;u,h) + \tfrac{1}{4}\eps. 
\end{align*}
and so
\[
\sup_{u\in Q^*} \|u_{>\kappa}\|_{L^2}^2 \leq \epsilon + \tfrac14 \sup_{u\in Q^*} \|u_{>\eta\kappa}\|_{L^2}^2 
\]
for all $\kappa\geq\kappa_0$.  

Consider the function $E:(0,\infty)\to[0,\infty)$ defined by
\[
E(\kappa):=  \sup_{u\in Q^*} \|u_{>\kappa}\|_{L^2}^2 .
\]
This function is decreasing and bounded above by $\delta^2$.  To finish the proof, it suffices to show that $E(\kappa)\to 0$ as $\kappa\to \infty$.  So far, we have shown that for any $0 < \epsilon\leq 1$, there exist $\kappa_0\geq 1$ and $R = \eta^{-1} > 1$ so that
\[
E(R\kappa) \leq \epsilon + \tfrac14 E(\kappa) \quad\text{for all }\kappa\geq \kappa_0. 
\]
Iterating this inequality $n$ times, we obtain
\[
E(R^n\kappa_0)
\leq \epsilon \sum_{k=0}^{n-1} (\tfrac14)^k + (\tfrac14)^n E(\kappa_0) .
\]
Taking $n\to\infty$, we deduce
\[
\limsup_{\kappa\to\infty} E(\kappa) \leq 2\epsilon .
\]
As $\epsilon>0$ can be arbitrarily small, we conclude $E(\kappa)\to 0$ as $\kappa\to \infty$, as desired.
\end{proof}

The remainder of this section is devoted to the proof of Theorem~\ref{t:a-priori}.  In order to iterate our local well-posedness result, it suffices to show that for $0<s<\frac12$, the $H^s$-norm of solutions remains bounded globally in time.
Our main tool to do so is the following weighted combination of the quantities used in the previous proof:
\[
\alpha_s(\kappa;u,h) := \int_{\kappa}^\infty \alpha(\varkappa;u,h) \varkappa^{2s-1}\,d\varkappa .
\]
This particular weight is chosen so that the quadratic term is controlled by
\begin{align}
\alpha_s^{[2]}(\kappa;u,h) 
:=& \int_{\kappa}^\infty \alpha^{[2]}(\varkappa;u,h) \varkappa^{2s-1}\,d\varkappa
\nonumber \\
=& \int  \bigg( \int_{\kk}^{\infty} \varkappa^{2s-1}\frac{\xi^2}{\varkappa^2 +\xi^2} d\varkappa\bigg) |\ft{u}(\xi)|^2\,d\xi + O\big(h^{-1}\kappa^{2s-1}\|u\|_{L^2}^2 \big).
\label{a2s}
\end{align}
In view of the estimate
\begin{align}
  \int_{\kk}^{\infty} \varkappa^{2s-1}\frac{\xi^2}{\varkappa^2 +\xi^2} d\varkappa \sim_{s}  \frac{\xi^2}{(\kappa^2+\xi^2)^{1-s}}, \label{sintegral}
\end{align}
we see that for $\kappa$ large, \eqref{a2s} measures the $H^s$-norm of $u$ at frequencies $|\xi|\gtrsim\kappa$.
Indeed, by following similar arguments leading to \eqref{higha2}, \eqref{a2s} and \eqref{sintegral} imply
\begin{align}
     \|u_{>a\kk}\|_{H^{s}}^{2} \les_{s} |\alpha_s^{[2]}(\kappa;u,h)|+h^{-1}\kappa^{2s-1}\|u\|_{L^2}^2
     \label{a2shigh}
\end{align}
for any $a> \frac{1}{10}$, with implicit constant uniform in $\kk$ and $a$.

\begin{proof}[Proof of Theorem~\ref{t:a-priori}]
Fix $0<s<\frac12$ and $T>0$.  
As $Q$ is bounded in $H^s$, it is also bounded and equicontinuous in $L^2$. By definition of $M^*$, \eqref{M*} guarantees $Q^*_T$ is also bounded and equicontinuous in $L^2$.  Therefore, by Proposition~\ref{t:A series}, we may choose $\kappa_0\geq 1$ so that the series for $A(\kappa;u,h)$ converges for all $\kappa\geq\kappa_0$ and $u\in Q^*_T$.  Moreover, after increasing $\kappa_0$ if necessary, \eqref{equicty 3} ensures that 
\begin{equation}
\big\| \cj{u}(\kappa-\dx)^{-1} u\big\|_{\op} \leq \tfrac12 
\label{op 4}
\end{equation}
for all $\kappa\geq\kappa_0$ and $u\in Q^*_T$.

We turn now to the key computation, which is to control the error of $\al_{s}(\kk;u,h)$ from its quadratic terms. For $u\in Q^*_T$ and $\kappa\geq\kappa_0$, we have
\[
\big|\alpha_s(\kappa;u,h) - \alpha_s^{[2]}(\kappa;u,h)\big| \lesssim \sum_{n\geq 2} \int_\kappa^\infty \varkappa^{2s} \big|\tr\big\{ \big[ (\varkappa-\dx)^{-1} u\Pih\cj{u} \big]^n (\varkappa-\dx)^{-1} \big\} \big|\,d\varkappa.
\]
Arguing as in Proposition~\ref{t:A series}, we then decompose each $\Pih$ according to~\eqref{Pihdecomp}

When there is at least one copy of $J_h$, we use \eqref{hs J}.  The remaining copies of $K_h$ are then estimated using \eqref{K run}.  For example, the terms with exactly one copy of $J_h$ at the beginning or end (which yield the least gain in $\varkappa$) are bounded by
\begin{align*}
\big\| uJ_h\cj{u} \big\|_{\hs} \big\| (\varkappa-\dx)^{-1} \big\|_{\op} \big\| (\varkappa-\dx)^{-1} u \big\|_{\hs}^2 \|K_h\|_{\op}^{n-1} \|\cj{u} (\varkappa-\dx)^{-1} u\|_{\op}^{n-2} 
\lesssim \varkappa^{-2}\|u\|_{L^2}^4 \big( \tfrac12 \big)^{n-2} .
\end{align*}
This yields the contribution
\[
\|u\|_{L^2}^4 \sum_{n\geq 2}\big( \tfrac12 \big)^{n-2} \int_\kappa^\infty \varkappa^{2s-2} \,d\varkappa
\lesssim_s \kappa^{2s-1} \|u\|_{L^2}^4.
\]

It remains to estimate the terms with only $K_h$.  Let us begin with the case $n\geq 3$:  
\[
\sum_{n\geq 3} \int_\kappa^\infty \varkappa^{2s} \big|\tr\big\{ \big[ (\varkappa-\dx)^{-1} uK_h\cj{u} \big]^n (\varkappa-\dx)^{-1} \big\} \big|\,d\varkappa .
\]
We decompose the first and last copy of $u$ as $u = u_{\leq \eta\kappa} + u_{>\eta\kappa}$ for a parameter $0<\eta<1$ to be chosen later.

For the term with two copies of $u_{\leq\eta\kappa}$, we use \eqref{hs 3}, \eqref{op K}, and \eqref{op 4} to bound
\begin{align*}
\sum_{n\geq 3}&\int_\kappa^\infty \varkappa^{2s} \big|\tr\big\{ (\varkappa-\dx)^{-1} u_{\leq\eta\kappa} \big[ K_h\cj{u}(\varkappa-\dx)^{-1}u \big]^{n-1} K_h\cj{u}_{\leq\eta\kappa}(\varkappa-\dx)^{-1} \big\} \big|\,d\varkappa \\
& \lesssim \sum_{n\geq 3}(\tfrac12)^{n-3}\int_\kappa^\infty  \big\|(\varkappa-\dx)^{-1}\big\|_{\op}^2 \| u_{\leq\eta\kappa}\|_{L^\infty}^2 \big\| \cj{u} (\varkappa-\dx)^{-1} u \big\|_{\hs}^{2}\varkappa^{2s}\,d\varkappa \\
& \lesssim  \eta\kappa \|u\|_{L^2}^6 \int_\kappa^\infty\varkappa^{2s-2}\,d\varkappa \\
& \lesssim_s \eta\kappa^{2s} \|u\|_{L^2}^6 .
\end{align*}

The remaining terms have at least one copy of $u_{>\eta\kappa}$, which we will estimate using \eqref{hs}.
Applying \eqref{op 6} to $r=s$ and $N=\kappa$, we find
\begin{equation*}
\big\| \cj{u}(\kk-\dx)^{-1} u\big\|_{\op} 
\lesssim \kk^{-2s} \|u_{\leq N}\|_{H^s}^2 + N^{-2s} \|u_{>N}\|_{H^s}^2
\lesssim \kk^{-2s} \|u\|_{H^s}^2 .
\end{equation*}
On the other hand, by \eqref{a2shigh}, we see that
\begin{equation}
\|u\|_{H^s}^2 
\lesssim \kappa^{2s} \|u_{\leq \kappa}\|_{L^2}^2 + \|u_{>\kappa}\|_{H^s}^2
\lesssim \kappa^{2s} \|u\|_{L^2}^2 + |\alpha^{[2]}_s(\kappa;u,h)|
\label{a2s 2}
\end{equation}
for all $\kappa\geq\kappa_0$, after increasing $\kappa_0 \geq h^{-1}$ if necessary.
Combining these observations with \eqref{hs} and \eqref{op 4}, we obtain
\begin{align*}
\sum_{n\geq 3} &\int_\kappa^\infty \varkappa^{2s} \big\| (\varkappa-\dx)^{-1} u_{>\eta\kappa} \big\|_{\hs} \big\| (\varkappa-\dx)^{-1} u \big\|_{\hs} \big\| \cj{u} (\varkappa-\dx)^{-1} u \big\|_{\op}^{n-1}\,d\varkappa \\
& \lesssim \sum_{n\geq 3}(\tfrac12)^{n-3} \|u\|_{L^2} \|u_{>\eta\kappa}\|_{L^2} \|u\|_{H^s}^4 \int_\kappa^\infty \varkappa^{-1-2s}\,d\varkappa \\
& \lesssim_{s} \kappa^{-2s} \|u\|_{L^2} \|u_{>\eta\kappa}\|_{L^2} \|u\|_{H^s}^4 \\
& \lesssim \kappa^{2s} \|u\|_{L^2}^5 \|u_{>\eta\kappa}\|_{L^2} + \kappa^{-2s} \|u\|_{L^2} \|u_{>\eta\kappa}\|_{L^2} \big(\alpha^{[2]}_s(\kappa;u,h)\big)^2 .
\end{align*}

Lastly, it remains to estimate the $n=2$ term with two copies of $K_h$:
\[
\int_\kappa^\infty \varkappa^{2s} \big|\tr\big\{ (\varkappa-\dx)^{-1} uK_h\cj{u} (\varkappa-\dx)^{-1} uK_h\cj{u} (\varkappa-\dx)^{-1} \big\} \big|\,d\varkappa.
\]
We decompose each $u = \sum u_{N_j}$ into its dyadic Littlewood--Paley pieces for $j=1,2,3,4$.  After swapping indices if necessary, we may always assume that $N_1\sim N_2 \geq N_3\geq N_4$.  Unfortunately, however, we must account for every possible permutation of the functions $u_{N_j}$ due to the lack of symmetry.

For a small parameter $0<\eta<1$ to be chosen later, we split the sum into five regions:
\begin{enumerate}
\item $N_2\leq \kappa$,
\item $\kappa < N_2\leq\varkappa$ and $N_3 \leq \eta\kappa$,
\item $\kappa < N_2\leq\varkappa$ and $N_3 > \eta\kappa$,
\item $N_2>\varkappa$ and $N_3\leq \eta\kappa$,
\item $N_2>\varkappa$ and $N_3 > \eta\kappa$.
\end{enumerate}

\noi
\underline{$\bullet$ \textbf{Case 1:} $N_2\leq \kappa$}

\smallskip
\noi
For most permutations, we can use \eqref{hs} to estimate the two highest frequency contributions in Hilbert--Schmidt norm.  Specifically, as long as neither of the operators
\[
\cj{u}_{N_1} (\varkappa-\dx)^{-1} u_{N_2}
\quad\text{nor}\quad
\cj{u}_{N_2} (\varkappa-\dx)^{-1} u_{N_1}
\]
appear, then \eqref{hs}, \eqref{op K}, and Bernstein's inequality yield
\begin{align}
&\sum_{\sigma\in S_1} \big|\tr\big\{ (\varkappa-\dx)^{-1} u_{N_{\sigma(1)}} K_h\cj{u}_{N_{\sigma(2)}} (\varkappa-\dx)^{-1} u_{N_{\sigma(3)}} K_h\cj{u}_{N_{\sigma(4)}} (\varkappa-\dx)^{-1} \big\} \big|
\nonumber \\
&\qquad\lesssim \big\| (\varkappa-\dx)^{-1} u_{N_1} \big\|_{\hs} \big\| (\varkappa-\dx)^{-1} u_{N_2} \big\|_{\hs}  \|u_{N_3}\|_{L^\infty} \|u_{N_4}\|_{L^\infty} \big\| (\varkappa-\dx)^{-1} \big\|_{\op}
\nonumber \\
&\qquad\lesssim \varkappa^{-2} N_3^{\frac12} N_4^{\frac12} \|u_{N_1}\|_{L^2} \|u_{N_2}\|_{L^2} \|u_{N_3}\|_{L^2} \|u_{N_4}\|_{L^2}.
\label{perm 1}
\end{align}
Here, $S_1$ is the subset of permutations of $\{ 1,2,3,4 \}$ given by
\[
S_1 := \big\{ \sigma : \{ 1,2,3,4 \}\to\{ 1,2,3,4 \} \text{ bijective such that } \sigma(\{2,3\}) \neq \{1,2\} \big\} .
\]

For the remaining permutations
\[
S_2 := \big\{ \sigma : \{ 1,2,3,4 \}\to\{ 1,2,3,4 \} \text{ bijective such that } \sigma(\{2,3\}) = \{1,2\} \big\} ,
\]
we use \eqref{op 3} with $r=\frac14$ and \eqref{hs 2} to bound
\begin{align}
&\sum_{\sigma\in S_2} \big|\tr\big\{ (\varkappa-\dx)^{-1} u_{N_{\sigma(1)}} K_h\cj{u}_{N_{\sigma(2)}} (\varkappa-\dx)^{-1} u_{N_{\sigma(3)}} K_h\cj{u}_{N_{\sigma(4)}} (\varkappa-\dx)^{-1} \big\} \big|
\nonumber \\
&\qquad\lesssim \big\| (\varkappa-\dx)^{-1} u_{N_{\sigma(1)}} \big\|_{\mathfrak{I}_4} \big\| \cj{u}_{N_{\sigma(2)}} (\varkappa-\dx)^{-1} u_{N_{\sigma(3)}} \big\|_{\hs} \big\| \cj{u}_{N_{\sigma(4)}} (\varkappa-\dx)^{-1} \big\|_{\mathfrak{I}_4} 
\nonumber \\
&\qquad\lesssim \varkappa^{-2} N_2^{\frac12} N_3^{\frac14} N_4^{\frac14} \|u_{N_1}\|_{L^2} \|u_{N_2}\|_{L^2} \|u_{N_3}\|_{L^2} \|u_{N_4}\|_{L^2} .
\label{perm 2}
\end{align}
Comparing this to \eqref{perm 1} and recalling that $N_4\leq N_3\leq N_2$, we see that the above bound holds for all permutations.

Combining \eqref{perm 1} and \eqref{perm 2} yields the contribution
\begin{align*}
\sum_{N_4\leq N_3\leq N_2 \sim N_1 \lesssim \kappa} &N_2^{\frac12} N_3^{\frac14} N_4^{\frac14} \|u_{N_1}\|_{L^2} \|u_{N_2}\|_{L^2} \|u_{N_3}\|_{L^2} \|u_{N_4}\|_{L^2}  \int_\kappa^\infty \varkappa^{2s-2}\,d\varkappa \\
& \lesssim_s \kappa^{2s-1} \|u\|_{L^2}^2 \sum_{N_2 \sim N_1 \lesssim \kappa} N_2 \|u_{N_1}\|_{L^2} \|u_{N_2}\|_{L^2}  \lesssim \kappa^{2s} \|u\|_{L^2}^4.
\end{align*}

\noi
\underline{$\bullet$ \textbf{Case 2:} $\kappa < N_2\leq \varkappa$ and $N_3 \leq \eta\kappa$}

\smallskip
\noi
Combining \eqref{perm 1} and \eqref{perm 2} again and employing \eqref{a2shigh}, we find
\begin{align*}
\sum_{\substack{ N_1\sim N_2>\kappa \\ N_4\leq N_3\leq \eta\kappa }} &\sum_{\sigma} \int_{N_2}^\infty \varkappa^{2s} \big|\tr\big\{ (\varkappa-\dx)^{-1} u_{N_{\sigma(1)}} K_h\cj{u}_{N_{\sigma(2)}} (\varkappa-\dx)^{-1} u_{N_{\sigma(3)}} K_h\cj{u}_{N_{\sigma(4)}} (\varkappa-\dx)^{-1} \big\} \big| \,d\varkappa \\
& \lesssim  \sum_{\substack{ N_1\sim N_2>\kappa \\ N_4\leq N_3\leq \eta\kappa }} N_2^{\frac12} N_3^{\frac14} N_4^{\frac14}  \|u_{N_1}\|_{L^2} \|u_{N_2}\|_{L^2} \|u_{N_3}\|_{L^2} \|u_{N_4}\|_{L^2} \int_{N_2}^\infty \varkappa^{2s-2} \,d\varkappa \\
&\lesssim_s \eta^{\frac12} \|u\|_{L^2}^2 \sum_{N_1\sim N_2>\kappa} N_2^{2s} \|u_{N_1}\|_{L^2} \|u_{N_2}\|_{L^2} \\
&\lesssim_s \eta^{\frac12} \|u\|_{L^2}^2 |\alpha^{[2]}_s(\kappa;u,h)| + \eta^{\frac 12}h^{-1}\kk^{2s-1}\|u\|_{L^2}^{4}.
\end{align*}
In the first line, the inner sum is over all permutations $\sigma$ of $\{1,2,3,4\}$.

\medskip
\noi
\underline{$\bullet$ \textbf{Case 3:} $\kappa < N_2\leq \varkappa$ and $N_3 > \eta\kappa$}

\smallskip
\noi
Applying \eqref{perm 1}, \eqref{perm 2}, and \eqref{a2shigh} yields
\begin{align*}
\sum_{\substack{ N_1\sim N_2>\kappa \\ N_4\leq N_3 \leq N_2 \\ N_3>\eta\kappa }} &\sum_{\sigma} \int_{N_2}^\infty \varkappa^{2s} \big|\tr\big\{ (\varkappa-\dx)^{-1} u_{N_{\sigma(1)}} K_h\cj{u}_{N_{\sigma(2)}} (\varkappa-\dx)^{-1} u_{N_{\sigma(3)}} K_h\cj{u}_{N_{\sigma(4)}} (\varkappa-\dx)^{-1} \big\} \big| \,d\varkappa \\
& \lesssim \sum_{\substack{ N_1\sim N_2>\kappa \\ N_4\leq N_3\leq N_2 \\ N_3>\eta\kappa }} N_2^{\frac12} N_3^{\frac14} N_4^{\frac14} \|u_{N_1}\|_{L^2} \|u_{N_2}\|_{L^2} \|u_{N_3}\|_{L^2} \|u_{N_4}\|_{L^2} \int_{N_2}^\infty \varkappa^{2s-2}\,d\varkappa \\
& \lesssim_s \|u\|_{L^2} \|u_{>\eta\kappa}\|_{L^2} \sum_{N_1\sim N_2>\kappa} N_2^{2s} \|u_{N_1}\|_{L^2} \|u_{N_2}\|_{L^2} \\
& \lesssim_s \|u\|_{L^2} \|u_{>\eta\kappa}\|_{L^2} |\alpha^{[2]}_s(\kappa;u,h)|+ h^{-1}\kk^{2s-1} \|u\|_{L^2}^{3} \|u_{>\eta\kappa}\|_{L^2} .
\end{align*}

\noi
\underline{$\bullet$ \textbf{Case 4:} $N_2>\varkappa$ and $N_3 \leq \eta\kappa$}

\smallskip
\noi
Unlike the previous cases, we no longer rely on \eqref{perm 1} and \eqref{perm 2}.
We estimate the term containing $u_{N_4}$ in operator norm using either \eqref{op 3} or \eqref{op 2} for some $0<r<\frac12$, and we estimate the remaining terms in Hilbert--Schmidt norm using \eqref{hs} and \eqref{hs 2}:
\begin{align}
\sum_{\sigma} &\big|\tr\big\{ (\varkappa-\dx)^{-1} u_{N_{\sigma(1)}} K_h\cj{u}_{N_{\sigma(2)}} (\varkappa-\dx)^{-1} u_{N_{\sigma(3)}} K_h\cj{u}_{N_{\sigma(4)}} (\varkappa-\dx)^{-1} \big\} \big|
\nonumber \\
&\lesssim_r \varkappa^{-1-r} N_4^{r} \|u_{N_1}\|_{L^2} \|u_{N_2}\|_{L^2} \|u_{N_3}\|_{L^2} \|u_{N_4}\|_{L^2} .
\label{perm 3}
\end{align}
We employ this with $r=\frac14$.  By \eqref{higha2} and \eqref{a2s}, we obtain
\begin{align*}
\int_{\kappa}^\infty &\varkappa^{2s-\frac54} \sum_{\substack{ N_1\sim N_2>\varkappa \\ N_4\leq N_3 \leq\eta\kappa }} N_4^{\frac14} \|u_{N_1}\|_{L^2} \|u_{N_2}\|_{L^2} \|u_{N_3}\|_{L^2} \|u_{N_4}\|_{L^2}  \,d\varkappa\\
&  \les    \eta^{\frac14} \|u\|_{L^2}^2 \int_{\kappa}^\infty \varkappa^{2s-1}  \|u_{>\varkappa}\|_{L^2}^2 \,d\varkappa   \\
&\lesssim_s \eta^{\frac14} \|u\|_{L^2}^2 \int_{\kappa}^\infty \varkappa^{2s-1} \alpha^{[2]}(\varkappa;u,h) \,d\varkappa + \eta^{\frac 14}h^{-1}\kk^{2s-1}\|u\|_{L^2}^4 \\
& \lesssim \eta^{\frac14}\|u\|_{L^2}^2 |\alpha^{[2]}_s(\kappa;u,h)| +\eta^{\frac 14}h^{-1}\kk^{2s-1}\|u\|_{L^2}^4 .
\end{align*}

\noi
\underline{$\bullet$ \textbf{Case 5:} $N_2>\varkappa$ and $N_3 > \eta\kappa$}  

\smallskip
\noi
Using \eqref{perm 3} with $r=s$ and \eqref{a2shigh} yields
\begin{align*}
\sum_{\substack{ N_1\sim N_2>\kappa \\ N_4\leq N_3 \leq N_2 \\ N_3>\eta\kappa }}&  N_4^{s} \|u_{N_1}\|_{L^2} \|u_{N_2}\|_{L^2} \|u_{N_3}\|_{L^2} \|u_{N_4}\|_{L^2} \int_\kappa^{N_2} \varkappa^{s-1}\,d\varkappa  \\
& \lesssim_s \|u\|_{L^2}\|u_{>\eta\kappa}\|_{L^2} \sum_{ N_1\sim N_2>\kappa } N_2^{2s} \|u_{N_1}\|_{L^2} \|u_{N_2}\|_{L^2}  \\
& \lesssim_s \|u\|_{L^2} \|u_{>\eta\kappa}\|_{L^2} |\alpha^{[2]}_s(\kappa;u,h)|+ h^{-1}\kk^{2s-1} \|u\|_{L^2}^{3} \|u_{>\eta\kappa}\|_{L^2} .
\end{align*}

Collecting all of our estimates, we conclude that
\begin{align*}
\big|\alpha_s(\kappa;u,h) - \alpha^{[2]}_s(\kappa;u,&h)\big|
\lesssim \kappa^{2s} R^4 \big( 1 + \eta^{\frac 14} R^2 + R\|u_{>\eta\kappa}\|_{L^2} \big) \\
&+ \big( \eta^{\frac14} R^2 + R \|u_{>\eta\kappa}\|_{L^2} \big) \alpha^{[2]}_s(\kappa;u,h) + \kappa^{-2s} R\|u_{>\eta\kappa}\|_{L^2} \big(\alpha^{[2]}_s(\kappa;u,h)\big)^2
\end{align*}
uniformly for $u\in Q^*_T$ and $\kappa\geq\kappa_0\geq \max(1,h^{-1})$, where $R := \sup_{u\in Q} \jb{\|u\|_{L^2}}$.

Recall that $Q^*_T$ is bounded and equicontinuous in $L^2$ (see the discussion surrounding \eqref{op 4} for details).  
Therefore, we may choose $\eta>0$ small and then $\kappa_1\geq\kappa_0$ large to render
\[
\eta^{\frac14} R^2 + R\sup_{u \in Q^{\ast}_{T}}\|u_{>\eta\kappa}\|_{L^2} \ll 1 ,
\]
so that for any $t\in [0,T]$,
\[
|\alpha^{[2]}_s(\kappa;u(t),h)|\lesssim \sup_{u\in Q} |\alpha^{[2]}_s(\kappa;u,h)| + \kappa^{2s} R^4 + \kappa^{-2s} R \Big( \sup_{u\in Q^*_T} \|u_{>\eta\kappa}\|_{L^2} \Big) \Big(  \alpha^{[2]}_s(\kappa;u(t),h)\Big)^2 
\]
uniformly for $\kappa\geq \kappa_1$.  As $Q$ is bounded in $H^s$, then by \eqref{a2s} and \eqref{sintegral} we may choose $\kappa_2\geq\kappa_1$ so that
\begin{align}
    \sup_{u\in Q} |\alpha^{[2]}_s(\kappa;u,h)| \lesssim_s \sup_{u\in Q} \|u\|_{H^s}^2 + h^{-1}\kk^{2s-1}R^{2} \leq \kappa^{2s} R^4 
    \label{boot 0}
\end{align}
uniformly for $\kappa\geq \kappa_2$, and so
\begin{equation}
|\alpha^{[2]}_s(\kappa;u(t),h)| \leq C \kappa^{2s} R^4 + C \kappa^{-2s} R \Big( \sup_{u\in Q^*_T} \|u_{>\eta\kappa}\|_{L^2} \Big) \Big( \alpha^{[2]}_s(\kappa;u(t),h)\Big)^2
\label{boot 1}
\end{equation}
for some constant $C\geq 1$.

Again, as $Q^*_T$ is equicontinuous in $L^2$, we may choose $\kappa = \kappa_3\geq \kappa_2$ larger if necessary so that 
\[
\sup_{u\in Q^*_T} \|u_{>\eta\kappa_3}\|_{L^2} \leq \tfrac{1}{8} C^{-2} R^{-5} .
\]
Therefore, from \eqref{boot 1} we have
\begin{equation}
|\alpha^{[2]}_s(\kappa_3;u(t),h)| \leq C \kappa_3^{2s} R^4 + \tfrac{1}{8}C^{-1} \kappa_3^{-2s} R^{-4} \Big( \alpha^{[2]}_s(\kappa_3;u(t),h)\Big)^2 .
\label{boot 2}
\end{equation}
Now we invoke a bootstrap argument. Let 
\begin{align*}
    T^{\ast} : = \sup\{ t\in [0,T] \,: \, |\alpha^{[2]}_s(\kappa_3;u(t),h)| \leq 2C\kappa_3^{2s}R^{4}\}.
\end{align*}
By \eqref{boot 0}, we see that $T^{\ast}>0$.  We will show that in fact $T^{\ast}=T$. Assume instead that $T^{\ast}<T$.  As the map $t\mapsto |\alpha^{[2]}_s(\kappa_3;u(t),h)|$ is continuous, then we must have $|\alpha^{[2]}_s(\kappa_3;u(T^{\ast}),h)|=2C\kappa_3^{2s}R^4$. Evaluating \eqref{boot 2} at $t=T^{\ast}$ then gives 
\begin{align*}
    2C\kappa_3^{2s}R^4 \leq C\kappa_3^{2s}R^4 +  \tfrac{1}{8}C^{-1} \kappa_3^{-2s} R^{-4} \Big(  2C\kappa_3^{2s}R^4  \Big)^2 = \tfrac{3}{2}C\kappa_3^{2s}R^4,
\end{align*}
which is absurd. Therefore $T^{\ast}=T$, and we conclude that
\begin{align*}
    \sup_{t\in [0,T]} |\alpha^{[2]}_s(\kappa_3;u(t),h)| \leq 2 C \kappa_3^{2s} R^4.
\end{align*}
As this is true for any $u(0)\in Q$, we then have 
\begin{align*}
    \sup_{u \in Q_{T}^{\ast}} |\alpha^{[2]}_s(\kappa_3;u,h)| \leq 2 C \kappa_3^{2s} R^4.
\end{align*}
Combining this with~\eqref{a2s 2}, we deduce that $Q^*_T$ is bounded in $H^s$ as desired.
\end{proof}

\begin{ackno}\rm 
J.F. was partially supported by the ARC project FT230100588.
T.L. was supported by an AMS-Simons Travel Grant. 
The authors are grateful for the hospitality and
support provided by the Institut Henri Poincaré during the thematic program
 ``Dispersive integrable equations: Pathfinders in Hamiltonian PDE" in June, 2026.
\end{ackno}


\begin{thebibliography}{99}

\bibitem{ABIK}
T. Akahori, R. Baddredine, S. Ibrahim, N. Kishimoto,
{\it Global well-posedness for the generalized intermediate NLS with a nonvanishing condition at infinity},
arXiv:2512.18998 [math.AP] (2025).

\bibitem{Rana1}
R. Badreddine, 
{\it On the global well-posedness of the Calogero-Sutherland derivative nonlinear Schrödinger equation}, 
Pure Appl. Anal. {6} (2024), no. 2, 379--414.

\bibitem{BdMS}
V. Barros, R.~P. de Moura, G. Santos,
{\it Local well-posedness for the nonlocal derivative nonlinear Schrödinger equation in Besov spaces}, 
Nonlinear Anal. 187 (2019), 320--338.

\bibitem{BO98}
J. Bourgain,
{\it Refinements of Strichartz’ inequality and applications to 2D-NLS
 with critical nonlinearity},
Int. Math. Res. Not. 1998, No. 5, 253--283 (1998).

\bibitem{BL}
J.~Bourgain, D.~Li, 
{\it On an endpoint Kato-Ponce inequality},
 Differential Integral Equations 27 (2014), no. 11-12, 1037--1072.


\bibitem{BP}
N.~Burq, F.~Planchon, 
{\it On well-posedness for the Benjamin-Ono equation}, 
Math. Ann. 340 (2008), no. 3, 497--542. 

\bibitem{CFL1}
A. Chapouto, J. Forlano, T. Laurens,
{\it On the well-posedness of the intermediate nonlinear Schrödinger equation on the line},
	arXiv:2511.00302 [math.AP] (2025).
	
\bibitem{CFL2}
A.~Chapouto, J.~Forlano, T.~Laurens, 
{\it  Well-posedness for the periodic intermediate nonlinear Schr\"{o}dinger equation},
	arXiv:2605.30657 [math.AP] (2026).

\bibitem{Chen}
X.~Chen,
{\it The defocusing Calogero–Moser derivative nonlinear Schr\"odinger equation with a non-vanishing condition at infinity},
SIAM Journal on Mathematical Analysis, Vol. 58, Iss. 1 (2026).


\bibitem{Chen2}
X.~Chen,
{\it Scattering of the defocusing Calogero–Moser derivative nonlinear Schr\"{o}dinger equation}
arXiv:2511.06432v6 [math.AP] (2026).

\bibitem{ChenLenz}
X. Chen, E. Lenzmann,
{\it Finite-time blow-up solutions for the Calogero--Sutherland derivative NLS},
arXiv:2605.28789 [math.AP] (2026).


\bibitem{CW}
M.~Christ, M.~Weinstein,
{\it Dispersion of small amplitude solutions of the generalized Korteweg-de Vries equation},
J. Funct. Anal.
100 (1991), 87--109.

\bibitem{demoura1}
R.~P. de Moura, 
{\it Well-posedness for the nonlocal nonlinear Schrödinger equation},
J. Math. Anal. Appl. 326 (2007), no. 2, 1254--1267.

\bibitem{PMP}
R.~P. de Moura, D. Pilod,
{\it Local well-posedness for the nonlocal nonlinear Schrödinger equation below the energy space},
Adv. Differential Equations 15 (2010), no. 9-10, 925--952.

\bibitem{ET}
M. B. Erdoǧan, N. Tzirakis,
{\it Dispersive partial differential equations. Wellposedness and applications}, 
London Mathematical Society Student Texts 86. Cambridge: Cambridge University Press (ISBN 978-1-316-60293-5/pbk; 978-1-107-14904-5/hbk; 978-1-316-56326-7/ebook). xvi, 186 p. (2016).

 \bibitem{GL}
 P. Gérard, E. Lenzmann, 
{\it The Calogero-Moser derivative nonlinear Schr\"{o}dinger equation},
Commun. Pure Appl. Math. 77, No. 10, 4008--4062 (2024).
 
 


\bibitem{GO}
L.~Grafakos, S.~Oh, 
{\it The Kato-Ponce inequality},
 Comm. Partial Differential Equations 39 (2014), no. 6, 1128--1157.




\bibitem{GLM}
Z. Guo, Y. Lin, L. Molinet,
{\it Well-posedness in energy space for the periodic modified Benjamin-Ono equation}, J. Differ. Equations 256, No. 8, 2778--2806 (2014).

\bibitem{Hadama1}
S.~Hadama,
{\it Well-posedness in the full scaling-subcritical range for a class of nonlocal NLS on the line}
	arXiv:2603.28055 [math.AP] (2026).

\bibitem{Hadama2}
S.~Hadama,
{\it Small-data $L^2$ theory for the intermediate NLS and the Calogero--Moser derivative NLS}
	arXiv:2608.01138 [math.AP] (2026).

\bibitem{HoganKowalski}
J.~Hogan, M.~Kowalski, 
{\it Turbulent threshold for continuum Calogero-Moser models},
Pure Appl. Anal. 6 (2024), no.~4, 941--954.

\bibitem{IK}
A.~Ionescu, C.~Kenig, 
{\it Global well-posedness of the Benjamin-Ono equation in low-regularity spaces}, 
J. Amer. Math. Soc. 20 (2007), no. 3, 753--798. 

\bibitem{KP}
T. Kato, G. Ponce, {\it Commutator estimates and the Euler and Navier-Stokes equations}, Comm. Pure Appl.
Math. 41 (1988) 891–907.



\bibitem{KLV}
R.~Killip, T.~Laurens, M.~Vi\c{s}an, 
{\it Sharp well-posedness for the Benjamin--Ono equation}, 
Invent. Math. 236 (2024), no. 3, 999--1054.

\bibitem{KLV2}
R.~Killip, T.~Laurens, M.~Vi\c{s}an, 
{\it Scaling-critical well-posedness for continuum Calogero-Moser models on the line}, 
Commun.~Am.~Math.~Soc.~5 (2025), 284--320.

\bibitem{KMV}
R.~Killip, K.~Marsden, M.~Vi\c{s}an, 
{\it The Hamiltonian formulation of continuum Calogero-Moser models},
	arXiv:2604.09479 [math.AP] (2026).

\bibitem{KNV}
R.~Killip, M.~Ntekoume, M.~Vi\c{s}an,
{\it On the well-posedness problem for the derivative nonlinear Schr\"odinger equation},
Anal.~PDE 16 (2023), no.~5, 1245--1270.

\bibitem{KVZ}
R.~Killip, M.~Vi\c{s}an, X.~Zhang,
{\it Low regularity conservation laws for integrable PDE},
Geom.~Funct.~Anal.~28 (2018), no.~4, 1062--1090.

\bibitem{KKK2}
K. Kim, T. Kim, S. Kwon, 
{\it Construction of smooth chiral finite-time blow-up solutions to Calogero–Moser derivative nonlinear Schrödinger equation}, 
to appear in Mem. Amer. Math. Soc.

\bibitem{KKK1}
T. Kim, S. Kwon, {\it Soliton resolution for Calogero–Moser derivative nonlinear Schrödinger equation},
to appear in J. Eur. Math. Soc.
% arXiv:2408.12843 [math.AP] (2024).

\bibitem{MP}
L.~Molinet, D.~Pilod, 
{\it The Cauchy problem for the Benjamin-Ono equation in $L^2$ revisited},
 Anal. PDE 5 (2012), no. 2, 365--395. 

\bibitem{MR}
L.~Molinet, F.~Ribaud, {\it 
Well-posedness in $H^1$ for generalized Benjamin-Ono equation on the circle}, 
 Discrete and Continuous Dynamical Systems-Series S, 23(4), 1295-1311.


\bibitem{OT}
T. Ozawa, Y. Tsutsumi,
{\it Space-time estimates for null gauge forms and nonlinear Schrödinger equations}, Differ. Integral Equ. 11, No. 2, 201--222 (1998).
 
 
 \bibitem{PdM}
 J.~A. Pava, R.~P. de Moura, 
{\it Ill-posedness and the nonexistence of standing-waves solutions for the nonlocal nonlinear Schrödinger equation},
Differential Integral Equations 20 (2007), no. 10, 1107--1130.
 
  \bibitem{Pel1}
 D.~E. Pelinovsky, {\it Intermediate nonlinear Schr\"{o}dinger equation for internal waves in
a fluid of finite depth}, Phys. Lett. A, 197 (1995), 401--406.


\bibitem{Pel3}
D.~E. Pelinovsky, R.~H.~J. Grimshaw, {\it A spectral transform for the intermediate
nonlinear Schr\"{o}dinger equation}, 
J. Math. Phys., 36 (1995), 4203--4219.


\bibitem{Pel2}
D. E. Pelinovsky, R.~H.~J. Grimshaw, {\it Nonlocal models for envelope waves in a
stratified fluid}, 
Stud. Appl. Math. 97 (1996), no. 4, 369--391.
 
\bibitem{Simon}
B.~Simon,
{\it Trace ideals and their applications}, 
Mathematical Surveys and Monographs, vol. 120, American Mathematical Society, Providence, RI, 2005.

\bibitem{TAO04}
T.~Tao, 
{\it Global well-posedness of the Benjamin-Ono equation in $H^1(\mathbf R)$}, 
J. Hyperbolic Differ. Equ. 1 (2004), no. 1, 27--49. 




\end{thebibliography}
\end{document}